\documentclass{article}

\usepackage{authblk}
\usepackage{lmodern}
\usepackage{amsmath}\usepackage{amscd}
\usepackage{amssymb}
\usepackage{theorem}
\usepackage{xypic}
\usepackage[all,v2]{xy}
\xyoption{2cell}
\UseAllTwocells
\usepackage{enumitem}
\usepackage{rotating}
\usepackage{mathrsfs}
\usepackage{stmaryrd}
\usepackage{mathdots}
\usepackage{hyperref}
\usepackage{caption}
\usepackage[normalem]{ulem}
\usepackage{xcolor}

\usepackage{subcaption}
\usepackage{tikz}

\newenvironment{items}
{\begin{enumerate}[topsep=3pt, itemsep=3pt, parsep=0pt, label=(\roman*)]}
{\end{enumerate}}

\renewcommand{\tilde}{\widetilde}
\newcommand{\dual}[1]{#1^{\vee}}

\newcommand{\rel}{{\mbox{\tiny rel}}}
\newcommand{\lin}{{\mbox{\tiny lin}}}
\newcommand{\con}{{\mbox{\tiny con}}}

\newcommand{\argument}{{{\,\cdot\,}}}

\newcommand{\XX}{{\mathfrak X}}
\renewcommand{\AA}{{\mathfrak A}}

\newcommand{\A}{{\mathscr A}}
\newcommand{\B}{{\mathscr B}}
\newcommand{\D}{{\mathscr D}}

\newcommand{\C}{{\mathscr C}}

\newcommand{\M}{{\mathscr M}}
\newcommand{\Y}{{\mathscr Y}}
\newcommand{\N}{{\mathscr N}}

\renewcommand{\P}{{\mathscr P}}

\newcommand{\LL}{{\mathfrak L}}

\newcommand{\zz}{{\mathbb Z}}

\newcommand{\rr}{{\mathbb R}}

\renewcommand{\O}{{\mathscr O}}

\newcommand{\st}{\mathrel{\mid}}

\newcommand{\ev}{\mathop{\rm ev}\nolimits}

\newcommand{\Mor}{\mathop{\rm Mor}\nolimits}

\newcommand{\twoheadlongrightarrow}{\relbar\joinrel\twoheadrightarrow}

\newcommand{\Der}{\mathop{\rm Der}\nolimits}

\newcommand{\Homu}{\mathop{\underline{\rm Hom}}\nolimits}
\newcommand{\Mapu}{\mathop{\underline{\rm Map}}\nolimits}

\newcommand{\Symu}{\mathop{\underline{\rm Sym}}\nolimits}
\newcommand{\Map}{\mathop{{\rm Map}}\nolimits}

\newcommand{\longiso}{\stackrel{\textstyle\sim}{\longrightarrow}}
\newcommand{\iso}{\stackrel{\sim}{\rightarrow}}

\newcommand{\doublearrowstack}[2]%
                      {{{{\scriptstyle#1}\atop{\textstyle\longrightarrow}}\atop{{\textstyle\longrightarrow}\atop{\scriptstyle#2}}}}
\newcommand{\rightleftarrowstack}[2]%
                      {{{{\scriptstyle#1}\atop{\textstyle\longrightarrow}}\atop{{\textstyle\longleftarrow}\atop{\scriptstyle#2}}}}
\newcommand{\leftrightarrowstack}[2]%
                      {{{{\scriptstyle#1}\atop{\textstyle\longleftarrow}}\atop{{\textstyle\longrightarrow}\atop{\scriptstyle#2}}}}

\newtheorem{thm}{Theorem}[section]
\newtheorem{cor}[thm]{Corollary}
\newtheorem{lem}[thm]{Lemma}
\newtheorem{prop}[thm]{Proposition}

\theorembodyfont{\upshape}
\newtheorem{defn}[thm]{Definition}

\newtheorem{rmk}[thm]{Remark}

\newtheorem{ex}[thm]{Example}

\newenvironment{pf}{\begin{trivlist}\item[]{\sc Proof.}}%
            {\nolinebreak $\Box$ \end{trivlist}}

\newenvironment{pfIFT}{\begin{trivlist}\item[]{\sc Proof of Theorem~\ref{ift}.}}%
            {\nolinebreak $\Box$ \end{trivlist}}

\newenvironment{proof}{\begin{trivlist}\item[]{\sc Proof.}}%
            {\nolinebreak $\Box$ \end{trivlist}}

\newcommand{\del}{\partial}
\newcommand{\longto}{\longrightarrow}
\DeclareMathOperator\id{id}

\newcommand{\im}{\mathop{\rm im}\nolimits}
\newcommand{\rk}{\mathop{\rm rk}\nolimits}
\newcommand{\pr}{\mathop{\rm pr}\nolimits}

\newcommand{\Sym}{\mathop{\rm Sym}\nolimits}

\newcommand{\Hom}{\mathop{\rm Hom}\nolimits}

\newcommand{\contract}{{\,\lrcorner\,}}

\newcommand{\noprint}[1]{}

\newcommand{\xxto}[1]{\xrightarrow{#1}}
\newcommand{\mapnm}{{\Map(\N, \M)}}
\newcommand{\Z}{{\mathscr Z}}
\newcommand{\pathast}{\Gamma(I,a^\ast(T_M\oplus L)\,dt\oplus a^\ast L)} 
\newcommand{\eetale}{a local diffeomorphism}

\newcommand{\fiber}{\text{fiber}}

\newcommand{\Ho}{\mathrm{Ho}}
\newcommand{\ddim}{\dim^{\text vir}}

\title{Derived Differentiable Manifolds}

\author{Kai Behrend, Hsuan-Yi Liao and Ping Xu}

\date{}

\newcommand{\Addresses}{{
  \bigskip
  \footnotesize

  Kai Behrend, \textsc{Department of Mathematics,  University of British Columbia}\par\nopagebreak
  \textit{E-mail address}: \texttt{behrend@math.ubc.ca}

  \medskip

  Hsuan-Yi Liao, \textsc{Department of Mathematics, 
National Tsing Hua University}\par\nopagebreak
  \textit{E-mail address}: \texttt{hyliao@math.nthu.edu.tw}

  \medskip

  Ping Xu, \textsc{Department of Mathematics, Pennsylvania State University}\par\nopagebreak
  \textit{E-mail address}: \texttt{ping@math.psu.edu}

}}

\begin{document}
\sloppy

\maketitle

\begin{abstract}
{We develop the  theory of derived differential geometry in terms of bundles
of  curved $L_\infty[1]$-algebras, i.e.  dg manifolds of positive amplitudes.
We prove the category of derived manifolds is a category of fibrant objects. 
 Therefore, we can make sense of ``homotopy fibered product'' and ``derived intersection" of submaifolds in a smooth manifold in the homotopy category of derived manifolds. We construct a factorization of the diagonal using path spaces.
  First we construct an infinite-dimensional factorization using actual
 path spaces motivated by the AKSZ construction, then we cut down to
 finite dimensions using the Fiorenza-Manetti method.
 The main ingredient is the homotopy transfer theorem for  curved $L_\infty[1]$-algebras.

We also prove the  inverse function theorem for derived manifolds, and
investigate the relationship  between weak equivalence and 
quasi-isomorphism for derived manifolds.
} 
\end{abstract}

%

\let\thefootnote\relax\footnotetext{Research partially supported by NSF grants
 DMS-1707545 and DMS-2001599,  and KIAS Individual Grant MG072801.}

\tableofcontents

\section*{Introduction}\addcontentsline{toc}{section}{Introduction}

\subsubsection*{Derived manifolds}

This work is a contribution to the theory of derived manifolds in the
context of $C^\infty$-geometry. Recently,  there has been a growing interest
in derived differential geometry. The main purpose is
 to use  ``homotopy fibered product" in a proper sense to replace
 the fibered product, which is not always defined in classical differential
geometry. There has appeared quite many works in this 
direction in the literature  mainly motivated by derived algebraic geometry
of Lurie and To\"{e}n-Vezzosi \cite{MR2717174, MR2394633}--- see also \cite{MR3033634}.
 For instance, see \cite{MR2641940,carchedi2012homological, MR3121621, MR3221297, borisov2011simplicial,  2018arXiv180407622P, macpherson2017universal, Nuiten}
for the $C^\infty$-setting and \cite{MR4036665, MR3959070} for
the analytic-setting.  Our approach is based on 
the geometry of bundles of positively graded curved $L_\infty[1]$-algebras, or equivalently,
  dg manifolds of positive amplitudes. In spirit, it is analogous 
to  the dg approach to derived algebraic geometry 
 \cite{2002math.....12225B,2002math.....12226B,MR1801413,MR1839580,MR2496057}  
and is closer to the approach of Carchedi-Roytenberg  \cite{carchedi2012homological,MR3121621}, by  avoiding the machinery of  $C^\infty$-ring. 

For us, a {\em derived manifold} is a triple $\M=(M,L,\lambda)$,
where $M$ is a $C^\infty$-manifold, $L=L^1\oplus\ldots\oplus L^n$ is a 
graded vector bundle over $M$, and $\lambda=(\lambda_k)_{0\leq k< n}$ is a
smooth family of multilinear operations $\lambda_k:L^{\otimes k}\to L$ of degree 1, making each fiber $L|_P$, where $P\in M$, a curved $L_\infty[1]$-algebra.
 (Equivalently, $\Sym_{\O_M}L^\vee$ is a sheaf of commutative differential graded algebras over the sheaf $\O_M$ of $C^\infty$-functions on $M$.)

We prove that derived manifolds form a category of fibrant objects (see \cite{MR341469} for details on categories of fibrant objects).  
Therefore, we can make sense of ``homotopy fibered product'' for derived manifolds. In particular, in the homotopy category of derived manifolds, we can talk about derived intersection of two submaifolds in a smooth manifold $M$. See Section~\ref{rmk:HptFibProd}. 
Recall that a {\em category of fibrant objects }is a category $\C$, together with two subcategories, the category of fibrations in $\C$, and the category of weak equivalences in $\C$, subject to the following axioms:
\begin{items}
\item weak equivalences satisfy two out of three, all isomorphisms are weak equivalences,
\item all isomorphisms are fibrations,
\item every base change (i.e. pull back) of a fibration exists, and is again a fibration,
\item every base change of a trivial fibration (i.e., a fibration which is a weak equivalence) is a trivial fibration,
\item $\C$ has a final object, and all morphisms whose target is the final object are fibrations,
\item every morphism can be factored as a weak equivalence followed by a fibration.
\end{items}
One is interested in the localization of $\C$ at the weak equivalences, i.e., the category obtained from $\C$ by formally inverting the weak equivalences.  This localization process
 gives rise to an $\infty$-category, whose associated 1-category is the {\em homotopy category }of $\C$.  The presence of the fibrations, satisfying the above axioms, simplifies the description of the localized category significantly.  In fact, Cisinsky \cite{MR2746284} has shown, 
that the homotopy type of the space of morphisms from $X$ to $Y$ can be represented by the category of spans from $X$ to $Y$, where a span $X\leftarrow X'\to Y$ has the property that the backwards map $X'\to X$ is a trivial fibration.

The zeroth operation $\lambda_0$ is a global  
section of the vector bundle $L^1$ over $M$.  We call the zero locus of $\lambda_0$ the {\em classical locus }of the derived manifold $\M=(M,L,\lambda)$. At a classical point $P\in M$, we can define the {\em tangent complex}
$$\xymatrix{
TM|_P\rto^{D_P\lambda_0}& L^1|_P\rto^{\lambda_1|_P}& L^2|_P\rto^{\lambda_1|_P}&\ldots}.$$
We define a morphism of derived manifolds to be a {\em weak equivalence }if
\begin{items}
\item it induces a bijection on classical loci,
\item  the linear part  induces a quasi-isomorphism on tangent complexes at all classical points.
\end{items}
We define a morphism of derived manifolds to be a {\em fibration}, if 
\begin{items}
\item the underlying morphism of manifolds is a differentiable submersion,
\item the linear part of the morphism of curved $L_\infty[1]$-algebras
 is a degreewise surjective morphism of graded vector bundles.
\end{items}
Our main theorem is the following:
\begin{trivlist}
\item {\bf Theorem A} {\it The category of derived manifolds is a category of fibrant objects.}
\end{trivlist}

Restricting, for example,  to derived manifolds where $M$ is a single point,
 which is classical \cite[Remark 7.4]{MR3832143} \cite{FioMan},
 we obtain the category of fibrant objects of
 finite-dimensional, positively graded, hence nilpotent $L_\infty[1]$-algebras,
 where weak equivalences are quasi-isomorphisms, and fibrations are 
degreewise surjections  (see also \cite{MR1843805, MR2628795}).

On the other extreme, restricting to derived manifolds with $L=0$, we obtain
 the category of smooth manifolds, with diffeomorphisms for weak equivalences and differential submersions as fibrations.

Our choice of definition of derived manifolds
 is chosen to be minimal, subject to including both of these extreme cases.

Most of the proof of Theorem~1 is straightforward.  The non-trivial part is the proof of the factorization property.  It is well-known, that it is sufficient to prove factorization for diagonals. The proof proceeds in two stages.

First, we factor the diagonal of a derived manifold $\M$ as 
$$\M\longrightarrow PT\M[-1]\longrightarrow\M\times\M\,.$$
Here, $PT\M[-1]$ is the path space $\Map(I,T\M[-1])$ 
of the shifted tangent bundle $T\M[-1]$, where $I$ is an open interval containing $[0,1]$.  
Another way to view $\Map(I,T\M[-1])$ is as $\Map(TI[1],\M)$, where $TI[1]$ is the shifted tangent bundle of the interval $I$. Note
that  $TI[1]$ is not a derived manifold in our sense (the corresponding graded vector bundle is not concentrated in positive degrees);
this is essentially the viewpoint of  AKSZ construction \cite{MR1432574}. See 
Appendix~\ref{sec:AKSZ}.

The path space $PT\M[-1]$ is infinite-dimensional, so in a second step we pass to a finite-dimensional model by using homotopical perturbation theory, in this case the curved $L_\infty[1]$-transfer theorem  \cite{ezra,MR1950958}.  

In fact, it is convenient to break up this second step into two subitems. The base manifold of $PT\M[-1]$ is $PM=\Map(I,M)$, the path space of $M$.  The first step is to restrict to $P_gM\subset PM$, the manifold of {\em short geodesic }paths with respect to an affine connection on $M$.
 This means that we restrict to paths which are sufficiently short (or, in other words, sufficiently slowly traversed) so that evaluation at $0$ and $1$ defines an open embedding $P_gM\to M\times M$.  
Note that as a smooth manifold
$P_gM$ can be identified with an open neighborhood of the zero section
of $TM$, the open embedding $P_gM\to M\times M$ follows from the
tubular  neighborhood theorem of smooth manifolds.

In a second step, we choose a connection in $L$, and apply the $L_\infty[1]$-transfer theorem fiberwise over $P_gM$ to cut down the graded vector bundle to finite dimensions. 
Our method in this step is inspired by the Fiorenza-Manetti
approach in \cite{MR3832143, FioMan}.

After this process, we are left with the following  finite-dimensional
\emph{derived  path space $\P\M$}. The base manifold is
 $P_gM\subset M\times M$. Over this we have the
 graded vector bundle
\begin{equation}\label{grvb}
P_\con(TM\oplus L)\,dt\oplus P_\lin L\,.
\end{equation}
The fiber over a (short geodesic) path $a:I\to M$ of this graded vector bundle is the direct sum of the space of covariant constant sections in $\Gamma\big(I,a^\ast (TM\oplus L)\big)$, shifted up in degree by multiplying with the formal symbol $dt$ of degree $1$,  and the covariant affine-linear sections in $\Gamma(I,a^\ast L)$. 
The operations $(\mu_k)_{0\leq k<n+1}$ on the graded vector bundle
 (\ref{grvb}) are determined by the the transfer process. They can all be made explicit as (finite) sums over rooted trees, and 
 therefore are all smooth operations. Note that as a graded vector bundle,
 $P_\con (TM\oplus L)\,dt\oplus P_\lin L$ is in fact isomorphic to 
\begin{equation}
\ev_0^* (TM\oplus  L)\,dt \oplus \ev_0^* L \oplus \ev_1^*L 
\end{equation}
where $\ev_0, \ \ev_1: P_gM \subset M\times M\to M$ denote the restriction of projection to
the first and second component, respectively.

To prove that 
\begin{items}
\item the object (\ref{grvb}) is, indeed, a vector bundle over the manifold $P_gM$, and
\item the operations $(\mu_k)$ on (\ref{grvb}) are differentiable, 
\end{items}
uses nothing but standard facts about bundles, connections, and parallel
 transport since the  formula for the $L_\infty[1]$-algebra operations
 in the transfer theorem is explicit and universal. 
  Yet, to give a conceptual approach in light of AKSZ construction
\cite{MR1432574}, we can view  path spaces as 
infinite-dimensional manifolds  using
 the theory of {\em diffeology}.  It allows us to put a diffeological structure on the infinite-dimensional intermediary $PT\M[-1]$, and use a generalization of the transfer theorem to the diffeological context to prove our theorem.  In order not to clutter the exposition with formalities involving diffeologies, we have assembled the needed facts in an appendix. 

\subsubsection*{The associated differential graded algebra}

By standard facts, we can associate to a derived manifold 
\begin{items}
\item a  differential graded manifold, with whose structure sheaf is the sheaf of commutative differential graded algebras $\A=\Sym_{\O_M}L^\vee$, with a 
derivation $Q:\A\to\A$ of degree 1, induced by $(\lambda_k)$, satisfying $Q^2= \frac{1}{2}[Q,Q]=0$,
\item the commutative differential graded algebra $\Gamma(M,\A)$ of global sections. 
\end{items}
The following theorem serves as a partial justification for our definition of weak equivalence of derived manifold.

\begin{trivlist}
\item {\bf Theorem B} {\it  Any weak equivalence  of derived manifolds induces a quasi-isomorphism of the associated differential graded algebras of global sections.}
\end{trivlist}

Indeed, we propose the following

\begin{trivlist}
\item {\bf  Conjecture} {\it
 A morphism of derived manifolds is a weak equivalence, if and only if the induced morphism of differential graded manifolds is a quasi-isomorphism. }
\end{trivlist}

\subsubsection*{Inverse function theorem}

Along the way we prove another theorem which shows the strength of our definition of weak equivalence.  It is a generalization of the inverse function theorem. 
\begin{trivlist}
\item {\bf Theorem C} {\it   Every trivial fibration of derived manifolds admits a section, at least after restricting to an open neighborhood of the classical locus of the target.}
\end{trivlist}
If we replace  derived manifolds by their germs (inside $M$) around their
 classical loci, we therefore obtain an equivalent  category of
 fibrant objects in which every trivial fibration admits a section.  Such categories of fibrant objects have a particularly simple description of their homotopy categories:  the morphisms are simply homotopy classes of morphisms
 \cite[Theorem 1]{MR341469}.  \color{black} Here two morphisms $f,g:\M\to \N$ are homotopic if there exists a morphism $h:\M\to\P\N$, which gives back $f$ and $g$ upon composing with the two evaluation maps $\P\N\to\N$. 

In a subsequent paper, we will study differentiable derived stacks
 in terms of Lie groupoids in the category of derived manifolds along a similar line as in \cite{MR2817778}. 
In particular, we will  study ``derived orbifolds".
It would be interesting to investigate the relation between
derived orbifolds and Kuranishi spaces of  Fukaya-Oh-Ohta-Ono
 \cite{fukaya2015kuranishi, MR3931096}
and  d-orbifolds of Joyce \cite{MR3221297, joyce2015new}.
We expect that derived orbifolds in our sense will provide a useful tool in
studying  symplectic reduction, and  it would be interesting  to explore
the relation  between ``symplectic derived orbifolds" 
with  Sjamaar-Lerman theory of stratified symplectic spaces and reduction \cite{MR1127479}. 
It is also worth to study the connetion of
derived manifolds with Costello's ``$L_\infty$-spaces" \cite{costello2011geometric, MR3358542, MR4079151}. 
After the posting of the first e-print version of the present paper on
 arXiv.org, we learned that  a similar version of inversion function theorem for derived
 manifolds (called  $L_\infty$ spaces by the authors)  
was also  obtained by Amorim-Tu using a different method \cite{amorim2020whitehead}.

\subsection*{Notations and conventions}

Differentiable means $C^\infty$. Manifold means differentiable manifold, which includes second countable and Hausdorff as part of the definition.  Hence manifolds admit partitions of unity, which implies that vector bundles admit connections, and fiberwise surjective homomorphisms of vector bundles admit sections.

For any graded vector bundle $E$ over a manifold $M$, we sometimes use the same symbol $E$ to denote the sheaf of its sections over $M$, by abuse of notation. 
In particular, for a graded vector bundle $L$ over $M$, by
$\Sym_{\O_M} L^\vee$, or simply $\Sym L^\vee$, we denote the sheaf, over $M$,
 of sections of the graded vector bundle $\Sym L^\vee$, i.e. the sheaf of
fiberwise polynomial functions on $L$.

The notation $| \cdot |$ denotes the degree of an element. When we use this notation, we assume the input is a homogeneous element.

\section*{Acknowledgments}\addcontentsline{toc}{section}{Acknowledgments}
We would like to thank several institutions for their hospitality,
 which allowed the completion of this project: Pennsylvania State University
 (Liao),  University of British Columbia (Liao and Xu),
National Center for Theoretical Science  (Liao) and KIAS (Xu).
We also wish to thank  Ruggero Bandiera, Damien Broka, 
David Carchedi, 
Alberto Cattaneo, 
 David Favero,
 Domenico Fiorenza, Ezra Getzler,
 Owen Gwilliam,
 Vladimir Hinich,
 Bumsig Kim, Jia-Ming Liou, Wille Liu,
 Marco Manetti, Raja Mehta, Pavel Mnev,  Byungdo Park,
Jonathan Pridham, 
 Dima Roytenberg,
Pavol Severa,  Mathieu Sti\'enon and Ted Voronov
 for fruitful discussions and useful comments.

\section{The category of derived manifolds}

\begin{defn}
A {\bf derived manifold} is a triple $\M=(M,L,\lambda)$ , where $M$ is
a manifold (called the base), $L$ is a (finite-dimensional) graded vector bundle
$$L=L^1\oplus\ldots\oplus L^{n}$$ over $M$, and 
$\lambda=(\lambda_k)_{k\geq0}$ is a sequence of 
multi-linear operations
of degree $+1$
\begin{equation}
\label{eq:SCE}
\lambda_k:\underbrace{ L\times_M\ldots \times_ML}_k\longrightarrow L
\,,\qquad k\geq0\,.
\end{equation}
The operations $\lambda_k$ are required to be differentiable maps over $M$, and to make each fiber $(L,\lambda)|_P$,
for $P\in M$, into a curved $L_\infty[1]$-algebra (see Appendix~\ref{linfty}).
\end{defn}
  We say that $(L,\lambda)$ is
a {\bf bundle of curved $L_\infty[1]$-algebras }over $M$. 
 By our assumption on the
degrees of $L$,  we have $\lambda_k=0$, for $k\geq n$. The integer $n$
is called the {\bf amplitude }of $\M$. 

A derived manifold is called {\bf quasi-smooth} if its amplitude is one.
 That is, a quasi-smooth derived manifold consists of
 a triple $(M,L,\lambda)$,
 where $L = L^1$ is a vector bundle of degree $1$ over $M$,
 and $\lambda = \lambda_0$ is a global section of $L$.
Such a  derived manifold can be considered as the ``derived intersection"
of $\lambda_0$ with the zero section of $L$.
 If $f$ is a smooth function on $M$, the quasi-smooth derived manifold
 $(M, T^\vee M[-1], df)$ can be thought  as the ``derived critical locus'' of $f$.

\color{black}

Recall that the curved $L_\infty[1]$-axioms for the $\lambda = (\lambda_k)_{k \geq 0}$ can be summarized in the single equation
$$\lambda\circ\lambda=0\,.$$
(See Appendix~\ref{linfty} for the notation.) 
The axioms are enumerated by the number of arguments they take.  The first few are 
\begin{items}
\item ($n=0$) $\lambda_1(\lambda_0)=0$,
\item ($n=1$) $\lambda_2(\lambda_0,x)+\lambda_1^2x=0$.
\item ($n=2$) $\lambda_3(\lambda_0,x,y)+\lambda_2(\lambda_1(x),y)+ (-1)^{|x||y|}\, \lambda_2(\lambda_1(y),x)+\lambda_1(\lambda_2(x,y))=0$. 
\end{items}
If all $\lambda_k$, for $k\geq3$, vanish, then $L[-1]$ is a bundle of curved differential graded Lie algebras over $M$. Guided by this correspondence, we sometimes call the section $\lambda_0$ of $L^1$ over $M$ the {\bf curvature }of $\M$, call  $\lambda^1:L\to L[1]$ the {\bf twisted differential}, and $\lambda_2$ the {\bf bracket}.  The $\lambda_k$, for $k\geq3$, are the {\bf higher brackets}.

\begin{defn}
The set of points $Z(\lambda_0)\subset M$ where the curvature vanishes is a closed subset of $M$.  It is called the {\bf classical or
  Maurer-Cartan locus }of the derived-manifold $\M$, notation
$\pi_0(\M)$. We consider $\pi_0(\M)$ as a set without any further
structure.
\end{defn}

\begin{defn}
A {\bf morphism }of derived manifolds $(f,\phi):(M,L,\lambda)\to
(M',L',\lambda')$ consists of a differential map  of manifolds $f:M\to M'$, and 
a sequence of operations  
$$\phi_k:\underbrace{L\times_M\ldots\times_M L}_k\longrightarrow  L'\,,\qquad k\geq1\,,$$
of degree zero. The operations $\phi_k$ are required to be differentiable maps covering $f:M\to M'$, such that for every $P\in M$, the induced
sequence of operations $\phi|_P$ defines a morphism of curved $L_\infty[1]$-algebras $(L,\lambda)|_P\to (L',\lambda')|_{f(P)}$.
 Our assumptions imply that $\phi_k=0$, for
$k>n'$, where $n'$ is the amplitude of $(M',L',\lambda')$. 
\end{defn}

In particular, $\phi_1:L\to f^\ast L'$ is a morphism of graded vector
bundles.

Recall that the $L_\infty[1]$-morphism axioms
 can be summarized in the single equation
(see Appendix~\ref{linfty} for the notation) 
$$\phi\circ\lambda=\lambda'\bullet\phi\,.$$
They are enumerated by the number of
arguments they take.  The first few are
\begin{items}
\item ($n=0$) $\phi_1(\lambda_0)=\lambda_0'$,
\item ($n=1$) $\phi_2(\lambda_0,x)+\phi_1(\lambda_1(x))=\lambda'_1(\phi_1(x))$.
\end{items}
The first property implies that a morphism of derived manifolds induces a
map on classical loci. 

It is possible to compose morphisms of derived manifolds, and so we have the
category of derived manifolds.

\begin{rmk}\label{finalremark}
The category of derived manifolds has a final object.  This is the
singleton manifold with the zero graded vector bundle on it, notation $\ast$.  The
classical locus of a derived manifold $\M$ is equal to the set of morphisms
from $\ast$ to $\M$.
\end{rmk}

\begin{rmk}
If $(M,L,\lambda)$ is a derived manifold, and $f:N\to M$ a differentiable
map, then $(f^\ast L,f^\ast\lambda)$, is a bundle of curved
$L_\infty [1]$-algebras over $N$, endowing $N$ with the structure of derived manifold $(N,f^\ast L,f^\ast\lambda)$, together with a morphism of
 derived manifolds $(N,f^\ast L,f^\ast\lambda)\to(M,L,\lambda)$. 
\end{rmk}

\begin{rmk}
If $(M,L,\lambda)$ is a derived manifold, then the set of global sections
$\Gamma(M,L)$ is a curved $L_\infty[1]$-algebra. For every point $P\in M$, evaluation at $P$ defines a linear morphism (see below) of $L_\infty[1]$-algebras
 $\Gamma(M,L)\to L|_P$. 
\end{rmk}

\begin{defn}\label{defnlin}
We call a morphism $(f,\phi):(M,L,\lambda)\to (M',L',\lambda')$ of derived manifolds {\bf linear}, if $\phi_k=0$, for all $k>1$. 
\end{defn}

The linear morphisms define a subcategory of the category of derived manifolds, containing all identities, although not all isomorphisms. 

\begin{prop}\label{oneiso}
Let $(f,\phi):(M,L,\lambda)\to (M',L',\lambda')$ be a morphism of derived manifolds.  If $f:M\to M'$ is a diffeomorphism of manifolds, and $\phi_1:L\to L'$ an isomorphism of vector bundles, then $(f,\phi)$ is an isomorphism of derived manifolds. 
\end{prop}
\begin{pf}
Without loss of generality, we assume $M=M'$, $L=L'$, and $f=\id_M$, $\phi_1=\id_L$. First, note that for every sequence of operations $\chi$ on $L$, the equation $\psi\bullet \phi=\chi$ has a unique solution for $\psi$.  In fact, the equation $\psi\bullet\phi=\chi$ defines $\psi_k$ recursively.  Then define $\psi$ such that $\psi\bullet\phi=\id$, where $\id_1=\id_L$, and $\id_k=0$, for $k>1$.  Then $(\phi\bullet\psi)\bullet\phi=\id\bullet\phi$ by associativity of $\bullet$. By the uniqueness result proved earlier, this implies $\phi\bullet\psi=\id$.  
Therefore, $\psi$ is an inverse to $\phi$.  One checks 
that $\lambda'\bullet\phi=\phi\circ\lambda$ implies $\lambda\bullet\psi=\psi\circ\lambda'$, and hence $\psi$ is an inverse for $\phi$ as a morphism of 
derived manifolds. 
\end{pf}

\subsection{Etale morphisms and weak equivalences}

Let $\M=(M,L,\lambda)$ be a derived manifold, with curvature $F=\lambda_0$. 
Let $P\in Z(F)$ be a classical point of $\M$. 
Since $F(P)=0\in L_1|_P$, there is a natural decomposition
$T L_1|_{F(P)} \cong TM|_P \oplus L_1|_P$. By  $D_PF$, we denote
the composition
\begin{equation}
\label{eq:Lyon}
D_PF:  T M |_P \xxto{TF|_P}T L_1|_{F(P)} \cong T M|_P \oplus L_1|_P \xxto{\pr}
 {L_1}|_P
\end{equation}
where $TF|_P$ denotes the tangent map of $F: M\to L_1$ at $P\in M$.
Thus $D_PF:TM|_P\to L^1|_P$ is a linear map, called  the {\em  derivative of the curvature $F$}.

\begin{defn}
\label{def:tangentcomplex}
Let $\M=(M,L,\lambda)$ be a derived manifold, with curvature $F=\lambda_0$ and twisted differential $d=\lambda_1$. Let $P\in
Z(F)$ be a classical point of $\M$. 
 We add the derivative of the curvature  $D_PF$ to $d|_P$ in the
graded vector space $TM|_P\oplus L|_P$ to define the {\bf tangent
  complex }of $\M$ at $P$:
\begin{equation}\label{eq:tangent}
T\M|_P\quad :=\quad\xymatrix{
TM|_P\rto^{D_PF}& L^1|_P\rto^{d|_P}& L^2|_P\rto^{d|_P} & \ldots}
\end{equation}
(Note that $TM|_P$ is in degree $0$.)
\end{defn}

  The fact that $dF=0$, and that $P$ is a Maurer-Cartan point,   implies that we have a complex.  (Note that we do not define $T\M$ itself here.) 
We refer the reader to Remark \ref{rmk:dg} for
geometric meaning of the  tangent complex.

Let $\M=(M,L,\lambda)$ be a derived manifold. 
The \textbf{virtual dimension} $\ddim(\M)$ of $\M$  is
 defined to be the Euler characteristic of the tangent complex
 $T\M|_P$ at a classical point $P$ of $\M$. Since $L$ is finite-dimensional,
then 
$$
\ddim(\M) = \dim(M) +\sum_{i=1}^n \rk(L^i), 
$$
which is a locally constant function on the classical loci of $\M$.

Given a morphism of derived manifolds $f:\M\to\M'$, and a classical
point $P$ of $\M$, we get an induced morphism of tangent complexes
$$Tf|_P:T\M|_P\longrightarrow T\M'|_{f(P)}\,.$$
Only the linear part of $f$ is used in $Tf|_P$.

\begin{defn}
Let  $f:\M\to\M'$ be a morphism of  derived manifolds  and $P$ a classical
point of $\M$. We call $f$ {\bf \'etale at $P$}, if the induced
morphism of tangent complexes at $P$ 
$$Tf|_P:T\M|_P\longrightarrow T\M'|_{f(P)}$$
is a quasi-isomorphism (of complexes of vector spaces). 

The morphism $f$ is {\bf \'etale}, if it is \'etale  at every
classical point of $\M$. 
\end{defn}

Every isomorphism of derived manifolds is \'etale.

\begin{rmk}
For classical smooth  manifolds, an \'etale morphism is
exactly \eetale. On  the other hand, if both  the base manifolds
of $\M$  and $\M'$ are  just a point $\{*\}$, a morphism $f:\M\to\M'$
is \'etale means that $f$ is a quasi-isomorphism of
the corresponding $L_\infty[1]$-algebras if the 
curvature of $\M$ vanishes. Otherwise, it is always \'etale. 
\end{rmk}

\begin{defn}
A morphism $f:\M\to\M'$ of derived manifolds is called a {\bf weak
  equivalence} if
\begin{items}
\item $f$ induces a bijection on classical loci,
\item $f$ is \'etale.
\end{items}
\end{defn}

\begin{rmk}
Weak equivalences satisfy two out of three, and so the category of derived manifolds is a category with weak equivalences.
\end{rmk}

\begin{rmk}
If $\M \to \M'$ is a weak equivalence, then $\M$ and $\M'$ have the same virtual dimension. 
\end{rmk}

\subsection{Algebra model}

For a manifold $M$, we denote the sheaf of $\rr$-valued $C^\infty$-functions by $\O_M$.

\begin{defn}
A {\bf graded manifold} $\M$ of amplitude $[-m, n]$
is a pair  $(M,\A)$, where $M$ is a manifold,
$$\A=\bigoplus_{i}\A^i$$
is a sheaf of $\zz$-graded commutative $\O_M$-algebras over (the underlying
topological space) $M$, such that there exists a  $\zz$-graded
vector space 
$$V=V^{-m}\oplus V^{-m+1} \oplus \cdots V^{n-1}\oplus V^{n}$$ 
and a covering of $M$ by open submanifolds $U\subset M$, and
for every $U$ in the covering family, we have 
$$\A|_U\cong C^\infty (U)\otimes \Sym V^\vee\,,$$ 
as sheaf of graded $\O_M$-algebras.
Here $m$ and $n$ may not be non-negative integers, and
we only assume that $-m\leq n$.
\end{defn}

We say that the graded manifold $\M$ is finite-dimensional
if $\dim M$ and $\dim V$ are both finite. 
We refer the reader to~\cite[Chapter~2]{MR2709144} and~\cite{MR2275685,MR2819233}
for a short introduction to graded manifolds and relevant references.
By $\Gamma (M, \A)$, we denote the global sections.

\begin{defn} 
A {\bf differential graded manifold} (or {\bf dg manifold}) is a triple $(M,\A,Q)$, where $(M,\A)$ is
a graded manifold, and $Q:\A\to \A$
is a  degree 1 derivation of $\A$ 
as sheaf of $\rr$-algebras, such that $[Q,Q]=0$.
\end{defn}

Since we  are only interested  in the $C^\infty$-context,  in
the above definition, the existence of such  $Q:\A\to \A$ is 
equivalent to the existence of  degree 1 derivation on the global sections
$$Q:  \Gamma({M, \A})\to \Gamma({M, \A})$$
such that  $[Q,Q]=0$. Thus $Q$ can be considered as
a vector field on $\M$, notation $Q\in \XX (\M)$,
called a  {\bf homological vector field} on the  graded manifold $\M$
 \cite{MR1432574,MR2819233}.

\begin{defn}
A {\bf morphism }of dg manifolds $(M,\A,Q)\to (N,\B,Q')$ is a pair
$(f,\Phi)$, where $f:M\to N$ is a differentiable map of manifolds, and
$\Phi:\B\to f_\ast \A$ (or equivalently $\Phi:f^\ast \B\to \A$) is a
morphism of sheaves of graded  algebras, such that $Q\,\Phi=\Phi\, f^\ast Q'$.
\end{defn}

This defines the category of dg manifolds.

Let $\M=(M,L,\lambda)$ be a derived manifold. 
The sheaf of graded algebras $\A=\Sym_{\O_M}L^\vee$ is locally free
 (on generators in negative bounded degree), and thus
$(M, \A)$ is a graded manifold of  amplitude  $[1, n]$, called
{\bf  graded manifolds of positive amplitudes} by abuse of notation.
  The sum of the duals of the $\lambda_k$ defines a homomorphism of
 $\O_M$-modules $L^\vee\to \Sym_{\O_M}L^\vee$, which extends in a unique way
 to a derivation $Q_\lambda:\A\to \A$ of degree $+1$. 
 The condition
 $\lambda\circ\lambda=0$ is equivalent to the condition $[Q_\lambda, Q_\lambda]=0$. 

Given a morphism of derived manifolds $(f,\phi):(M,L,\lambda)\to (N,E,\mu)$, the sum of the duals of the $\phi_k$  gives rise to a homomorphism of $\O_M$-modules $f^\ast E^\vee\to \Sym_{\O_M} L^\vee$, which extends uniquely to a morphism of sheaves of $\O_M$-algebras $\Phi:f^\ast\Sym_{\O_N}E^\vee\to \Sym_{\O_M}L^\vee$.  Conversely, every morphism of $\O_M$-algebras $f^\ast\Sym_{\O_N}E^\vee\to \Sym_{\O_M}L^\vee$ arises in this way from a unique family of operations $(\phi_k)$. The condition $\phi\circ\lambda=\mu\bullet\phi$ is equivalent to
 $Q_\lambda\Phi=\Phi f^\ast Q_\mu$. 

These considerations show that we have a functor
\begin{align}\label{eq3}
(\text{derived manifolds})&\longrightarrow (\text{dg manifolds of
positive amplitudes})\\
(M,L,\lambda)&\longmapsto (M,\Sym_{\O_M} L^\vee,Q_\lambda)\,.\nonumber
\end{align}
which is fully faithful.

\begin{prop}
\label{pro:Batchelor}
The functor (\ref{eq3}) is an equivalence of categories. 
\end{prop}
\begin{pf}
It remains to show essential surjectivity.
  Let $M$ be a manifold. We need to prove that if $\A$ is a
  sheaf of algebras, locally free on generators of negative degree  $\geq -n$,
 then there exist a graded vector bundle $E=E^{-n}\oplus\ldots\oplus E^{-1}$,
 and an embedding $E\hookrightarrow \A$,  such that $\Sym_{\O_M} E\iso\A$.
  Let $\B\subset\A$ be the sheaf of subalgebras generated by all local sections of degree $\geq -n+1$. Then the $\O_M$-module
 $E^{-n}=\A^{-n}/\B^{-n}$ is locally free, and the projection
 $\A^{-n}\to E^{-n}$ admits a section $E^{-n}\hookrightarrow\A^{-n}$. 
 By induction, we have $E^{-n+1}\oplus\ldots\oplus E^{-1}\hookrightarrow \A$, 
inducing an isomorphism $\Sym_{\O_M}(E^{-n+1}\oplus\ldots\oplus E^{-1})\iso \B$.  It follows that $\Sym_{\O_M}(E^{-n}\oplus\ldots\oplus E^{-1})\iso \A$.
Let  $L^{i}=(E^{-i})^\vee$, $i=1, \ldots, n$,
 and $L=L^1\oplus\ldots\oplus L^{n}$. Thus $\A\cong \Sym_{\O_M} L^\vee$.

Moreover,  as the morphism $\O_M\to\A$ has image in degree $0$, for degree
reason, we have $Q(\O_M)=0$, so that $(\O_M, 0)\to (\A,Q)$ is a morphism of 
sheaves of differential graded algebras. It thus follows that the
homological vector field $Q$ is indeed tangent to the fibers of
the graded vector bundle $L$.
Therefore, $Q$ induces a unique family of fiberwise
 operations $(\lambda_k)_{k \geq 0}$ of degree 1 as in  \eqref{eq:SCE}.
 The condition  $[Q, Q]=0$ implies that
 $\lambda\circ\lambda=0$. Hence $(L,\lambda)$ is
a  bundle of curved $L_\infty[1]$-algebras over $M$, and
$(M, L, \lambda)$ is a derived manifold.
\end{pf}

In the supermanifold (i.e.\ $\zz_2$-graded with $Q$ being zero) case,
Proposition~\ref{pro:Batchelor} is known as Batchelor's
theorem --- see~\cite{MR536951,MR2275685,QFT_2-vol_book}.

\begin{rmk}
\label{rmk:dg}
Thus, derived manifolds, according to our definition, are differential
 graded manifolds endowed with  a generating graded vector bundle of
 positive degrees. 
Some constructions are easier in the algebra model, but many use the graded vector bundle $L$. Of course, a major consideration is that $L$, being finite-rank, is more geometrical than $\A$.

For instance, 
 the tangent complex  \eqref{eq:Lyon} admits a natural
geometric interpretation. In classical differential geometry,
if $Q$ is a vector field on a  manifold $\M$ and $P$ is
a point in $\M$ at which $Q$ vanishes, the Lie derivative $\LL_Q$ induces
a well-defined linear  map
$$\LL_Q: T\M |_P \to T\M |_P$$
called the linearization of $Q$ at $P$ \cite[p.72]{MR515141}.

 For a derived manifold $\M=(M,L,\lambda)$, let $Q_\lambda$ be the
 homological vector field as in (\ref{eq3}). The  classical
points  can be thought as those points  where  $Q_\lambda$
vanishes. Note that, if  $P$ is a classical point,
$TL|_{0_P} \cong T M|_P \oplus L|_P$, where $0_P\in L|_P$
is the zero vector. By a formal calculation,
one can easily check that the linearization of $Q_\lambda$ at $P$
gives rise to  exactly the tangent complex \eqref{eq:tangent} at $P$.
\end{rmk}

\begin{rmk} 	
Dg manifolds of amplitude $[-m, -1]$ can be thought of as
 \emph{Lie $m$-algebroids} \cite{MR2521116,  MR2441255,MR4007376, MR2223155, MR3090103, MR2768006, MR1958835} 
and \cite[Letters~7 and~8]{2017arXiv170700265S}. 
They can be considered as the infinitesimal counterparts of higher groupoids.

Hence a general dg manifold of amplitude $[-m, n]$ 
can encode both stacky and derived singularities in
differential geometry.
\end{rmk}

\subsection{Fibrations}

\begin{defn}\label{defnfib}
We call the morphism $(f,\phi):(M,L,\lambda)\to (M',L',\lambda')$ of
derived manifolds a {\bf fibration}, if 
\begin{items}
\item $f:M\to M'$ is a submersion,
\item $\phi_1:L\to f^\ast L'$ is a degree-wise surjective morphism of
  graded vector bundles over $M$.
\end{items}
\end{defn}

\begin{rmk}
If we consider $\phi_1$ as a bundle map of graded vector bundles
$$
\xymatrix{
L \ar[d] \ar[r]^{\phi_1} & L' \ar[d] \\
M \ar[r]_{f} & M',
}
$$
 the combination of  conditions (i)-(ii)
in Definition \ref{defnfib}
is  simply equivalent to that $\phi_1:L \to L'$ is a submersion
as a differentiable map.
\end{rmk}

For every derived manifold $\M$, the unique morphism $\M\to\ast$ is a (linear)
fibration.

If a fibrations is a linear morphism (Definition~\ref{defnlin}), we call it a {\em linear fibration}.

\begin{prop}\label{prop:IsoToLinearMor}
Every fibration is equal to 
the composition of a linear fibration with an isomorphism. 
\end{prop}
\begin{pf}
Let $(f,\phi):(M,L,\lambda)\to (N,E,\mu)$ be a fibration of derived manifolds. By splitting the surjection $\phi_1:L\to f^\ast E$ we may assume that $L=f^\ast E\oplus F$, and that $\phi_1:L\to f^\ast E$ is the projection, which we shall denote by $\pi:L\to f^\ast F$.  We also denote the inclusion by $\iota:f^\ast E\to L$. 

We define operations $\phi_k':\Sym^k L\to L$ of degree zero by $\phi_1'=\id_L$, and $\phi_k'=\iota\,\phi_k$, for $k>1$. 

There exists a unique sequence of operations $\lambda':\Sym^k L\to L$ of degree 1, such that $\phi'\circ\lambda=\lambda'\bullet\phi'$. (Simply solve this equation recursively for $\lambda_k'$.) Thus
$ (M,L,\lambda')$ is a derived manifold, and
 $(\id_M,\phi'):(M,L,\lambda)\to (M,L,\lambda')$ is an isomorphism of derived manifolds (see Proposition~\ref{oneiso}).

Then we have $\pi\bullet\phi'=\phi$ and $\mu\bullet\pi=\pi\circ \lambda'$. The latter equation can be checked after applying $\bullet \phi'$ on the right, as $\phi'$ is an isomorphism. Note that we   have $\phi\circ\lambda=\mu\bullet\phi$, by assumption, and hence $(\mu\bullet \pi)\bullet\phi'=\mu\bullet(\pi\bullet\phi')=\mu\bullet\phi=\phi\circ\lambda= (\pi\bullet\phi')\circ\lambda= \pi(\phi'\circ\lambda)=\pi(\lambda'\bullet\phi')=(\pi\circ\lambda')\bullet\phi'$.

We have proved that $(f,\pi):(M,L,\lambda')\to (N,E,\mu)$ is a linear morphism of derived manifolds, and that $(f,\pi)\circ(\id_M,\phi')=(f,\phi)$. 
\end{pf}

\begin{rmk}
Suppose $\pi:L\to E$ is a linear fibration of curved $L_\infty[1]$-algebras
 over a manifold $M$.  Then $K=\ker\pi$ is a curved $L_\infty[1]$-ideal in $L$. 
This means that, for all $n$, if for some $i=1,\ldots,n$ we have $x_i\in K$,
 then $\lambda_n(x_1,\ldots,x_n)\in K$. 
(It does not imply that $\lambda_0\in K$.)
 Here $\lambda$ is the curved $L_\infty[1]$-structure on $L$. 
\end{rmk}

\begin{prop}\label{pullex} 
\begin{items}
\item Pullbacks of  fibrations exist in the category of
derived manifolds, and  are still fibrations.
\item Pullback of derived manifolds induces a pullback of classical loci.
\end{items}
\end{prop}
\begin{pf}
Let $(p,\pi):(M',L',\lambda')\to (M,L,\lambda)$ be a fibration,
 and $(f,\phi):(N,E,\mu)\to (M,L,\lambda)$ an arbitrary morphism of derived manifolds. Without loss of generality, we assume that $\pi$ is linear. We denote the corresponding sheaves of algebras by 
$\A=\Sym_{\O_\M} L^\vee$, $\A'=\Sym_{\O_\M'} {L'}^\vee$, and $\B=\Sym_{\O_\N} E^\vee$. 

We define the manifold $N'$ to be the fibered product $N'=N\times_M M'$, and endow it with the sheaf of differential graded algebras 
$$\B'=\B|_{N'}\otimes_{\A|_{N'}}\A'|_{N'}\,.$$
To prove $(N',\B')$ is a derived manifold, let $E'=E\times_L L'$ be the fibered product of the smooth maps $\phi_1:E \to L$ and $\pi_1:L' \to L$, and
consider the  projection map $E'\to N'$. It is
straightforward to see that the latter is naturally  a graded vector bundle and  $\B'=\Sym_{\O_{N'}}{E'}^\vee$.  And
it is clear  that $(N', E')\to (N, E)$
is a fibration.  
It follows from the fact that tensor product is a coproduct in the category of sheaves of differential graded algebras, that $(N',\B')$ is the fibered product of $(N,\B)$ and $(M',\A')$ over $(M,\A)$ in the category of differential graded manifolds. 

For the second claim, about classical loci, one can prove it by Remark~\ref{finalremark} and the universal property of pullbacks.
\end{pf}

\begin{rmk}\label{rmk:DerDim&FibProd}
Let $\M'\to \M$ be a fibration, and $\N\to \M$ an arbitrary morphism of derived manifolds. If the underlying graded vector bundles of $\M$, $\M'$ and $\N$ are $(M,L)$, $(M',L')$ and $(N,E)$, respectively, then the underlying graded vector bundle of $\N \times_{\M} \M'$ is $(N \times_M M', E\times_L L')$. Thus, the 
virtual dimensions satisfy the   relation
$$
\ddim(\N \times_{\M} \M') = \ddim(\N) + \ddim(\M') - \ddim(\M)\,.
$$
\end{rmk}

\begin{prop}
Pullbacks of \'etale fibrations are \'etale fibrations.
\end{prop}
\begin{pf}
Let $\M'\to \M$ be an \'etale fibration of derived manifolds,
 and $\N\to \M$ an arbitrary morphism of derived manifolds.
 Let $\N'=\N\times_{\M}\M'$ be the fibered product, and $S'$  a classical point of $\N'$. 
 Denote the images of $S'$ under the maps $N' \to N$, $N' \to M$ and $N' \to M'$ by $S$, $P$, and $P'$, respectively. One checks that
$$T\N'|_{S'}=T\N|_S\times_{T\M|_P}T\M'|_{P'}\,.$$
The result follows.
\end{pf}

A {\bf trivial fibration }is a weak equivalence which is also a fibration.

\begin{cor}
Pullbacks of trivial fibrations are trivial fibrations.
\end{cor}

Therefore, the only thing left to prove that derived manifolds form a category of fibrant objects is the factorization property.

\subsubsection{Openness}

Let $(f,\phi):(M,E, \lambda)\to (N,F, \mu)$ be a morphism of derived manifolds, and $P\in M$ a  point. We say that $(f,\phi)$ is a fibration at $P$, if 
\begin{items}
\item $Tf|_P:T_M|_P\to T_N|_{f(P)}$ is surjective,
\item $\phi_1|_P:E|_P\to F|_{f(P)}$ is degree-wise surjective. 
\end{items}
Note that being a fibration is an open property: if $(f,\phi)$ is a fibration at $P\in M$, then there exists an open neighborhood $P\in U\subset M$, such that the restriction of $(f,\phi)$ to $U$ is a fibration in the sense of Definition~\ref{defnfib}.

\begin{prop}\label{prop:OpennessEtale}
Let $(f,\phi):(M,E,\lambda)\to (N,F,\mu)$ be a fibration of derived manifolds. Let $P$ be a classical point of $(M,E,\lambda)$. Suppose that $(f,\phi)$ is  \'etale at $P$. Then there exists an open neighborhood $U$ of $P$ in $M$, such that $(f,\phi)$ is   \'etale at every classical point of $(U,E|_U,\lambda|_U)$. 
\end{prop}
\begin{pf}
For any derived manifold $\M= (M,E,\lambda)$, the tangent complexes at the points of $Z=\pi_0(\M)$ fit together into a topological vector bundle over the topological space $Z$. The morphism $(f,\phi)$ defines a degree-wise surjective morphism of complexes over $Z$. Hence the kernel is a complex of vector bundles on $Z$, which is acyclic at the point $P\in Z$.  For a bounded complex of topological vector bundles, the acyclicity locus is open.
\end{pf}

\subsection{Global section functor}

If $(M,\A,Q)$ is a differential graded manifold, $\big(\Gamma(M,\A),Q\big)$  is a differential graded commutative
algebra, called the  {\em algebra of global
  sections }of $(M,\A,Q)$. Every morphism of differential graded manifolds induces a
morphism of  differential graded algebras of global sections in the opposite direction.

\begin{defn}
A morphism of differential graded manifolds is a {\bf 
  quasi-isomorphism}, if it induces a quasi-isomorphism on algebras of
global sections.
\end{defn}

\begin{rmk}
We \emph{conjecture} that a morphism of derived manifolds is a weak equivalence, if and only if the induced morphism of differential graded manifolds is a quasi-isomorphism.  We prove the `only if' part in Theorem~\ref{thm:WeakEqImplyQuasiIso}. 
As evidences for the `if' part, we include 
Proposition~\ref{prop:QuasiIsoBijMC} and Proposition~\ref{prop:QuasiIsoWkEqFib} below. 
Note that our conjecture is indeed
true in the case when  the base manifold $M$ is a point $\{*\}$ due  to Hinich \cite[Proposition 3.3.2]{MR1843805}.
\end{rmk}

\begin{prop}\label{prop:QuasiIsoBijMC}
If a morphism of derived manifolds induces a quasi-isomorphism on global section algebras, it induces a bijection on classical loci.
\end{prop}
\begin{pf}
Let $\M$ and $\N$ be derived manifolds with base manifolds $M$ and $N$, respectively.  We denote by $A$ and $B$ the corresponding cdgas of global sections. Let $(f,\phi): \M \to \N$ be a morphism of derived manifolds. Note that the set of points of $M$ is identified with the set of
$\rr$-algebra morphisms $A^0\to \rr$, where $A^0 = C^\infty(M)$. See \cite[Problem 1-C]{MR0440554}.  It follows that the set of
classical points of $\M$ is equal to the set of algebra morphisms
$H^0(A)\to\rr$. Similarly, the set of classical points of $\N$ is the
set of algebra morphisms $H^0(B)\to \rr$.  Since the quasi-isomorphism
$B\to A$ induces an isomorphism $H^0(B)\to H^0(A)$, it follows that $f$
induces a bijection $\pi_0(\M)\to\pi_0(\N)$ on classical loci.
\end{pf}

\begin{prop}\label{prop:QuasiIsoWkEqFib}
If a fibration of quasi-smooth derived manifolds induces a quasi-isomorphism of dg manifolds, it is a weak equivalence.
\end{prop}
\begin{pf}
We are in the quasi-smooth case. So $L=L^1$, and $E=E^1$.  By assumption, the following three-term sequence is exact on the right:
$$\xymatrix{
\Gamma(M,\Lambda^2 L^\vee)\rto^{\contract \lambda} & \Gamma(M,L^\vee)\rto^{\contract\lambda} & \Gamma(M,\O)\\
&\Gamma(N,E^\vee)\rto^{\contract \mu}\uto & \Gamma(N,\O)\uto}$$
As a formal consequence, the following diagram is exact on the left:
\begin{equation}\label{left}
\vcenter{\xymatrix{
\Gamma(M,\O)^\ast\rto\dto & \Gamma(M,L^\vee)^\ast\rto\dto & \Gamma(M,\Lambda^2 L^\vee)^\ast\\
\Gamma(N,\O)^\ast\rto & \Gamma(N,E^\vee)^\ast}}
\end{equation}
Here the asterisques denote $\rr$-linear maps to $\rr$. 
Our claim is that for $P\in Z(\lambda)\subset M$, the square
\begin{equation}\label{squr}
\vcenter{\xymatrix{TM|_P\rto\dto &L|_P\dto\\
TN|_{f(P)}\rto & E|_{f(P)}}}\end{equation}
is exact. By our assumption that $(f,\phi)$ is a fibration, the vertical maps are surjective, so we need to prove that (\ref{squr}) induces a bijection on kernels. The point $P$ defines a $\Gamma(M,\O)$-module structure on $\rr$, and the point $f(P)$ defines a $\Gamma(N,\O)$-module structure on $\rr$.  The vector space 
$$TM|_P\subset \Gamma(M,\O)^\ast$$
consists of the linear maps $\Gamma(M,\O)\to\rr$ which are derivations.  The vector space 
$$L|_P\subset\Gamma(M,L^\vee)^\ast$$
consists of the $\Gamma(M,\O)$-linear maps $\Gamma(M,L^\vee)\to\rr$. 

Consider an element $x\in L|_P$, mapping to zero in $E|_{f(P)}$. The element $x$ maps to zero in $\Gamma(M,\Lambda^2L^\vee)^\ast$, because 
the map $L|_P\to \Lambda^2 L|_P$ induced by multiplication by $\lambda$ is zero, as $\lambda$ vanishes at $P$.  Comparing with (\ref{left}), we see that there exists a unique $\rr$-linear map $v:\Gamma(M,\O)\to\rr$, such that 
\begin{items}
\item $v\circ f^\sharp:\Gamma(N)\to\rr$ vanishes,
\item $v\circ(\contract\lambda):\Gamma(L^\vee)\to \rr$ is equal to $x$.
\end{items}
We need to show that $v$ is a derivation, i.e., that for $g,h\in\Gamma(M,\O)$, we have
$$v(gh)=v(g)h+gv(h)\,.$$
Since $\Gamma(M,\O)$ is generated as an $\rr$-vector space by the images of $\Gamma(N,\O)$ and $\Gamma(M,L^\vee)$, it suffices to consider the following two cases.

{\bf Case 1.} Assume both $g,h$ are pullbacks from $N$. So there exist $\tilde g,\tilde h:N\to\rr$, such that $g=\tilde g\circ f$ and $h=\tilde h\circ f$. We have
$$v(gh)=v((\tilde g\circ f)(\tilde h\circ f))=v((\tilde g\tilde h)\circ f)=0=v(\tilde g\circ f)h+gv(\tilde h\circ f)=v(g)h+gv(h)\,.$$

{\bf Case 2.} Assume   that $h=\alpha\contract\lambda$ comes from $\alpha\in \Gamma(M,L^\vee)$. We have
$$v(gh)= v(g(\alpha\contract\lambda))=v((g\alpha)\contract\lambda)
=x(g\alpha)=g(P)x(\alpha)\,,$$
and
$$v(g)h+gv(h)=v(g)(\alpha\contract\lambda)(P)+g(P) v(\alpha\contract\lambda)=0+g(P)x(\alpha)\,.$$
This finishes the proof.
\end{pf}

\subsection{The shifted tangent derived manifolds}\label{sec:Ping}

\newcommand{\qQ}{Q}
\newcommand{\tQ}{\hat{\qQ}[-1]}
\newcommand{\tQQ}{\tilde{\qQ}}
\newcommand{\calx}{\XX}
\newcommand{\Aaa}{\AA}
\newcommand{\plambda}{\hat{\lambda}}

Let $\M=(M,\A, \qQ)$ be a dg manifold. Let $\Omega^1_\A$ be the sheaf of
differentials of $\A$, which is a sheaf of graded $\A$-modules.
Then $\Sym_\A(\Omega^1_\A[1])$ is  new sheaf of  graded
$\O_M$-algebras on $M$, and
 $\big(M, \Sym_\A(\Omega^1_\A[1])\big)$ is a graded manifold, denoted  $T\M[-1]$. 
The Lie derivative $\LL_\qQ$ with respect to $\qQ$ defines the structure of a
sheaf of differential graded $\A$-modules on $\Omega^1_\A$. We pass to
$\Sym_\A(\Omega_\A[1])$ to define  new sheaf of differential graded
$\rr$-algebras on $M$. Thus  $T\M[-1]=(M, \Sym_\A(\Omega^1_\A[1]), \LL_\qQ)$
is  a dg manifold.

\begin{rmk}
For  a  graded manifold $\M=(M,\A)$, $T\M$ is the graded manifold
$(M, T_\A)$, where $T_\A=\Sym_\A \dual{\Der_{\A}}  $ is a sheaf of  graded
$\O_M$-algebras on $M$. Here $\Der_{\A}$ stands for the sheaf of
graded derivation of  $\A$, which is a sheaf of graded $\A$-modules. 
$T\M$ is called the tangent bundle of $\M$ and
$T\M\to \M$ is a vector bundle in the category of graded manifolds.
If $\M=(M,\A, \qQ)$ is a dg manifold, 
it is standard \cite{MR3319134,MR3754617} that its tangent bundle $T\M$ is naturally  equipped
with a homological vector field $\hat{\qQ}$, called the complete lift \cite{MR0350650}, which makes it into a dg manifold.
The degree $1$ derivation $\hat{\qQ}: T_\A\to T_\A$ is
essentially induced by the Lie   derivative $\LL_\qQ$. 
 According to Mehta \cite{Mehta},
$\hat{\qQ}$ is a linear vector field with respect to
the tangent vector bundle $T\M\to \M$, therefore the 
shifted functor makes sense \cite[Proposition 4.11]{Mehta}. 
As a consequence, $T\M[-1]$, together with the homological
vector field $\tQ$,   is a dg manifold. It is simple
to check that the resulting dg manifold 
coincides with  the dg manifold  $(M, \Sym_\A(\Omega^1_\A[1]), \LL_\qQ)$ 
described above. Note that in this case
both $T\M\to \M$ and $T\M[-1]\to \M$  are vector bundles in the category
of dg manifolds.
\end{rmk}

If $(M,\A,Q)$ comes from a derived manifold $(M,L,\lambda)$ via the comparison functor (\ref{eq3}), 
then so does $(M,\Sym_\A(\Omega_\A[1]),\LL_\qQ)$ as we shall  see  below.

The underlying graded manifold of $T\M[-1]$ is $TL[-1]$, which
admits a double vector bundle structure
\begin{equation}\label{squr1}
\vcenter{\xymatrix{TL[-1]\rto\dto &TM[-1]\dto\\
L\rto &M }}
\end{equation}
A priori, $TL[-1]$ is {\it  not} a vector bundle over $M$.
However, by choosing a  connection $\nabla$ on $L\to M$,
 one  can identify  $TL$ with $L\times_M TM \times_M L$.
Hence, one obtains a diffeomorphism  
\begin{equation}
\label{eq:phinabla}
\phi^\nabla: \ \ TL[-1]\xxto{\cong}   L\times_M TM \times_M L \, .
\end{equation}
The latter is a graded vector bundle over $M$. 
On the level of sheaves, \eqref{eq:phinabla} is equivalent to
a splitting of the following
 short exact sequence of sheaves of graded $\A$-modules over $\O_M$
\begin{equation}\label{notsplit}
\xymatrix{
0\rto & \Omega^1_M [1]\otimes_{\O_M}\A\rto & \Omega^1_\A [1]\rto &
\Omega^1_{\A/\O_M}[1]\rto & 0\rlap{\,.}}
\end{equation}
where  $\A=\Sym_{\O_M}L^\vee$.

Thus one can transfer, via $\phi^\nabla$,  the homological vector field
  $\LL_\qQ$ on $TL[-1]$ into a  homological vector field $\tQQ$ on
the graded vector bundle $TM[-1]\oplus L[-1]\oplus L$. In what follows, we show that
$\tQQ$ is indeed tangent to the fibers of the graded vector  
bundle $TM[-1]\oplus L[-1]\oplus L$. Therefore, we obtain
a derived manifold with base manifold $M$.
To describe the $L_\infty[1]$-operations on
 $TM[-1]\oplus L[-1]\oplus L$, induced by $\tQQ$,  
let us introduce a formal variable $dt$ of degree $1$, and write
 $TM[-1]=TM\,dt$ and $L[-1]=L\, dt$. In the sequel, we will use
both notations $[-1]$ and $dt$ interchangely.

\begin{prop}
\label{pro:TM1}
Let $\M=(M,L,\lambda)$ be a derived manifold.
Any connection $\nabla$ on $L$ induces a derived manifold
structure on $(M, \, TM \, dt\oplus L \, dt \oplus L,\,  \mu)$, where 
$$\mu = \lambda+\tilde\lambda+\nabla\lambda\,.$$
Here, for all $k\geq0$, $\lambda$, $\tilde\lambda$, and $\nabla\lambda$ are
given, respectively, by  
\begin{eqnarray}
&&\lambda_k(\xi_1 ,\ldots,\xi_k)=
\lambda_k(x_1,\ldots,x_k), \label{eq:lambda}\\
&&\tilde\lambda_k(\xi_1 ,\ldots,\xi_k)=
\sum_{i=1}^k(-1)^{|x_{i+1}|+\ldots+ |x_k|}\lambda_k(x_1,\ldots,y_i,\ldots,x_k)\,dt\, , \\
&&(\nabla\lambda)_{k+1}(\xi_0 ,\ldots,\xi_k)
=\sum_{i=0}^k(-1)^{|x_{i+1}|+\ldots+ |x_k|}(\nabla_{v_i}\lambda_k)(x_0,\ldots,\hat
x_i,\ldots,x_k)\,dt\, , \qquad
\end{eqnarray}
$\forall \, \xi_i = v_i \, dt + y_i \, dt + x_i \in TM \, dt\oplus L \, dt \oplus L$,   
$i=0, \ldots k$.

Moreover,  different choices of connections $\nabla$ on $L$ induce
isomorphic derived manifold structures on $TM \, dt\oplus L\, dt\oplus L$.
\end{prop}
\begin{pf}
Let  $\A=\Sym_{\O_M}L^\vee$ be the sheaf of functions on $\M$,
$R=\Gamma (M, \O_M)$  the algebra of $C^\infty$-functions on $M$,
  and
$\Aaa=\Gamma (M, \Sym L^\vee)$ the space of the global sections
of $\A$  being considered as  fiberwise  polynomial functions on $L$.

To describe the homological vector field $\tQQ$ on the graded vector
 bundle $TM \, dt \oplus L \, dt \oplus L$, 
we  will  compute its induced  degree 1 derivation on the
algebra of global sections of the  sheaf of functions.
First of all,  since  $Q$ is tangent to the fibers of 
 $\pi: L\to M$, it follows
that 
$$ \tQQ (f)= L_Q (\pi^* f)=0,  \ \ \forall f \in  R.$$
 As a consequence, the induced homological
vector field $\tQQ$ on  $TM\, dt\oplus L\, dt \oplus L$  is tangent
to the fibers. Hence we indeed obtain a derived manifold
on $TM\, dt\oplus L\, dt\oplus L$ with base manifold $M$.

To compute the $L_\infty[1]$-operations, we need to compute $\tQQ$
on the three types of generating functions $\Omega^1 (M)[1]$, $ \Gamma (M,  L^\vee[1])$ and
 $\Gamma (M, L^\vee)$, by using the identification  $\phi^\nabla$ in  \eqref{eq:phinabla}
and    applying $\LL_Q$.

Note that the identification  $\phi^\nabla$ induces the identity map
on $\Aaa=\Gamma (M, \Sym L^\vee)$. Therefore,
 for any  $\xi \in \Gamma (M, L^\vee)$,
$$\tQQ \xi=\LL_Q \xi=Q (\xi)=\sum_{k=0}\lambda_k^\vee (\xi )\in \bigoplus_k 
 \Gamma (M, \Sym^k L^\vee)$$
Hence, by taking its dual,   we obtain $\lambda$ as in Eq. \eqref{eq:lambda}.

The map  $\phi^\nabla$ in  \eqref{eq:phinabla} induces an isomorphism:
\begin{equation}
\label{eq:psi}
\psi^\nabla: \ \Gamma(M,\Omega_\A^1[1]) \xxto{\cong}   \big ( \Omega^1 (M) [1]\oplus \Gamma (M,  L^\vee[1])\big)\otimes_{R}\Aaa \,.
\end{equation}

From now on, we identify the  right hand side with the left hand side.
Denote by $d$ the composition of the de Rham differential 
(shifted by degree $1$) with the isomorphism $\psi^\nabla$ as in
\eqref{eq:psi}: 
$$d: \Aaa\xxto{d_{DR}} \Gamma(M,\Omega_\A^1[1]) \xxto{\psi^\nabla}   \big ( \Omega^1  (M) [1]\oplus \Gamma (M,  L^\vee[1])\big)\otimes_{R}\Aaa$$
By $d^\nabla$,  we denote the covariant differential (shifted by degree $1$)
induced by the connection $\nabla$:
$$d^\nabla: \ \   \Aaa \to \Omega^1 (M)[1] \otimes_{R}\Aaa \hookrightarrow
 \big ( \Omega^1 (M) [1]\oplus \Gamma (M,  L^\vee[1])\big)\otimes_{R}\Aaa$$
It is simple to see that the image
 $(d-d^\nabla)(\Aaa) \subseteq  \Gamma (M,  L^\vee[1])\otimes_{R}\Aaa$.
Thus we obtain a map 
$$d_\rel:  \Aaa\to \Gamma (M,  L^\vee[1])\otimes_{R}\Aaa$$
such that 
$$d_\rel=d-d^\nabla .$$
It is easy to check that  $d_\rel$ is in fact  $R$-linear, and therefore
it corresponds to a bundle map 
$$d_\rel: \ \Sym L^\vee \to L^\vee[1]\otimes \Sym L^\vee$$
over $M$, by abuse of notations.
 Furthermore, for any $\eta \in  \Gamma (M, L^\vee)$ considered
as a linear function, we have 
$$d_\rel \eta =\eta [1] \, .$$
That is, when being restricted to  $L^\vee$, the bundle
map $d_\rel$ becomes the degree shifting
map  $L^\vee \to L^\vee [1]$,
and for general  $\Sym L^\vee$, one needs to apply the Leibniz rule.

Since $[\LL_Q, \ d]=0$,  it thus  follows that, for any
$\eta \in \Gamma (M, L^\vee)$,
$$ \tQQ (\eta[1])=  \LL_Q\circ d_\rel \eta= - d_\rel\circ \LL_Q\eta - [\LL_Q, d^\nabla]\eta\,.$$ 
Now $ \LL_Q\eta\in \Aaa$ and $-d_\rel\circ \LL_Q\eta\in \Gamma (M, L^\vee[1])\otimes_{R}\Aaa$.
 By taking its dual, it gives rise to $\tilde\lambda$.

On the other hand, it is simple to see that 
$$[\LL_Q, d^\nabla](\eta ) =(d^\nabla Q)(\eta )\in \Gamma (M, T^\vee M \otimes \Sym L^\vee), $$
where the homological vector field $Q$ is considered as a section
in $\Gamma (M, 	\Sym L^\vee \otimes L)$, and
 $d^\nabla Q\in \Gamma (M, T^\vee M \otimes \Sym L^\vee \otimes L)$. Hence it follows
that $-[\LL_Q, d^\nabla](\eta)=-(d^\nabla Q)(\eta)\in \Gamma (M, T^\vee M \otimes \Sym L^\vee)$.
 By taking its dual, we obtain  $\nabla\lambda$.

Finally,  since $\pi_* Q=0$, it follows that
$$  \tQQ \theta=\LL_Q(\pi^* \theta )=0, \ \ \forall \theta\in \Omega^1 (M) [1]. $$
Hence the $\Omega^1 (M) [1]$-part does not contribute to  any  $L_\infty[1]$-operations. Therefore we conclude that $\mu = \lambda + \tilde \lambda + \nabla \lambda$.

To prove the last part, let $\nabla$ and $\bar\nabla$ be any two connections on $L$. 
Denote by $\mu^\nabla$ and $\mu^{\bar\nabla}$ their corresponding $L_\infty[1]$-operations on $TM \, dt \oplus L\, dt \oplus L$, respectively. 
There exists a bundle map $\alpha:TM \otimes L \to L$ such that 
$
\alpha(v,x) = \bar\nabla_v x - \nabla_v x, \, \forall \, v \in \Gamma(TM), \, x \in \Gamma(L).
$ 
Let $\phi^{\bar\nabla}$ and $\phi^\nabla$ be the maps defined as in \eqref{eq:phinabla}.  Although the map $\Phi^{\bar\nabla,\nabla} = \phi^{\bar\nabla} \circ (\phi^\nabla)^{-1}$ is not an isomorphism of vector bundles over $M$, it induces an isomorphism of derived manifolds:
\begin{equation}\label{eq:CanonicalIsoTM[-1]}
 \Phi^{\bar\nabla,\nabla} :\big( M, TM \, dt \oplus L\, dt \oplus L, \mu^\nabla \big) \to \big( M, TM \, dt \oplus L\, dt \oplus L, \mu^{\bar\nabla} \big)
\end{equation}
given explicitly by 
\begin{align*}
\Phi^{\bar\nabla,\nabla}_1 &= \id, \\
\Phi^{\bar\nabla,\nabla}_2(\xi_1, \xi_2)&= \alpha(v_2,x_1) \, dt + (-1)^{|x_2|}\alpha(v_1,x_2) \, dt,  \\
\Phi^{\bar\nabla,\nabla}_n & = 0, \quad \forall n \geq 3,
\end{align*}
for any $\xi_i = v_i\, dt + y_i \, dt + x_i \in TM \, dt \oplus L\, dt \oplus L$, $i =1,2$. 
This concludes the proof.
\end{pf}

\begin{defn}
Let $\M=(M,L,\lambda)$ be a derived manifold.
 Choose a  connection $\nabla$ on $L$.
We call the induced  derived manifold $(M,TM\, dt\oplus L\, dt\oplus L, \mu)$
as in Proposition \ref{pro:TM1} the {\bf shifted tangent bundle }of $\M$, denoted $T\M[-1]$.
\end{defn}

\subsubsection*{Twisted shifted tangent derived manifolds}

Recall that a  degree $0$  vector field $\tau\in  \XX(L)$  is said to be
{\it  linear} \cite{MR1617335} if  $\tau$ is projectible, i.e.
$\pi_* (\tau)=X\in \XX(M)$ is a well-defined
vector field on $M$, and, furthermore, the diagram
$$
\xymatrix{
L \ar[d] \ar[r]^{\tau} & TL \ar[d] \\
M \ar[r]_{X} & TM
}
$$
is a morphism of vector bundles. Geometrically, linear vector fields
correspond exactly  to those whose flows are automorphisms
 of the vector bundle $L$.

Given a linear vector field  $\tau\in  \XX(L)$, 
consider the map $ \tilde{\delta}  : L\to L[-1]$
defined by
$$\tilde{\delta} (l)=\tau|_l -\widehat{X(\pi (l))}|_l, \  \ \ \forall l\in L,$$
where $\widehat{X(\pi (l))}|_l\in TL|_l $ denotes the horizontal lift of $X(\pi (l))\in TM|_{\pi (l)}$ at $l$.

\begin{lem}
\label{lem:Rome}
The map $\tilde{\delta}  : L\to L[-1]$ is a bundle map.
\end{lem}
\begin{pf}
 Let $\alpha$ be a path in $M$ such that $\alpha(0) = \pi(l), \alpha'(0) 
= X(\pi(l))$, and let $\gamma(l,t)$ be the parallel transport of $l$
 along $\alpha(t)$. Then the horizontal lift of $X$ at $l$ is
 $\gamma'(l,0)$. Since the parallel transport
 $L|_{\alpha(0)} \to L|_{\alpha(t)}: l \mapsto \gamma(l,t)$
 is linear at each $t$, the horizontal lift of $X$ at $l+\tilde l$ is
 $\gamma'(l + \tilde l,0) = \gamma'(l,0) + 
\gamma'(\tilde l,0)$.
Hence $\tilde \delta(l+\tilde l) = \tau(l+\tilde l) - ( \gamma'(l,0) + \gamma'(\tilde l,0)) = \tilde\delta(l)+ \tilde\delta(\tilde l)$.
\end{pf}

 Consider the dg manifold 
$$(TL[-1], \ \  \iota_\tau)$$
Again, a priori, this dg manifold is not a derived manifold over $M$.
However, if $\tau$ is a linear vector field on $L$,
 a connection $\nabla$ on
$L$ will make it into a derived manifold.

The following proposition follows from a straightforward
verification, which is left to the reader.

\begin{prop}
\label{pro:TM2}
Let $L$ be a positively graded vector bundle, and $\tau \in \calx (L)$
a linear   vector field on $L$. Then any connection $\nabla$ on $L$ 
induces a derived manifold structure on $(M, TM\, dt\oplus L\, dt\oplus L, \nu)$,
where
$$\nu=X+{\delta}. $$
Here $X=\pi_* (\tau) \in \Gamma (M, TM \, dt\, )$ is the curvature, and 
${\delta}$ is the degree 1 bundle map on $TM\, dt\oplus L\, dt\oplus L$
naturally extending the bundle map $ \tilde{\delta}: L\to L\, dt$
as in Lemma \ref{lem:Rome}.

Moreover,  different choices of connections $\nabla$ on $L$ induce
isomorphic derived manifold  structures on $TM\, dt\oplus L\, dt\oplus L$.
\end{prop}

Below is the situation
which is a combination of Proposition \ref{pro:TM1} and Proposition \ref{pro:TM2}. 
\begin{prop}
\label{pro:TM3}
Let $\M=(M,L,\lambda)$ be a derived manifold with $Q$ 
being its corresponding homological vector field.
Let $\tau \in \calx (L)$ be a degree $0$ linear  vector field on
 $L$ such that
$\pi_* (\tau)=X\in \calx (M)$.
Assume that 
\begin{equation}
\label{eq:tauQ}
[\tau , \ Q]=0.
\end{equation}
Choose a  connection $\nabla$ on $L$. Then there is an induced
  derived manifold structure on $(M, TM\, dt\oplus L\, dt\oplus L, \nu+\mu)$,
 where $\mu$ and $\nu$ are described as in Proposition \ref{pro:TM1} and Proposition \ref{pro:TM2}, respectively.

Moreover,  different choices of connections $\nabla$ on $L$ induce
isomorphic derived manifold  structures on $TM\, dt\oplus L\, dt\oplus L$.
\end{prop}
\begin{pf}
Consider the graded manifold $TL[-1]$. Both $\iota_\tau$  and $\tQQ$
are homological vector fields on $TL[-1]$. Now
\begin{eqnarray*}
[\iota_\tau+\tQQ, \ \iota_\tau+\tQQ]&=&2[\iota_\tau, \tQQ]\\
&=&2[\iota_\tau, \LL_Q]\\
&=&\iota_{[\tau, Q]}\\
&=&0.
\end{eqnarray*}
Therefore $\big( TL[-1], \iota_\tau+\tQQ \big)$ is indeed
 a dg manifold.  Choosing a  connection $\nabla$ on $L$,  one can identify $TL[-1]$
with $TM\, dt\oplus L\, dt\oplus L$ via the map $\phi^\nabla$ as
in \eqref{eq:phinabla}.
The rest of the proof follows exactly from that of
Proposition \ref{pro:TM1} and Proposition \ref{pro:TM2}.
\end{pf} 

\begin{rmk}
\label{rmk:tauQ}
Note that $\tau$ is degree $0$ vector field on $L$, while
$Q$ is of degree $1$. The commutator $[\tau , \ Q]$ is a
vector field of degree $1$. By $\phi_t$, we denote the (local)
flow on $L$   generated by $\tau$. Then \eqref{eq:tauQ}
is equivalent to saying that $\phi_t$ is a family of (local)
isomorphisms  of the dg manifold $(\M, Q)$. 
Since  $\tau$ is a linear vector field,
$\phi_t$ is a family automorphisms of the vector bundle $L$.
Then \eqref{eq:tauQ} is equivalent to saying that 
the  (local) flow $\phi_t$ is a family of isomorphisms  of the derived manifold
$\M=(M,L,\lambda)$.
\end{rmk}

\section{The derived path space}\label{sec:DerPathSp}

The purpose of what follows is to construct the derived path space of a
derived manifold. This will lead to a proof  of the factorization
property for derived manifolds.

\subsection{The case of a manifold}\label{sec:coam}

\subsubsection{Short geodesic paths}

We choose an affine connection in $M$, which gives rise to the notion of geodesic path in $M$, and defines an exponential map for $M$.

Let $I=(a,b)$ be an open interval containing $[0,1]$.  We will use $I$ as domain for all our paths. 

\begin{prop}
There exists a family of geodesic paths, parameterized by a manifold $P_gM$ 
\begin{equation}
\label{eq:geodesic}
P_gM\times I\longrightarrow M\,, (\gamma, t)\to \gamma (t) 
\end{equation}
such that, in the diagram
\begin{equation}
\label{eq:PgM}
\begin{split}
\xymatrix{
& P_gM\drto^{\gamma(0)\times \gamma(1)}\dlto_{\gamma(0)\times \gamma'(0)}\\
TM&& M\times M\\
&M\rlap{\,,}\ulto^0\uuto^{\text{\rm const}}\urto_\Delta}
\end{split}
\end{equation}
the two upper diagonal maps are open embeddings. We call $P_gM$ a manifold of {\bf short geodesic paths }in $M$. 
\end{prop} 
\begin{pf}
Let 
$V\subset TM$ be an open neighborhood of the zero section, where the exponential map is
defined. Replacing $V$ by a smaller open neighborhood, if necessary, we may assume $\exp(tv)$ is defined for any $t \in (a,b)$, $v \in V$. Then $V$ parameterizes a family of geodesic paths with domain $I$:
\begin{align*}
V\times I&\longrightarrow M\\
(x,v,t)&\longmapsto \gamma_{x,v}(t)=\exp_x(tv)\,.
\end{align*}
By considering the derivatives  
of the maps
$$V\longrightarrow  TM\,,\qquad (x,v)\longmapsto \big(\gamma_{x,v}(0),\gamma_{x,v}'(0)\big)\,,$$
and
$$V\longrightarrow M\times M\,,\qquad (x,v)\longmapsto \big(\gamma_{x,v}(0),\gamma_{x,v}(1)\big)\,,$$
at the zero section, we can find an open neighborhood $U\subset V$ of the zero section, such that both induced maps $U\to TM$ and $U\to M\times M$ are diffeomorphisms onto open submanifolds.
Then $U=P_gM$ works.
\end{pf}

For any $t\in I$, the map $\ev_t: P_g M\to M$, $t\mapsto \gamma (t)$,
(see \eqref{eq:geodesic}) is called the {\bf evaluation map}.

Note that as a  smooth manifold, $P_gM$ is diffeomorphic to an 
 open neighborhood of  the zero section of the tangent bundle $TM$
via the left  upper diagonal map in diagram \eqref{eq:PgM}.
Under such an identification,  the evaluation map
$\ev_0: P_gM\to M$ becomes the projection $ TM \to M$,  while
$\ev_1: P_gM\to M$ becoming the map $TM \to M, \ \ v_x\to\exp_x  v_x$.
The combination of upper diagonal maps in diagram \eqref{eq:PgM}
gives rise to a diffeomorphism from an open neighborhood of  the zero section of the tangent bundle $TM$ to an open neighborhood of the diagonal
of $M\times M$, known as tubular  neighborhood theorem of smooth manifolds.

\subsubsection{The derived path space of a manifold}

Over $P_gM$, we consider the vector bundle $P_\con TM$, whose fiber over the short 
geodesic path $\gamma$ is the set of covariant constant sections of
$\gamma^\ast TM$ over $I$.  We can identify $P_\con TM$
with the pullback of $TM$ to $P_g M$ via any evaluation
map $P_gM\to M$, see Lemma~\ref{pulcon}.

There is a canonical section of $P_\con TM$ over $P_gM$, which maps a path to its derivative, see Lemma~\ref{dersec}.  We denote it
by $D:P_gM\to P_\con TM$. Then 
$$\P=(P_gM,P_\con TM\,dt,D)$$ is a derived manifold of amplitude~1, the
{\bf derived path space }of $M$. (Here $dt$ is a formal variable of degree $+1$.) The classical locus of $\P$ is the
set of constant paths, as a geodesic path is constant if and only if
its derivative vanishes.

Mapping a point $P\in M$ to the constant path at $P$ gives rise to a
morphism of derived manifolds $M\to \P$. In fact, this map is a weak equivalence: It is
a bijection on classical loci because the classical locus of $\P$
consists of the constant paths.  To see that it is \'etale, consider a
point $P\in M$, and the corresponding constant path $a:I\to
M$. Then we have a short exact sequence of vector spaces
\begin{equation}\label{ses1}
TM|_P\longrightarrow T(P_gM)|_a\longrightarrow TM|_P\,.
\end{equation}
The first map is the map induced by $M\to P_gM$ on tangent spaces, the
second is the derivative of $D$, occurring in the definition of the
tangent complex of $\P$. To see that this sequence is exact, identify
$P_gM$ with an open neighborhood of the zero section in $TM$. This
identifies $T(P_gM)|_a$ with $TM|_P\times TM|_P$. The first map in
(\ref{ses1}) is then the inclusion into the first component. The second map is the
projection onto the second component. Hence (\ref{ses1}) is a short exact sequence, and $M\to \P$ is \'etale.

Moreover, the product of the two evaluation maps defines a fibration
of derived manifolds $\P\to M\times M$.  We have achieved the desired
factorization of the diagonal of $M$ via $\P$, proving that $\P$ is,
indeed, a path space for $M$ in the sense of a category of fibrant objects \cite{MR341469}.

\subsection{The general case}

\subsubsection{The derived  path space of a derived manifold}\label{sec:DrivedPathSp}

Now we extend the construction above   to a general derived manifold.
The main theorem is summarized below. 
See Section~\ref{sec:PathSpCntn} for the notations of $P_\con$ and $P_\lin$.

\begin{thm}
\label{thm:Seoul}
Let $\M=(M,L,\lambda)$ be a derived manifold. Choose an affine connection
on $M$ and a connection on the vector bundle $L$. There is an induced
derived manifold 
\begin{equation}
\label{eq:Cov19}
\P\M=\big(P_g M, P_\con (TM\oplus L)\,dt\oplus P_\lin L, \, \delta + \nu \big)\,,
\end{equation}
such that
\begin{enumerate}
\item 
the bundle map
\begin{equation}
\label{eq:UBC1}
\begin{split}
\xymatrix{
L
 \ar[d] \ar[r]^-{\iota} & P_\con (TM\oplus L)\,dt\oplus P_\lin L \ar[d] \\
M \ar[r]_-{\iota} & P_g M
}
\end{split}
\end{equation}
is a {\em linear} weak equivalence of derived manifolds  $\M \to \P\M$, where
$\iota$ stands for the natural inclusion map of constant paths.

\item The bundle map 
\begin{equation}
\label{eq:UBC2}
\begin{split}
\xymatrix@C=4pc{
P_\con (TM\oplus L)\,dt\oplus P_\lin L \ar[d] \ar[r]^-{(\ev_0\oplus \ev_1)\circ \pr} & \ev_0^*L \oplus \ev_1^*L \ar[d] \\
P_g M \ar[r]_-{\ev_0\times \ev_1} & M\times M
}
\end{split}
\end{equation}
is a {\em linear} morphism of derived manifolds  $\P\M\to\M\times\M$.

\end{enumerate}
Moreover, for any two choices of affine connections on $M$ and connections on $L$, there exists an open neighborhood $U$ of the constant paths in $P_gM$ such that the restrictions to $U$ of their corresponding derived path spaces are isomorphic. 
\end{thm}

The derived manifold $\P\M$ in \eqref{eq:Cov19} is called the {\bf derived path space }of $\M$.

\begin{rmk}\label{rmk:DerPathSpAsVB}
As  graded vector bundles  over $P_g M$,
 $P_\con (TM\oplus L)\,dt\oplus P_\lin L$ is 
 isomorphic to $\ev_0^*(TM\oplus  L)\,dt\oplus \ev_0^*L \oplus \ev_1^*L$,
 where $\ev_0, \ev_1 : P_g M \to M$ are the
evaluation maps at $0$ and $1$, respectively. (When
$P_g M$ is identified with an open neighborhood of the diagonal
of $M\times M$,  $\ev_0$ and  $\ev_1$ correspond
to the projections to the first and the second component, respectively.)
Under this identification, the bundle map \eqref{eq:UBC1} becomes

\begin{equation}
\label{eq:UBC3}
\begin{split}
\xymatrix{
L
 \ar[d] \ar[r]^-{0\oplus \Delta} & \ev_0^*(TM\oplus  L)\,dt\oplus \ev_0^*L \oplus \ev_1^*L \ar[d] \\
M \ar[r]_-{\iota} & P_g M
}
\end{split}
\end{equation}
 where $\Delta: L \to \ev_0^*L \oplus \ev_1^*L$ is the diagonal map.
Similarly,  the bundle map \eqref{eq:UBC2} becomes

\begin{equation}
\label{eq:UBC4}
\begin{split}
\xymatrix{
\ev_0^*(TM\oplus  L)\,dt\oplus \ev_0^*L \oplus \ev_1^*L \ar[d] \ar[r]^-{\pr} & \ev_0^*L \oplus \ev_1^*L \ar[d] \\
P_g M \ar[r]_-{\ev_0\times \ev_1} & M\times M
}
\end{split}
\end{equation}
\end{rmk}

\begin{rmk}
Theorem \ref{thm:Seoul} (2) is equivalent to stating
that for $t= 0$ and  $1$,  the bundle map
$$
\xymatrix{
P_\con (TM\oplus L)\,dt\oplus P_\lin L\ar[d] \ar[r]^-{\ev_t \circ \pr} & L  \ar[d] \\
P_g M \ar[r]_-{\ev_t} & M
}
$$
is a {\em linear} morphism of derived manifolds  $\P\M\to\M$.
This, however, may be   false for arbitrary $t\in I$, which may
need a  morphism with  higher terms. 
\end{rmk}

Our approach is based  on path space construction of derived manifolds  
together with the $L_\infty[1]$-transfer theorem.
We will  divide the proof of  Theorem \ref{thm:Seoul}
into several steps, which are discussed in subsequent subsections.

\subsubsection{The path space of the shifted tangent bundle}

We first construct an infinite-dimensional curved $L_\infty[1]$-algebra over each path $a$ in $M$. The consideration of path spaces is motivated by AKSZ construction which is summarized in Appendix~\ref{sec:AKSZ}.

Here we prove Proposition~\ref{pro:paris} by  a  direct verification, 
without referring to the infinite dimensional derived manifolds of path spaces.
 The discussion in Appendix~\ref{sec:AKSZ}, however, gives a heuristic
reasoning as to where such a curved $L_\infty[1]$-structure formula comes from.

Recall from Proposition~\ref{pro:TM1} that, for each derived manifold $\M = (M,L,\lambda)$ and a connection $\nabla$ on $L$, we have an associated derived
 manifold, the shifted tangent bundle,
 $T\M[-1] = (M,T_M \, dt  \oplus L \, dt \oplus L, \mu)$.
The family of $L_\infty[1]$-operations is
$\mu=\lambda+\tilde\lambda+\nabla\lambda$.

For every path $a:I\to M$ in $M$, we get an induced curved $L_\infty[1]$-structure in the vector space
$\Gamma(I,a^\ast T\M[-1])=\Gamma(I,a^\ast(T_M\oplus L)\,dt\oplus a^\ast L)$, by pulling back $\mu$ via $a$.

The covariant derivative with respect to $a^\ast\nabla$ is  a linear map
$\delta:\Gamma(I,a^\ast L)\to \Gamma(I,a^\ast L)\,dt$.
That is, $\delta (l(t))=(-1)^{k} \, \nabla_{a' (t)}l(t)\, dt$, \, $\forall l(t)\in
\Gamma(I,a^\ast L^{k})$. 
As $\delta^2=0$, it induces the structure of a complex on
$\Gamma(I,a^\ast T\M[-1])$.

We also consider the derivative $a'\, dt \in \Gamma(I,a^\ast T_M)\,dt$.

\begin{prop}
\label{pro:paris}
The sum $a'\, dt \, +a^\ast\mu$ is a curved $L_\infty[1]$-structure
on the complex $\big(\Gamma(I,a^\ast T\M[-1]), \delta \big) = \big( \pathast ,    \delta\big)$.
\end{prop}
\begin{pf}
The condition we need to check is
$$(\delta+a'\, dt +a^\ast\mu)\circ(\delta+a'\, dt +a^\ast\mu)=0\,.$$
This reduces  to the condition
$$\delta\circ a^\ast\lambda+a^\ast\tilde\lambda\circ\delta+ a^\ast\nabla\lambda\circ
a'\, dt=0\,,$$
as all other terms vanish, either for degree reasons, or because
$\mu\circ\mu=0$. 

Let $x_1,\ldots,x_n$, $n\geq0$, be homogeneous elements  in $\Gamma(I,a^\ast
L)$. We need to prove that
  \begin{multline*}
\delta\big((a^\ast\lambda_n)(x_1,\ldots,x_n)\big) \\
+ \sum_i(-1)^{|x_i|(1+|x_1|+ \ldots + |x_{i-1}|)}(a^\ast\tilde \lambda_n)(\nabla_{a'} x_i\,dt,x_1,\ldots,\hat x_i,\ldots,
x_n) \\
+(a^\ast\nabla\lambda_{n})(a'\, dt ,x_1,\ldots,x_n)=0\,.
\end{multline*}
Using the definitions of $\tilde \lambda$, and $\nabla\lambda$, and
canceling a factor of $dt$, we are reduced to
\begin{multline*}
(-1)^{1+|x_1|+ \ldots + |x_n|}\nabla_{a'} \big(\lambda_n(x_1,\ldots,x_n)\big) \\
+ \sum_i(-1)^{|x_i|(|x_1|+ \ldots + |x_{i-1}|) + |x_1|+ \ldots + |x_n|}(a^\ast\tilde \lambda_n)(\nabla_{a'} x_i,x_1,\ldots,\hat x_i,\ldots,
x_n) \\
+(-1)^{|x_1|+ \ldots + |x_n|}(\nabla_{a'} \lambda_{n} )(x_1,\ldots,x_n)=0\,.
\end{multline*}
which is true by the definition of $\nabla_{a'} \lambda_{n}$. 
\end{pf}

The diagram below describes all operations with 0 or 1 inputs: 
$$\xymatrix@C=4pc{
a'\in \Gamma(I,a^\ast TM)\,dt\rto^-{a^\ast \nabla\lambda_0} &
\Gamma(I,a^\ast L^1)\,dt\rto^-{a^\ast\tilde\lambda_1}
&
\Gamma(I,a^\ast L^2)\,dt\rto^-{a^\ast\tilde\lambda_1}
&\ldots\\
a^\ast\lambda_0\in \Gamma(I,a^\ast L^1)\urto^-\delta\rto_-{a^\ast\lambda_1}  &
\Gamma(I,a^\ast L^2)\urto^-\delta\rto_-{a^\ast\lambda_1}& \Gamma(I,a^\ast L^3)\urto^-\delta\rto_-{a^\ast\lambda_1} &\ldots  }$$

\begin{rmk}
When the base manifold  $M$ is a point $\{*\}$,  a derived
manifold is simply a curved $L_\infty [1]$-algebra $({\mathfrak g}, \lambda)$.
One can check that the  curved $L_\infty [1]$-algebra described in
Proposition \ref{pro:paris}
 can be identified with the one on
 $\Omega (I)\otimes {\mathfrak g}$ induced from $({\mathfrak g}, \lambda)$
by tensoring the cdga $(\Omega (I), d_{DR})$, where   $d_{DR}$ stands
for the de Rham differential.
We are thus reduced to the case studied by Fiorenza-Manetti  \cite{FioMan}
(they consider $L_\infty$-algebras without curvatures and differential
forms of polynomial functions). 
\end{rmk}

\subsubsection{Restriction to short geodesic paths}

We now choose an affine connection in $M$, and restrict to the case where $a\in P_gM$ is a short geodesic path. 
Write $\Gamma_\con(I,a^\ast TM)$ for the subspace of covariant constant sections of $a^\ast TM$ over $I$. Then $D(a) = a'\, dt\, \in
\Gamma_\con(I,a^\ast TM)\,dt$, and we have a curved
$L_\infty[1]$-subalgebra 
$$\widetilde P T\M[-1]|_a := \Gamma_\con(I,a^\ast TM)\,dt\oplus\Gamma(I,a^\ast L)\,dt\oplus \Gamma(I,a^\ast L)$$ 
in $\Gamma(I,a^\ast TM)\,dt\oplus\Gamma(I,a^\ast L)\,dt\oplus \Gamma(I,a^\ast L)$.

\subsubsection{The contraction and transfer}

Following the Fiorenza-Manetti method  \cite{FioMan}, we define 
\begin{align*}
\eta: \Gamma(I,a^\ast L^k)\,dt&\longrightarrow \Gamma(I,a^\ast L^k)\\
\alpha(t)\,dt&\longmapsto(-1)^{k} \Big(\int_0^t\alpha(u)\,du-t\int_0^1\alpha(u)\,du\,\Big).
\end{align*}
The connection $\nabla$
trivializes the vector bundle $a^\ast L$ over the one-dimensional
manifold $I$,
i.e., we have a canonical identification $a^\ast L\cong I\times V$, where 
$V=\Gamma_\con(I,a^\ast L)$. This identifies $\Gamma(I,a^\ast L)$ with
$C^\infty(I,V)$. Then $\int \alpha(u)\,du$ is the Riemann integral of
the vector valued function $\alpha$ defined on the open interval $I$. 

We consider $\eta$ as a linear endomorphism of degree $-1$ in the graded vector space $\widetilde P T\M[-1]|_a$. 
 We check that
\begin{items}
\item $\eta^2=0$,
\item $\eta \delta \eta=\eta$. 
\end{items}
Therefore,  $\delta\eta$ and $\eta\delta$ are orthogonal idempotents in $\widetilde P T\M[-1]|_a$,
and hence decompose this space into a direct sum
$$\widetilde P T\M[-1]|_a =H\oplus \im \delta\eta\oplus \im \eta\delta\,,\qquad H=\ker
\delta\eta\cap \ker \eta \delta\,.$$ 
The differential $\delta$ preserves $H$.  The  inclusion $\iota:H\to
\widetilde P T\M[-1]|_a$, and the projection
$\pi=1-[\delta,\eta]:\widetilde P T\M[-1]|_a \to H$ set up a homotopy 
equivalence between $(H,\delta)$ and $(\widetilde P T\M[-1]|_a, \delta)$.

The transfer theorem for curved $L_\infty[1]$-algebras, Proposition~\ref{transfertheorem}, gives rise to a
curved $L_\infty[1]$-structure on $(H,\delta)$, and a morphism of curved
$L_\infty[1]$-algebras from $(H,\delta)$ to $(\widetilde P T\M[-1]|_a ,\delta)$.  To apply the transfer theorem, we use the filtration given by  $F_k=\Gamma(I,a^\ast L^{\geq k}\oplus L^{\geq k}\,dt))$, for $k\geq1$.

The kernel of $\delta\eta$ on $\Gamma(I,a^\ast L)\,dt$ is the space of
covariant constant sections  $V\,dt=\Gamma_\con(I,a^\ast L)\,dt$.

The kernel of $\eta\delta$ on $\Gamma(I,a^\ast L)$ is the
space of {\em linear }sections, denoted $\Gamma_\lin(I,a^\ast
L)$. These are the maps $s:I\to V$ which interpolate linearly between $s(0)$
and $s(1)$, i.e., which satisfy the equation
$$s(t)=(1-t)\,s(0)+t\,s(1)\,,\qquad\text{for all $t\in I$\,.}$$

The projection $\pi_\lin:\Gamma(I,a^\ast L)\to \Gamma_\lin(I,a^\ast L)$ maps the path
$\alpha$ to the linear path with the same start and end point:
\begin{equation}\label{eq:piLin}
\pi_\lin(\alpha)(t)=(1-t)\alpha(0)+t\,\alpha(1)\,.
\end{equation}
The projection $\pi_\con:\Gamma(I,a^\ast L)\,dt\to
\Gamma_\con(I,a^\ast L)\, dt$ maps the
path $\alpha\, dt$ to its integral:
\begin{equation}\label{eq:piCon}
\pi_\con(\alpha\, dt)=\Big( \int_0^1\alpha( u )\,du\,\Big) \,  dt \, .
\end{equation}

To sum up, for  each $a \in P_gM$, we have  a curved $L_\infty[1]$-algebra 
\begin{equation}\label{eq:FibDerPathSp}
\big(\Gamma_\con(I,a^\ast TM) \, dt  \oplus \Gamma_\con(I, a^\ast L)\,dt \oplus \Gamma_\lin(I,a^\ast L), \, \delta + \nu \big)
\end{equation}
together with $L_\infty[1]$-morphisms
\begin{eqnarray} 
&&\phi: \Gamma_\con(I,a^\ast TM) \, dt  \oplus \Gamma_\con(I, a^\ast L)\,dt \oplus \Gamma_\lin(I,a^\ast L) \longrightarrow   \widetilde P T\M[-1]|_a \qquad \label{eq:KIAS1}\\
&&\tilde{\pi}: \widetilde P T\M[-1]|_a \longrightarrow \Gamma_\con(I,a^\ast TM) \, dt  \oplus \Gamma_\con(I, a^\ast L)\,dt \oplus \Gamma_\lin(I,a^\ast L) \qquad \label{eq:KIAS2}
\end{eqnarray}
obtained from the homotopy transfer theorem. See Proposition~\ref{transfertheorem}.

\subsubsection{The $L_\infty[1]$-operations}

We apply Proposition~\ref{transfertheorem} to compute the zeroth and first
$L_\infty[1]$-operations in \eqref{eq:FibDerPathSp}. 

Recall that we transfer the $L_\infty[1]$-structure $\delta + a'\, dt \, + a^\ast\lambda +  a^\ast \tilde\lambda + a^\ast \nabla\lambda$ on $\widetilde P T\M[-1]|_a$ to an $L_\infty[1]$-structure $\delta + \nu$ on $\Gamma_\con(I,a^\ast TM) \, dt  \oplus \Gamma_\con(I, a^\ast L)\,dt \oplus \Gamma_\lin(I,a^\ast L)$. The transfer is via the homotopy operator $\eta$, and the associated projection is $\pi = \pi_\con + \pi_\lin$ as defined in \eqref{eq:piLin} and \eqref{eq:piCon}. 
All the $L_\infty[1]$-operations and morphisms are considered as maps 
$\Sym E \to F$, where $E$ and $F$ denote the source and target spaces of the $L_\infty[1]$-operations/morphisms, respectively. 
The compositions are denoted by $\bullet$ or $\circ$ 
as defined in Appendix~\ref{linfty}.

The transferred curvature is 
\begin{equation}\label{eq:TransferredCurvature}
\pi(\mu_0)=a'\, dt \, +\pi_\lin (a^\ast \lambda_0)\,.
\end{equation}
Here the second term is the linearization of the curvature: the linear path
from the starting point of the curvature to the end point of the
curvature along the geodesic path $a$.

Let $\phi$ be the $L_\infty[1]$-morphism \eqref{eq:KIAS1}
 induced from the homotopy transfer, 
and $\iota: \Gamma_\con(I,a^\ast TM \oplus a^\ast L)\,dt \oplus \Gamma_\lin(I,a^\ast L) \to   \widetilde P T\M[-1]|_a$  the inclusion map. Since the homotopy operator
 $\eta$ vanishes on $\Gamma(I,a^\ast L)$ and $\Gamma(I,a^\ast TM )dt$, it follows
that $\eta(a'\, dt \, + a^\ast\lambda) \bullet \phi =0$. Hence the homotopy transfer equation \eqref{recu} reduces to 
\begin{equation}\label{eq:TransInc}
\phi   
= \iota - \eta \, (a^\ast \tilde\lambda) \bullet \phi - \eta\, (a^\ast \nabla\lambda)\bullet \phi.
\end{equation}
The transferred $L_\infty[1]$-structure is $\delta+ \nu$, where 
\begin{equation}\label{eq:TransLoo}
\nu = \pi\, (a'\, dt \, + a^\ast \lambda + a^\ast\tilde\lambda + a^\ast\nabla\lambda )\bullet \phi.
\end{equation}

\begin{lem}
\label{lem:lambda}
$
\pi(a^\ast \lambda)\bullet \phi 
= \pi(a^\ast \lambda)\bullet \iota.
$ 
\end{lem}
\begin{pf}
The output of $\eta$ consists of sections $\alpha\in\Gamma(I,a^\ast L)$
 such that $\alpha(0)=\alpha(1)=0$.
For any $n\geq 1$, by multi-linearity of $\lambda_n$,
  $(a^\ast \lambda_n)\big( \ldots, \eta(sth), \ldots\big)(t)=0$, for $t=0$ and
$t=1$.
  As the projection onto $P_\lin L$ is the linear interpolation, 
it follows that $\pi(a^\ast \lambda)\big( \ldots, \eta(sth), \ldots\big)=0$.
 Therefore, by  \eqref{eq:TransInc} and the definition of $\bullet$, 
\begin{align*}
\pi(a^\ast \lambda)\bullet \phi & = \pi(a^\ast \lambda)\bullet \iota + \pi(a^\ast \lambda)\big( \ldots, \eta(sth), \ldots\big) \\
& = \pi(a^\ast \lambda)\bullet \iota.
\end{align*}
\end{pf}

In particular, if we are given only one input, the recursive equation \eqref{eq:TransInc} reduces to 
\begin{equation}\label{eq:FirstMapInclusion}
\phi_1  = \iota - \eta \, (a^\ast \tilde\lambda_1 + a^\ast \nabla\lambda_0) \phi_1,
\end{equation}
and thus 
\begin{align*}
\nu_1 & = \pi\, ( a^\ast \lambda_1) \phi_1 + \pi\, (a^\ast\tilde\lambda_1 + a^\ast\nabla\lambda_0 ) \phi_1 \\
& = \pi\, ( a^\ast \lambda_1) \iota + \pi\, (a^\ast\tilde\lambda_1 + a^\ast\nabla\lambda_0 ) \iota + \pi\, (a^\ast\tilde\lambda_1 + a^\ast\nabla\lambda_0 ) \big(\eta(sth)\big) \\
& = \pi\, ( a^\ast \lambda_1) \iota + \pi\, (a^\ast\tilde\lambda_1 + a^\ast\nabla\lambda_0 ) \iota \\
& = \pi_\lin \, a^\ast \lambda_1 + \pi_\con \, a^\ast\tilde\lambda_1 + \pi_\con \, a^\ast\nabla\lambda_0. 
\end{align*}
Here we applied  Lemma \ref{lem:lambda} in  the second equality.

The operations with $0$ and $1$ inputs are summarized as follows.
\begin{equation}\label{eq:01inputs}
\begin{split}
\xymatrix@C=4pc{
a'\, dt \in  \Gamma_\con(I,a^\ast TM)\, dt\rto^-{\pi_\con a^\ast(\nabla\lambda_0)} &
\Gamma_\con(I,a^\ast L^1)\,dt \rto^-{\pi_\con a^\ast\tilde\lambda_1 }
&\ldots\\
 \pi_\lin a^\ast\lambda_0\in
\Gamma_\lin(I,a^\ast L^1)\ar[ru]_-{\delta}\rto_-{\pi_\lin
  a^\ast\lambda_1} & 
\Gamma_\lin(I,a^\ast L^2)\ar[ru]_-{\delta}\rto_-{\pi_\lin a^\ast\lambda_1} & \ldots 
}
\end{split}
\end{equation}

%

We need the following:

\begin{lem}\label{lem:ForSubalg}
For any $v_1, \ldots, v_n \in \Gamma_\lin(I,a^\ast L)$, we have 
\begin{items}
\item\label{term:TransIncProp}
$\phi_1(v_1) = \iota(v_1)$, $\phi_n(v_1, \ldots, v_n) = 0, \; \forall \, n \geq 2$;
\item
$\nu_n(v_1, \ldots, v_n) = \pi_\lin \big((a^\ast\lambda_n)(v_1, \ldots, v_n)\big),  \; \forall \, n \geq 1$.
\end{items}
\end{lem}
\begin{pf}
By \eqref{eq:TransLoo}, the second statement follows from the first one and the fact $a^\ast \tilde\lambda + a^\ast \nabla\lambda$ vanishes on $\Gamma_\lin(I,a^\ast L)$. 
To prove the first assertion, we  will show the following claim by induction on $n$. 

{\bf Claim.} The image of $\Gamma_\lin(I,a^\ast L)^{\otimes n}$ under $\phi_n$ is still in $\Gamma_\lin(I,a^\ast L)$. \\
In fact, since the operation $a^\ast \tilde\lambda + a^\ast \nabla\lambda$ vanishes if all the inputs are elements in $\Gamma_\lin(I,a^\ast L)$, Equation \eqref{eq:TransInc}, together with the claim, imply that $\phi_n(v_1, \ldots, v_n) = 0$,  for any $n \geq 2$ and any $v_1, \ldots, v_n \in \Gamma_\lin(I,a^\ast L)$.

We start with computing $\phi_1$. Since the operator $\big( \eta (a^\ast \tilde\lambda_1 + a^\ast \nabla\lambda_0)\big)^2$ vanishes, it follows from \eqref{eq:FirstMapInclusion} that  
\begin{align*}
\phi_1 & = \iota - \eta \, (a^\ast \tilde\lambda_1 + a^\ast \nabla\lambda_0) \phi_1 \\
& = \iota - \eta \, (a^\ast \tilde\lambda_1 + a^\ast \nabla\lambda_0) \big(\iota - \eta \, (a^\ast \tilde\lambda_1 + a^\ast \nabla\lambda_0) \phi_1 \big) \\
& = \iota - \eta \, (a^\ast \tilde\lambda_1 + a^\ast \nabla\lambda_0) \iota + \big( \eta (a^\ast \tilde\lambda_1 + a^\ast \nabla\lambda_0)\big)^2\phi_1 \\
 & = \iota - \eta (a^\ast \tilde\lambda_1 + a^\ast \nabla\lambda_0) \iota.
\end{align*}
Thus, the restriction of $\phi_1$ on $\Gamma_\lin(I,a^\ast L)$ coincides with $\iota$, which verifies the first part of  \ref{term:TransIncProp}.  Now assume the claim is true for $n$. Then since $a^\ast \tilde\lambda + a^\ast \nabla\lambda$ vanishes on $\Gamma_\lin(I,a^\ast L)$, it follows from \eqref{eq:TransInc} that the claim is also true for $n+1$.  
\end{pf}

\subsubsection{Differentiability}

The vector spaces $\Gamma_\con(I,a^\ast TM)\,dt$, as $a$ varies over $P_gM$, combine into a vector bundle $P_\con TM$ over $P_gM$, as mentioned above. Similarly, the vector spaces $\Gamma_\con(I,a^\ast L)\,dt$ and $\Gamma_\lin(I,a^\ast L)$ combine 
into vector bundles $P_\con L\,dt$ and $P_\lin L$ over $P_gM$, respectively. Note that $P_\lin L=\ev_0^\ast L\oplus \ev_1^\ast L$, canonically.  

We have constructed curved $L_\infty[1]$-structures in the fibers of $(P_\con TM\oplus P_\con L)\,dt\oplus P_\lin L$ over $P_gM$.  These combine into a bundle of smooth
curved $L_\infty[1]$-algebras, because all ingredients,
 $D$, $\lambda$, $\delta$ and $\eta$ are differentiable with respect to the base manifold $P_gM$, and all transferred operations are given by explicit formulas involving finite sums over rooted trees (see Remark~\ref{rmk:TreeFormula}).

\subsubsection{Factorization property}

Let $\M=(M,L,\lambda)$ be a derived manifold.
Consider the derived path space
$$\P\M=\big(P_g M, P_\con (TM\oplus  L)\,dt\oplus P_\lin L, \delta +\nu\big).$$
Let us denote the embedding of constant paths by $\iota:\M\to \P\M$.
By Lemma~\ref{lem:ForSubalg}, the map $\iota$ is an embedding of derived manifolds.

\begin{prop}
The morphism $\iota: \M\to \P\M$ of constant paths is a weak equivalence.
\end{prop}
\begin{pf}
The classical locus of $\M$ is the zero locus of $\lambda_0:M\to L^1$.  The classical locus of $\P\M$ is the zero locus of 
$$D+P_\lin\lambda_0:P_gM\longrightarrow P_\con TM \, dt \, \oplus P_\lin L^1: \, a \mapsto a'\, dt\, + \pi_\lin a^\ast \lambda_0 .$$
The canonical inclusion of the former into the latter is a bijection, because we can first compute the zero locus of $D$, which consists of constant paths. 

We compute the induced map on tangent complexes at a classical point $P$ of $\M$.  It is given by the following 
$$ \xymatrix{
TM|_P\rto^-{D_P\lambda_0}\dto_-{\Delta} & L^1|_P\rto^-{\lambda_1|_P}\dto_-{0\oplus\Delta} & L^2|_P\rto^-{\lambda_1|_P} \dto_-{0\oplus\Delta} &\ldots\\
TM|_P\oplus TM|_P\rto_-{D_P\nu_0} & TM\,dt|_P\oplus (L^1\oplus L^1)|_P\rto_-{(\delta+\nu_1)|_P } &L^1\,dt|_P\oplus(L^2\oplus L^2)|_P\rto_-{(\delta+\nu_1)|_P } & \ldots} $$
Here the bottom row, i.e., the tangent complex of  $\P\M$ at the classical point $P$, follows from \eqref{eq:01inputs}. More 
explicitly, the maps are 
\begin{gather*}
(D_P\nu_0)(v,w) = \big((w-v)\, dt, \, D_P\lambda_0(v), \, D_P\lambda_0(w)\big), \\
(\delta+\nu_1)|_P(v\, dt, y_1,z_1) = \big(
 D_P\lambda_0(v) \, dt +  (y_1-z_1)\, dt\, ,\, \lambda_1(y_1), \,  \lambda_1(z_1) \big), \\
(\delta+\nu_1)|_P(x_k \, dt, y_{k+1},z_{k+1}) = \big( \lambda_1(x_k) \, dt\, +(-1)^{k+1}(z_{k+1}-y_{k+1})\, dt\, , \, \lambda_1(y_{k+1}), \,  \lambda_1(z_{k+1}) \big),
\end{gather*}
where $v,w \in TM|_P$, $x_k, y_k, z_k \in L^k|_P$. Note that the first component $(w-v)\, dt$ in $(D_P\nu_0)(v,w)$  
is computed by the composition 
$$
T(M\times M)|_{(P,P)} \xrightarrow{T(\exp^\nabla)^{-1}|_{(P,P)}} T(TM)|_{0_P} \xrightarrow{\pr} T^\fiber_{0_P}(TM) \cong T M|_P.
$$ 

After computing cohomology vertically, we are left with 
the mapping cone of the identity cochain map of the tangent complex of $\M$ at $P$. 
Since the remaining complex is exact, a standard spectral sequence argument shows that the morphism induced on tangent complexes at $P$ is a quasi-isomorphism, as required. 
\end{pf}

This concludes the proof of   Theorem~\ref{thm:Seoul} (1).

We define a morphism $\P\M\to\M\times\M$ as the fiberwise composition of $L_\infty[1]$-morphisms 
\begin{equation}\label{eq:Evaluation}
P\M|_a \xrightarrow{\phi} \widetilde PT\M[-1]|_a 
\xrightarrow{\ev_0 \oplus \ev_1} L|_{a(0)} \oplus L|_{a(1)},
\end{equation}
where $\phi$ is defined by the homotopy transfer equation \eqref{eq:TransInc}, and $\ev_t$ is the composition of the projection onto $\Gamma(I, a^\ast L)$ followed by the evaluation map at $t$.

\begin{prop}\label{wow1}
The morphism $\P\M\to \M \times \M$ defined above 
is the linear morphism given by projecting onto $P_\lin L$ composing with evaluation at $0$ and $1$. 
\end{prop}
\begin{pf}
It suffices to show that the composition of $L_\infty[1]$-morphisms \eqref{eq:Evaluation} is a linear morphism. 
To prove linearity, we observe from \eqref{eq:TransInc} that $\im(\phi - \iota) \subset \im(\eta)$. Here $\phi$ and $\iota$ are considered as maps $\Sym\P\M|_a \to \widetilde PT\M[-1]|_a$. Thus, the linearity follows from the fact $\im(\eta) \subset \ker(\ev_0 \oplus \ev_1)$.
\end{pf}

This concludes the proof of   Theorem~\ref{thm:Seoul} (2).

\subsubsection{Canonicality}

We now prove the last part of Theorem~\ref{thm:Seoul}: canonicality. First, different choices of affine connections on $M$ result diffeomorphic $P_gM$ (along an open neighborhood of constant paths). Hence, without loss of generality, we may fix any affine connection on $M$ and the corresponding $P_gM$.   

Let $\nabla$ and $\bar\nabla$ be any two connections on $L$. We will construct a morphism $\Psi:\P^\nabla\M \to \P^{\bar\nabla}\M$ of the corresponding derived manifolds such that it is an isomorphism in an open neighborhood of constant paths in $P_gM$. 

Following the notations in the proof of Proposition~\ref{pro:TM1}, we denote by $\alpha$ the bundle map $\alpha:TM \otimes L \to L$ such that 
$\alpha(v,x) = \bar\nabla_v x - \nabla_v x, \, \forall \, v \in \Gamma(TM), \, x \in \Gamma(L)$, and denote by  $\Phi = \Phi^{\bar\nabla,\nabla}$ the isomorphism \eqref{eq:CanonicalIsoTM[-1]}.

\begin{lem}
The morphism $\Phi$ induces a morphism of curved $L_\infty[1]$-algebras
$a^\ast\Phi: \widetilde P^\nabla T\M[-1]|_a \to \widetilde P^{\bar\nabla} T\M[-1]|_a$ 
at each short geodesic path $a \in P_gM$. 
\end{lem}
\begin{pf}
To prove $a^\ast\Phi$ is a morphism, we check the defining equation of $L_\infty[1]$-morphism:
$$
a^\ast\Phi \circ (a'\, dt\, + \delta^\nabla + a^\ast\mu^\nabla) = (a'\, dt\, + \delta^{\bar\nabla} + a^\ast\mu^{\bar\nabla}) \bullet a^\ast\Phi
$$
which is considered as a sequence of equations labeled by the number of inputs. In fact, it is simple to see that the higher equations (with 2 or more inputs) follow from the fact $\Phi$ is an $L_\infty[1]$-morphism: $\Phi\circ \mu^\nabla = \mu^{\bar\nabla} \bullet \Phi$, and the zeroth equation $a^\ast\Phi_1(a' \, dt + a^\ast \mu^\nabla_0) = a' \, dt + a^\ast \mu^{\bar\nabla}_0 $ simply follows from the definitions of $\mu^\nabla$, $\mu^{\bar\nabla}$ and the fact $\Phi_1 = \id$.  Thus, it remains to check the first equation:  
For any $\xi = v\, dt + y\, dt + x \in  \widetilde P^\nabla T\M[-1]|_a$, we have 
\begin{align*}
& a^\ast\Phi_2(a' \, dt + a^\ast \mu^\nabla_0 , \xi) + a^\ast\Phi_1\big((\delta^\nabla + a^\ast\mu^\nabla_1)(\xi) \big) \\
& \qquad = 
\alpha(v , a^\ast\lambda_0)\, dt\, + (-1)^{|x|} \alpha(a', x)\, dt \, \\
& \qquad \quad + (-1)^{|x|} \nabla_{a'} x \, dt  \, + a^\ast(\nabla_{v} \lambda_0) \, dt \, + a^\ast\lambda_1(y) \, dt \, + a^\ast\lambda_1(x) \\
& \qquad = (-1)^{|x|} \bar\nabla_{a'} x \, dt  \, + a^\ast(\bar\nabla_{v} \lambda_0) \, dt \, + a^\ast\lambda_1(y) \, dt \, + a^\ast\lambda_1(x) \\
& \qquad  =  (\delta^{\bar\nabla} + a^\ast \mu^{\bar\nabla}_1)\big(\Phi^{\bar\nabla,\nabla}_1(\xi)\big).
\end{align*} 
This completes the proof. 
\end{pf}

By Proposition~\ref{transfertheorem}, we have the $L_\infty[1]$-morphisms
\begin{gather*}
\phi^\nabla: \P^\nabla\M|_a \longrightarrow \widetilde P^\nabla T\M[-1]|_a, \\
\tilde\pi^{\bar\nabla}: \widetilde P^{\bar\nabla} T\M[-1]|_a\longrightarrow \P^{\bar\nabla}\M|_a.
\end{gather*}
See also \eqref{eq:KIAS1} and \eqref{eq:KIAS2}. We define $\Psi:\P^\nabla\M \to \P^{\bar\nabla}\M$ by 
$$
\Psi|_a := \tilde\pi^{\bar\nabla} \bullet (a^\ast\Phi^{\bar\nabla,\nabla}) \bullet \phi^\nabla.
$$

\begin{prop}
There exists an open neighborhood $U$ of constant paths in $P_gM$ such that the $L_\infty[1]$-morphism $\Psi|_a:\P^\nabla\M|_a \to \P^{\bar\nabla}\M|_a$ is an isomorphism for any $a \in U$.
\end{prop}
\begin{pf}
For $a \in P_gM$, the linear part $\Psi_1|_a$ is given by   
\begin{align*}
\Psi_1|_a & = \tilde\pi^{\bar\nabla}_1  \id  \phi^\nabla_1  \\
& = \pi^{\bar\nabla}( \id + \nu^{\bar\nabla}_1\eta^{\bar\nabla})^{-1} (\id + \eta^\nabla \nu^\nabla_1)^{-1} \iota^\nabla 
\qquad (\text{by Remark~\ref{rmk:0th1stMapInHptTransfer}}) \\
& = \pi^{\bar\nabla}( \id - \nu^{\bar\nabla}_1\eta^{\bar\nabla}) (\id - \eta^\nabla \nu^\nabla_1) \iota^\nabla  \\
& = \pi^{\bar\nabla}\iota^\nabla - \pi^{\bar\nabla}\nu^{\bar\nabla}_1\eta^{\bar\nabla}\iota^\nabla - \pi^{\bar\nabla}\eta^\nabla \nu^\nabla_1 \iota^\nabla  \qquad   \\
& = \pi^{\bar\nabla}\iota^\nabla.    
\end{align*}
Here we used the equations $\eta^{\bar\nabla}\eta^{\nabla} = 0$, $\pi^{\bar\nabla}\nu^{\bar\nabla}_1\eta^{\bar\nabla} =0$ and $\pi^{\bar\nabla}\eta^\nabla =0$.  
The last two equations can be proved by a similar argument as in the proof of Lemma~\ref{lem:lambda}. 

Now let $a: I \to M$ be a constant path at $P\in M$. 
Under the identification in Remark~\ref{rmk:DerPathSpAsVB}, the map $\Psi|_a = \pi^{\bar\nabla}\iota^\nabla$ is the identity map and hence an isomorphism. Therefore, $\Psi|_a$ is an isomorphism in an open neighborhood of constant paths. 
\end{pf}

The following example shows that $\Psi|_a$ might not be an isomorphism when $a$ is far away from constant paths.

\begin{ex}
Let $M=\rr$, $L = M \times \rr^2[-1]$. Let $\nabla^0$ and $\nabla$ be the connections with the properties:
\begin{gather*}
\nabla^0_{\partial_t}e_1 = 0, \quad \nabla^0_{\partial_t}e_1 = 0; \\
\nabla_{\partial_t}e_1 = -2\pi \, e_2, \quad \nabla_{\partial_t}e_2 = 2\pi \, e_1;
\end{gather*}
where $\{e_1, e_2\}$ is the standard frame for $L$. For the path $a(t)=t$, the $\nabla$-constant paths $\Gamma^{\nabla}_\con(I,a^\ast L)$ over $a$ is the subspace in $\Gamma(I,a^\ast L) \cong C^\infty(\rr, \rr^2)$ spanned by the basis 
\begin{align*}
e^\nabla_1 &= \hphantom{-} \cos(2\pi t) e_1 + \sin(2\pi t) e_2, \\
e^\nabla_2 &= -\sin(2\pi t) e_1 + \cos(2\pi t) e_2. 
\end{align*}
Then 
\begin{multline*}
\Psi_1^{\nabla^0 \nabla}|_a(e_1^\nabla \, dt) = \pi^{\nabla^0}\iota^\nabla(e_1^\nabla \, dt) \\
 = \Big( \int_0^1\cos(2\pi s) \, ds\, \cdot \, e_1 + \int_0^1 \sin(2\pi s) \, ds \, \cdot\, e_2 \Big) \, dt = 0.
\end{multline*}
In fact, one can show that $\Psi_1^{\nabla^0 \nabla}|_a$ is isomorphic to the projection operator $\rr \, dt  \oplus \rr^2[-1] \, dt  \oplus \rr^4[-1] \to \rr \, dt  \oplus \rr^2[-1] \, dt  \oplus \rr^4[-1]: v\, dt\, + y\, dt\, + x \mapsto v\, dt\, + x$.
\end{ex}

\subsubsection{Categorical structure}

The composition $\M\to\P\M\to\M\times\M$ is clearly the
 diagonal.

\begin{prop}
The morphism $\P\M\to\M\times\M$ is a fibration of derived manifolds.
\end{prop}
\begin{pf}
On the level of the base manifolds, this morphism is the open embedding $P_gM\to M\times M$. The linear part (which is the whole, by 
Proposition~\ref{wow1}) of the $L_\infty[1]$-morphism from
 $P_\con(TM\oplus L)\,dt \oplus P_\lin L$ to $\ev^*_0L\oplus \ev^*_1 L$ is simply the projection onto $P_\lin L$, followed by the identification $P_\lin L 
\cong \ev^*_0L\oplus \ev^*_1 L$.  This is an epimorphism of vector bundles.
\end{pf}

\begin{thm}
\label{thm:main}
The category of derived manifolds forms a category of fibrant objects.
\end{thm}
\begin{pf}
The factorization $\M\to\P\M\to\M\times\M$ of the diagonal proved above was the last ingredient needed.
\end{pf}

\subsubsection{Further remarks}

As we have seen in the previous subsections, the derived path space $\P\M$ of the derived manifold $\M=(M,L,\lambda)$ is given by the graded vector bundle
$$\P\M=P_\con(TM\oplus L)\,dt\oplus P_\lin L\,,$$
over the manifold $P_gM$.  Here $P_\con(TM\oplus L)\,dt$ is a curved $L_\infty[1]$-ideal in $\P\M$.  The fibration $\P\M\to\M\times \M$ is given by dividing out by this $L_\infty[1]$-ideal, and identifying the quotient $P_\lin L$ with $\ev_0^\ast L\oplus \ev_1^\ast L$, over the embedding $P_gM\subset M\times M$. 

Note that $P_\lin L\subset \P\M$ is not a curved $L_\infty[1]$-subalgebra, as the embedding does not preserve the curvature.

If $\M$ is of amplitude $n$, then $\P\M$ is a derived manifold of amplitude $n+1$, and thus for degree reasons, the $k$-th brackets on $\P\M$ vanish for all $k \geq n+1$.

The curvature on $\P\M$ is given in \eqref{eq:TransferredCurvature}. 
In general, the operations on $\P\M$ can be computed by the recursive formulas \eqref{eq:TransInc} and \eqref{eq:TransLoo}, or by the tree formulas in Remark~\ref{rmk:TreeFormula}. More explicitly, the operations $\delta+\nu$ on $\P\M$ decompose into a part that has image in $P_\lin L$, and a part that has image in $P_\con(T_M\oplus L)\,dt$.  The part that has image in $P_\lin L$ is the term $\pi(a^\ast\lambda) \bullet \phi$ in \eqref{eq:TransLoo} and is identified with $\ev_0^\ast L\oplus \ev_1^\ast L$, as seen in Lemma~\ref{lem:lambda}. 

So the only interesting part is that which has image in $P_\con(T_M\oplus L)\,dt$.  The tree sum for this part is such that the operation on every node is $\nabla\lambda+\tilde\lambda$ with exactly one of the incoming edges is labeled with an input from $P_\con(T_M\oplus L)\,dt$.

Let us work out a few examples.  The case where $\M$ is a manifold was treated in Section~\ref{sec:coam}. 

\begin{ex}
Suppose that $L=L^1$. This is the {\em quasi-smooth }case. 
Then the derived path space of $\M=(M,L,F)$ is the following  derived manifold of
amplitude~2: the base is $P_gM$, the manifold of short geodesic paths. The vector bundle in degree 1 has fiber $\Gamma_\con(I,a^\ast TM) \, dt   \oplus \Gamma_\lin(I,a^\ast L)$ over the geodesic path $a:I \to M$. The curvature
is $a'dt \in \Gamma_\con(I,a^\ast TM) \, dt$, and the linear interpolation between
$F|_{a(0)}$ and $F|_{a(1)}$ in $\Gamma_\lin(I,a^\ast L)$. In degree 2, the vector
bundle has fiber $\Gamma_\con(I, a^\ast L)\,dt$, and the twisted differential is given by 
\begin{align*}
\Gamma_\con(I,a^\ast TM) \, dt & \longrightarrow \Gamma_\con(I, a^\ast L)\,dt \\
v & \longmapsto \Big(\int_0^1 \nabla_{v(u)}F(u)\,du \Big)\, dt
\end{align*}
and
\begin{align*}
\Gamma_\lin(I, a^\ast L)& \longrightarrow \Gamma_\con(I, a^\ast L)\,dt\\
\alpha & \longmapsto \delta\alpha = -\big(\alpha(1)-\alpha(0)\big)\,dt \, .
\end{align*}
All the higher operations vanish. 
\end{ex}

\begin{ex}
If $\M=(M,L,\lambda)$ is of amplitude 2, the only non-zero operations are the curvature $\lambda_0$ and the differential $\lambda_1$.  In this  case $\P\M$ has amplitude 3.  
The binary bracket on $\P\M$ is the sum of two operations
$$\pi_\con\nabla\lambda_1:\Gamma_\con(I,a^\ast TM) \, dt\otimes \Gamma_\lin(I, a^\ast L^1)\longrightarrow \Gamma_\con(I, a^\ast L^2)\,dt$$
and the operation 
$$\Gamma_\con(I,a^\ast TM) \, dt\otimes \Gamma_\con(I,a^\ast TM) \, dt \longrightarrow \Gamma_\con(I, a^\ast L^2)\,dt$$
given by 
$$x\otimes y\longmapsto - \pi_\con(\nabla\lambda_1)\big(\eta\nabla\lambda_0(x),y\big) - \pi_\con(\nabla\lambda_1)\big(x,\eta\nabla\lambda_0(y)\big) \,.$$
\end{ex}

\begin{rmk}
It   would be interesting to extend the above  construction of derived path spaces to that of
``higher derived path spaces", and study their relation with
the infinite category structure of derived manifolds.
\end{rmk}

\subsection{Homotopy fibered products and derived intersections}\label{rmk:HptFibProd}

As an important application, Theorem \ref{thm:main} allows us
to form the ``homotopy fibered product"  of derived manifolds.

Let $\C$ be a category of fibrant objects. Consider its \emph{homotopy category} $\Ho(\C)$ which is obtained by formally inverting weak equivalences. 
For given arbitrary morphisms $X \to Z \leftarrow Y$  in $\C$, one can form a well-defined \emph{homotopy fibered product} $X \times_Z^h Y$ in the homotopy category $\Ho(\C)$.  If $P$ is a path space of $Z$, then $X \times_Z P \times_Z Y$ is a representation of the homotopy fibered product $X \times_Z^h Y$ \cite{MR341469}. It is standard that such a fibered product is indeed well-defined. However, since we are not able to locate this result in literature, we sketch a proof below.

\begin{lem}
The isomorphism class of $X \times_Z P \times_Z Y$ in the homotopy category $\Ho(\C)$ is independent of the choice of the path space $P$. 
\end{lem}
\begin{pf}
Suppose $P$ and $P'$ are path spaces of $Z$, i.e., we have two factorizations of the diagonal map $Z \xrightarrow{\Delta} Z \times Z$: 
$$
\xymatrix{
Z \ar[r]^\sim \ar[d]_\sim \ar[rd]_\Delta & P' \ar@{->>}[d] \\
P \ar@{->>}[r] & Z \times Z \, .
}
$$
%
Since both $P \twoheadrightarrow Z\times Z$ and $P' \twoheadrightarrow Z\times Z$ are fibrations, there exists the fibered product $P \times_{Z \times Z} P'$ whose projections $P \times_{Z \times Z} P' \twoheadrightarrow P$ and $P \times_{Z \times Z} P' \twoheadrightarrow P$ are also fibrations. By the universal property of $P \times_{Z \times Z} P'$, there is a unique map $Z \to P \times_{Z \times Z} P'$ such that the diagram 
$$
\xymatrix{
Z \ar@/^1pc/[rrd]^\sim \ar@/_1pc/[rdd]_\sim \ar@{-->}[rd] & & \\
& P \times_{Z \times Z} P' \ar@{->>}[r] \ar@{->>}[d]  & P' \ar@{->>}[d] \\
& P \ar@{->>}[r]  &   Z \times Z
}
$$
commutes. By factorizing the morphism $Z \to P \times_{Z \times Z} P'$, there exists $P''$ such that $Z \xrightarrow{\sim} P'' \twoheadrightarrow P \times_{Z \times Z} P'$ is equal to $Z \to P \times_{Z \times Z} P'$. 
Since the diagram 
$$
\xymatrix{
P'' \ar@{->>}[r] & P \times_{Z \times Z} P' \ar@{->>}[d] \\
Z \ar[u]^\sim \ar[r]_\sim & P
}
$$
commutes, it follows from the 2-out-of-3 axiom that the composition $P'' \twoheadrightarrow P \times_{Z \times Z} P' \twoheadrightarrow P$ is a weak equivalence. It is clear that it is also a fibration and hence a trivial fibration. Similarly, so is the composition $P'' \twoheadrightarrow P \times_{Z \times Z} P' \twoheadrightarrow P'$. 
In summary, we obtain a third path space $P''$ which dominates $P$ and $P'$ in the sense that the diagram
$$
\xymatrix{
 & P' \ar@{->>}[rd] &  \\
Z \ar[ru]^\sim \ar[r]^\sim \ar[rd]_\sim & P'' \ar@{->>}[r] \ar@{->>}[u]_\sim \ar@{->>}[d]^\sim & Z \times Z  \\
 & P \ar@{->>}[ru] & 
}
$$
commutes.

Recall that the fibered product $X \times_Z P \times_Z Y$ can be obtained by first pulling back $P \twoheadrightarrow Z$ along $Y \to Z$ and then pulling back $P\times_Z Y \twoheadrightarrow Z$ along $X \to Z$: 
$$
\xymatrix{
X \times_Z P \times_Z Y \ar[r] \ar@{->>}[dd] & P \times_Z Y \ar[d] \ar@/_1pc/@{->>}[dd] \ar@{->>}[r] & Y \ar[d] \\
& P \ar@{->>}[r] \ar@{->>}[d] & Z \, . \\
X \ar[r] & Z & 
}
$$
See \cite{MR341469} for details. Now we pull back $P'' \twoheadrightarrow P$ along $P \times_Z Y \to  P$ and obtain the trivial fibration
$$
P'' \times_Z Y \cong P'' \times_P (P \times_Z Y) \twoheadlongrightarrow P \times_Z Y \, . 
$$ 
Then we pull back $P'' \times_Z Y \twoheadrightarrow P \times_Z Y$ along $X \times_Z P \times_Z Y \to P \times_Z Y$  and obtain the trivial fibration
$$
X \times_Z P'' \times_Z Y \cong (X \times_Z P \times_Z Y)\times_{P \times_Z Y} (P'' \times_Z Y) \twoheadlongrightarrow  X \times_Z P \times_Z Y \,.
$$
By a similar argument, we also have a trivial fibration 
$$
X \times_Z P'' \times_Z Y  \twoheadlongrightarrow  X \times_Z P' \times_Z Y \,.
$$
Therefore $X \times_Z P \times_Z Y$ and $X \times_Z P' \times_Z Y$ are isomorphic in the homotopy category $\Ho(\C)$. 
\end{pf}

Let $\M \to \N \leftarrow \M'$ be arbitrary morphisms of derived manifolds, and $\P$ be a path space of $\N$. Since there is a weak equivalence $\N \xrightarrow\sim \P$, the virtual dimensions of $\N$ and $\P$ are equal. Thus, by Remark~\ref{rmk:DerDim&FibProd}, we have 
\begin{align*}
\ddim(\M \times^h_\N \M') & = \ddim(\M \times_\N \P \times_\N \M') \\
& = \ddim(\M) - \ddim(\N) + \ddim(\P) -\ddim(\N) + \ddim(\M') \\
& = \ddim(\M) + \ddim(\M') - \ddim(\N)\,.
\end{align*}

%
%

\subsubsection*{Derived intersections}

Let $M$ be a smooth manifold, and $X,Y$ be submanifolds of $M$. The \emph{derived intersection} $X \cap^h_M Y$ of $X$ and $Y$ in $M$ is understood as the homotopy fibered product $X \cap^h_M Y := X \times_M^h Y$ in the homotopy category of derived manifolds. By Section~\ref{sec:coam}, the quasi-smooth derived manifold $\P=(P_gM,P_\con TM\,dt,D)$ is a path space of $M$. Thus, the derived intersection $X \cap^h_M Y$ is represented by the quasi-smooth derived manifold $X \times_M \P \times_M Y = (N,E,\tilde D)$. 
Here, the base space $N = X \times_M P_g M \times_M Y$ is the space of short geodesics which start from a point in $X$ and end at a point in $Y$. Since $P_gM$ can be identified with an open neighborhood of the diagonal in $M \times M$, the base space $N$ is an open submanifold of $X\times Y$
consisting of pairs of points $(x , y)\in X\times Y$
such that both $x$ and $y$ are sufficiently close to
the set-theoretical intersection $X \cap Y$. 
 The fiber $E|_a$ over $a \in N$ is the space 
$$
E|_a = \{ \alpha\, dt \mid \alpha \in \Gamma(a^\ast TM), \; (a^\ast\nabla) (\alpha) = 0 \} \,,
$$ 
and the section $\tilde D: N \to E: a \mapsto a' \, dt$ is given by derivatives. 
The classical locus of $X \cap^h_M Y$ is the set-theoretical intersection $X \cap Y$. 
The virtual dimension of $X \cap^h_M Y$ is 
\begin{align*}
\ddim(X \cap^h_M Y) &= \ddim(X \times_M^h Y )\\
 & = \dim(X) + \dim(Y)- \dim(M) \,.
\end{align*}

\section{Inverse function theorem} 

\subsection{The inverse function theorem for derived manifolds}

\begin{thm}\label{ift}
Let $\M=(M,L,\lambda)$ and $\N=(N,L',\lambda')$ be derived manifolds, and
$(f,\phi):\M\to\N$ a fibration.  Let $Z\subset \pi_0(\M)$ be a subset
of points at which $(f,\phi)$ is \'etale. Then
there exists  \eetale \
  of manifolds  $Y\to N$ and a commutative diagram
$$\xymatrix{
Z\rto\dto & \M\dto\\
\Y\rto\urto & \N\rlap{\,.}}$$
Here $\Y$ is the pullback of  $(L',\lambda')$ via $Y\to N$, and
the lower triangle consists of   morphisms of derived manifolds.   

If, moreover,  $\pi_0(\M) \to\pi_0(\N)$ is injective, we can choose $Y$ to be an open
 submanifold  of $N$.
\end{thm}

\begin{cor}\label{cor:LocalSec}
A fibration admits local sections through every point at which it is \'etale.

Every trivial fibration of derived manifolds admits a section, after
restricting the target to an open neighborhood of its classical
locus. 
\end{cor}
\begin{pf}
Apply the theorem for either $Z$ being a single point, or $Z=\pi_0(\M)$. 
\end{pf}

\begin{rmk}
Therefore, if we replace all derived manifolds by the germ or
$\infty$-jets around its
classical locus, then we have a category of fibrant objects
where every trivial fibration has a section. 
This would have important  consequences about the Hom-categories in the simplicial localization.
\end{rmk}

\begin{rmk}
Let $\M$ and $\N$ be smooth manifolds (i.e. $L = M \times 0$ and $L' = N \times 0$). Then a morphism $(f,0):\M \to \N$ is \'etale at $P$ if and only if the tangent map of $f$ is an isomorphism at $P$. In this extreme case, Corollary~\ref{cor:LocalSec} reduces to the inverse function theorem for smooth manifolds.
\end{rmk}

\begin{rmk}
In classical differential geometry, any fibration, i.e. submersion,
admits a local section through any point. This may not be true for derived manifolds. 
It would be interesting to investigate when a local section exists in general.
\end{rmk}

\subsection{Proof of Theorem~\ref{ift}}

Without loss of generality, we will assume the fibration $(f,\phi):\M\to\N$ is linear, i.e. $\phi=\phi_1$.

We proceed the proof of Theorem~\ref{ift} in two steps: 
\begin{items}
\item
construct a  weak equivalence $(M,H,\mu) \to  (M,L,\lambda)$ such that the composition 
$$
(M,H,\mu) \to (M,L,\lambda) \to (N,L',\lambda')
$$
is a linear fibration whose restriction $\big(M, H^{\geq 2} \big) \to \big( N, (L')^{\geq 2} \big)$ is an isomorphism of graded vector bundles; 
\item 
construct a \emph{derived submanifold} $(Y,E,\nu) \hookrightarrow (M,H,\mu)$ such that $Z \subset Y$ and the composition 
$$
(Y,E,\nu) \to (M,H,\mu) \to (M,L,\lambda) \to (N,L',\lambda')
$$
is a   fiberwise linear isomorphism of bundle of  curved $L_\infty[1]$-algebras. 
\end{items}
Here by a derived submanifold, we mean $Y$ is a submanifold of $M$ and $E$ is a subbundle of $H|_Y$ such that the inclusion map is a linear morphism of derived manifolds.

\begin{defn}
 A morphism $(M,H,\mu) \to (M,L,\lambda)$ of
 derived manifolds is called a \textbf{transferred inclusion}
if $(M,H,\mu)$ is a bundle of curved $L_\infty[1]$-algebras  over $M$
obtained from $(M,L,\lambda)$ by applying
the homotopy  transfer  theorem (Proposition  \ref{transfertheorem}).
 
A morphism $(M,H,\mu) \to (M,L,\lambda)$ of derived manifolds
 is a called a \textbf{strong embedding}
 if it is the composition of
  a (finite) sequence of  transferred inclusions.
\end{defn}

\begin{lem}
\label{firstcase}
Consider a fixed base manifold
$M$, with a linear morphism of bundles of curved $L_\infty[1]$-algebras
$\phi: (L, \lambda)\to (L', \lambda') $ over it.
Suppose that 
\begin{items}
\item $\phi:L\to  L'$ is an isomorphism of vector bundles in
degrees $\geq k+1$, and an epimorphism in degrees $\leq k$. 
\item Moreover, let $Z\subset MC(L)$ be a subset of the Maurer-Cartan locus
of $(L, \lambda)$ such that at every point $P$ of $Z$ we have that 
$h(L|_P, \lambda_1)\to h(L'|_P, \lambda'_1)$ is an isomorphism in degrees
 $\geq k$ and
surjective in degree $k-1$.
\end{items}
Then, after restricting to an open neighborhood of $Z$ in $M$, if
necessary,  there exists a transferred  inclusion
$\iota:H\to L$  such that 
the composition $\phi\circ\iota:H\to L'$ is a linear morphism of 
bundles of $L_\infty[1]$-algebras, 
which is an isomorphism in degrees $\geq k$ and an epimorphism in degrees $\leq k-1$.
\end{lem}
\begin{pf}
Denote by $K$ the kernel of $\phi:L\to L'$. It is a graded vector
bundle, and an $L_\infty[1]$-ideal in $L$. 
Consider the diagram
$$\xymatrix{
&K^{k-1}\dto^j\rto^{\lambda_1} & K^k\dto^j\\
L^{k-2}\rto&L^{k-1}\rto^{\lambda_1} & L^k\rto& L^{k+1}\rlap{\,.}}$$
A diagram chase proves that at all points of $Z$, the vector bundle
homomorphism $\lambda_1:K^{k-1}\to K^k$ is surjective.  This
will still be the case in an open neighborhood, so we may assume,
without loss of generality, that this map is an epimorphism over all
of $M$.  We then choose a section $\chi:K^k\to K^{k-1}$ of $\lambda_1$, and a retraction of $\theta:L^k\to K^k$ of $j:K^k\to L^k$, and define
$\eta:L^k\to L^{k-1}$ to be equal to $\eta=j\chi\theta$.  
$$\xymatrix@=3pc{
&K^{k-1}\dto^j\ar[r]^{\lambda_1 } & K^k\dto^j\ar@/^/[l]^\chi\\
L^{k-2}\rto & L^{k-1}\rto_{\delta=\lambda_1} &
L^k\rto\ar@/^/[u]^\theta\ar@/_/[l]_{\eta=j\chi\theta}& L^{k+1}}$$
We also write $\delta$ for $\lambda_1:L^{k-1}\to L^k$, but set
$\delta=0$ elsewhere, to define a differential $\delta:L\to L$ of degree~1.  Our definition ensures that $\delta^2=0$, although
$\lambda_1^2\not=0$. 

One checks that $\eta\delta\eta=\eta$, and so $\eta\delta$ and
$\delta\eta$ are projection operators, so in both cases kernel and
image are bundles. Let $H^{k-1}=\ker \eta\delta$, and $H^k=\ker
\delta\eta$. For $i\not=k,k-1$, we set $H^i=L^i$. 

The two projection operators 
$j\theta= \delta\eta$ coincide, and hence   $H^k=\ker \delta\eta$
is a complement for $K^k=\im j\theta$.  Therefore,   the
composition $H^k\to L^k\to {L'}^{k}$ is an isomorphism. 

We have $L^{k-1}=\ker \eta\delta+\im \eta \delta\subset
\ker\eta\delta+K^{k-1}=H^{k-1}+K^{k-1}$, because $\eta$ factors
through $K^{k-1}$, by definition.  It follows that the composition
$H^{k-1}\to L^{k-1}\to {L'}^{k-1}$ is surjective. 

Now $\lambda-\delta$ is a curved $L_\infty[1]$-structure on the complex of vector bundles $(L,\delta)$, and $\eta$ is a contraction of $\delta$. We apply the transfer theorem: Proposition~\ref{transfertheorem} for bundles of curved $L_\infty[1]$-algebras.    This gives rise
to a structure $\mu$, of a bundle of curved $L_\infty[1]$-algebras on the complex $(H,\delta)$,
together with a morphism of curved $L_\infty[1]$-structures $\iota:H\to L$, whose linear part is given by
the inclusion $H\subset L$. Here the linear part is the inclusion because $\eta(\lambda_1 -\delta) =0$. 
We consider $\delta+\mu$ as a curved $L_\infty[1]$-structure on $H$. 

The composition $\phi\bullet\iota$ is linear, because all non-trivial trees involved in $\iota$ have $\eta$ at the root, and $\phi\circ\eta=0$.

In order to apply the transfer theorem, we use the  variation of the
natural filtration at level $k$, see Example~\ref{varfil},  on
$L$. Both $\delta$ and $\eta$ preserve this filtration, hence $H$
inherits this filtration, and  the transfer theorem applies. 
\end{pf}

\begin{lem}\label{lastcase}
Let $(f,\phi):(M,H,\mu)\to (N,L',\lambda')$ be a linear fibration of derived manifolds such that  $\phi:H^k\to f^\ast (L')^k$ is an isomorphism, for all $k\geq2$.
Let $Z\subset MC(H)$ be a subset of the Maurer-Cartan locus of $(H,\mu)$,
such that the morphism of derived manifolds $(f,\phi)$ is \'etale at
every point of $Z$. 
After restricting to an open neighborhood of $Z$ in $M$, there exists a
submanifold $Y\subset M$, and subbundle $E^1$ of $H^1|_Y$, such that 
\begin{items}
\item $Z\subset Y$,
\item the restriction of the curvature $\lambda_0|_Y$ is contained in $\Gamma(Y,E^1)\subset \Gamma(Y,H^1|_Y)$, so that $E:= E^1\oplus H^{\geq2}|_Y$ is a bundle of curved $L_\infty[1]$-subalgebras of $H|_Y$, 
\item   the composition $Y\to M\to N$ is  \eetale,
\item    the map $\phi|_Y:E^1|_Y\to f^\ast (L')^1|_Y$ is an isomorphism of vector bundles, so that the composition $(Y,E)\to (Y,f^\ast L'|_Y)$ is a linear isomorphism of bundles of curved $L_\infty[1]$-algebras.
\end{items}
In particular, the composition $(Y,E)\to (N, L')$ is \'etale at all points of $Z\subset Y$.  
\end{lem}
\begin{pf}
We have a diagram
$$\xymatrix{
M\rto^-s\dto_f & H^1\dto^\phi   \\
N\rto^-t & (L')^1}$$
Here $f:M\to N$ is a submersion, and $\phi:H^1\to f^\ast (L')^1$ an
epimorphism. We have written $s=\mu_0$ and $t=\lambda_0'$ for the respective curvatures. 

We denote by $K$  the kernel of $\phi:H^1\to f^\ast (L')^1$.  We choose 
\begin{items}
\item a connection in $(L')^1$, 
\item a retraction  $\theta:H^1\to K = \ker(\phi)$ of the inclusion 
$K\to   H^1$, giving rise to a splitting $H^1=K\oplus \tilde E^1$, where $\tilde E^1 \subset H^1$ is a subbundle such that $\phi:\tilde E^1 \to f^\ast (L')^1$ is an isomorphism,
\item a connection in $K$, giving rise to the direct sum connection   in $H^1$.
\end{items}
The curvature $s \in \Gamma(M,H^1)$ splits into a sum 
\begin{equation}\label{eq:DecompCurv}
s=u+f^\ast t,
\end{equation}
where $u\in \Gamma(M,K)$ is a section of $K$.

Consider the diagram of vector bundles over $M$:
$$\xymatrix{
T_{M/N}\dto_j \rto^{(\nabla u)\circ j}& K\dto\\
T_M\rto^{\nabla s}\dto\urto^{\nabla u} & H^1\dto_\phi\ar@/^/[u]^\theta\\
f^\ast T_N\rto^{f^\ast(\nabla t)} & f^\ast
(L')^1\rlap{\,.}}$$
Both squares with downward pointing arrows commute. Moreover,  we have  
 $\theta\circ(\nabla s)=\nabla u$. 

Let $P\in Z$. The fact $(f,\phi)$ is \'etale at $P$ means that   in the morphism of short
exact sequences of vector spaces
$$\xymatrix{
T_{M/N}|_P\rto\dto & K|_P\dto\\
T_M|_P\rto \dto& H^1|_P\dto\\
T_N|_{\phi(P)}\rto & (L')^1|_{\phi(P)}\rlap{\,,}}$$
the map $T_{M/N}|_P\to K|_P$ is an isomorphism.

Note that the composition $(\nabla u)\circ j$ is an isomorphism at
points $P\in Z$, because at these points, our second diagram is equal
to the first.  Then $(\nabla u)\circ j$  is still an isomorphism in a neighborhood of $Z$, so we
will assume that it is true everywhere: $(\nabla u)\circ j$ is an
isomorphism.  

We   note that $u\in \Gamma(M,K)$ is a regular section, by which we
mean that for every $P\in Z(u)$, the derivative $TM|_P\to K|_P$ is
surjective. This is true, because at points where $u$ vanishes, the
derivative $TM|_P\to K|_P$ coincides with $(\nabla u)|_P$, but $\nabla
u:TM\to K$ is an epimorphism because $(\nabla u)\circ j$ is an
isomorphism. 

Hence, $Y=Z(u)$ is a submanifold of $M$, and we have a short exact sequence
of vector bundles $T_Y\to T_M|_Y\to K|_Y$. Note that $Z\subset Y$.

Since $T_{M/N}|_Y$ is a complement for $T_Y$ in $T_M|_Y$, it follows
that the composition $T_Y\to T_M|_Y\to (f^\ast T_N)|_Y$ is an isomorphism, so
that the composition $Y\to M\to N$ is \'etale. 

By definition $u|_Y=0$, so after restricting to $Y$, the curvature $s|_Y$ is contained in the subbundle $E^1 := \tilde E^1 |_Y\subset H^1|_Y$. 
\end{pf}

\begin{lem}\label{lem:TopoLem}
Let $f:M\to N$ be  a local diffeomorphism, and $Z\subset N$ a closed subset, with preimage $X=f^{-1}(Z)\subset M$. Assume that the induced map $X\to Z$ is injective.  Then there exists an open neighborhood $V$ of $X$ in $M$, such that $f|_V:V\to N$ is a diffeomorphism onto an open neighborhood of $Z$ in $N$. 
\end{lem}
\begin{pf}
It suffices to prove that there exists an open neighborhood $V$ of $X$ in $M$, such that $f|_V:V\to N$ is injective. 

{\bf Case I\@. } $Z$ is compact.

Choose a sequence of relatively compact open neighborhoods $V_1\supset V_2\supset\ldots$ of $X$ in $M$, such that 
$$\bigcap_i V_i=X\,.$$
We claim that there exists an $i$, such that $f$ is injective when restricted to $V_i$. If not, choose in every $V_i$ a pair of points $(P_i,Q_i)$, such that $f(P_i)=f(Q_i)$. Upon replacing $(P_i,Q_i)$ by a subsequence, we may assume that $\lim_{i\to\infty} P_i=P$, and $\lim_{i\to\infty}Q_i=Q$. We have $P,Q\in X$, and by continuity, $f(P)=f(Q)$.  This implies $P=Q$. Let $V$ be a neighborhood of $P=Q$ in $M$, such that $f|_V$ is injective.  For sufficiently large $i$, both $P_i$ and $Q_i$ are in $V$, which is a contradiction.

{\bf Case II\@. } General case. 

Let $(U_i)_{i\in I}$ be a locally finite open cover of $N$, such that all $U_i$ are relatively compact. It follows that, for every $i\in I$, the set
$$I_i=\{j\in I\st  U_j\cap U_i\not=\varnothing\}$$
is finite.  Hence,
$$\tilde U_i=\bigcup_{j\in I_i}U_j$$
is still relatively compact, and by Case~I, there exists an open neighborhood $\tilde V_i$ of $X\cap f^{-1}\tilde U_i$ on which $f$ is injective. We may assume that $\tilde V_i\subset f^{-1} \tilde U_i$. Define 
$$V_i=f^{-1}U_i\cap\bigcap_{j\in I_i}\tilde V_j\,.$$
This is an open subset of $M$.  If $P\in X$, such that $f(P)\in U_i$, then $f(P)\in \tilde U_j$, for all $j\in I_i$.  Hence, $P\in f^{-1}\tilde U_j\subset \tilde V_j$, for all $j\in I_i$, and so $P\in  V_i$. Thus, $X\cap f^{-1}U_i\subset V_i$.

We define
$$V=\bigcup_{i\in I} V_i\,.$$
This is an open neighborhood of $X$ in $M$. We claim that $f$ is injective on $V$. So let $P,Q\in V$ be two points such that $f(P)=f(Q)$. Suppose $P\in V_i$ and $Q\in V_j$. Then $f(P)\in U_i$, and $f(Q)\in U_j$, so that $U_i$ and $U_j$ intersect. Hence $i\in I_j$, and so $V_j\subset \tilde V_i$. Since we have $V_i\subset\tilde V_i$, we have that both $P,Q\in \tilde V_i$. Since $f$ is injective on $\tilde V_i$, we conclude that $P=Q$. 
\end{pf}

\begin{pfIFT}
To prove Theorem~\ref{ift}, we factor the linear fibration $(f,\phi):\M\to\N$ into $(M,L,\lambda)\to (M,f^\ast L',f^\ast \lambda')\to (N,L',\lambda')$. Both morphisms in this factorization are linear fibrations, but they may not be \'etale at $Z$.

We apply Lemma \ref{firstcase}
 recursively to $(M,L,\lambda)\to (M,f^\ast L',f^\ast \lambda')$, starting with $k\gg0$. The smallest $k$ to which the lemma applies is $k=2$. We end up in the situation as in the assumption of Lemma~\ref{lastcase}. (The $L_\infty$[1]-algebra $H$ is constructed recursively by applying Lemma~\ref{firstcase}.)

Applying Lemma~\ref{firstcase} and Lemma~\ref{lastcase}, we achieve the first part of Theorem~\ref{ift}. 
To finish the proof, we now assume $\pi_0(\M) \to \pi_0(\N)$ is injective. Since $Y$ was chosen to be the zero locus of $u$ in Lemma~\ref{lastcase}, it follows from \eqref{eq:DecompCurv} that $f^{-1}\big(Z(t)\big) \cap Y \subset Z(s)$, and thus $f$ maps $f^{-1}\big(Z(t)\big) \cap Y$ injectively to $Z(t)$. Therefore, by Lemma~\ref{lem:TopoLem}, the space $Y$ can be chosen to be an open submanifold of $N$. This concludes the proof.
%
\end{pfIFT}

As a byproduct of the sequence of above lemmas, we obtain the follow proposition which will be useful for investigating quasi-isomorphisms in the next section.

\begin{prop}\label{recap}
Let $(f,\phi): (M,L,\lambda) \to (N,L',\lambda')$ be a linear trivial
 fibration of derived manifolds. Then there exists a strong embedding
 $(M,H,\mu) \to (M,L,\lambda)$ of derived manifolds, and a derived submanifold
 $(Y,E,\nu) \hookrightarrow (M,H,\mu)$ such that 
\begin{items}
\item
the composition 
$
(M,H,\mu) \to (M,L,\lambda) \to (N,L',\lambda')
$ 
is a linear fibration  such that $H^{\geq 2} \to f^*(L')^{\geq 2}$
is an isomorphism of graded vector bundles;
\item
the morphism $(Y,E,\nu) \to (N,L',\lambda') |_{f(Y)}$ is a linear isomorphism of derived manifolds.
\end{items}

More explicitly, by choosing a splitting of the short exact sequence $0 \to K \hookrightarrow H^1 \xrightarrow{\psi_1} f^\ast(L')^1 \to 0$ of vector bundles, one has a decomposition
\begin{align*}
H^1 & = K \oplus  f^\ast (L')^1, \\
\mu_0 \,  & = \, u \, + \, f^\ast\lambda_0' \, .
\end{align*}
Then $u$ is a regular section of $K = \ker(\psi_1)$ over $M$. One can choose $Y$ to be an open neighborhood of $\pi_0(\M)$ in $Z(u)$ and $E^1 = f^\ast (L')^1|_Y$ such that the restriction $f|_Y$ is a diffeomorphism from $Y$ to an open submanifold of $N$ which contains $\pi_0(\N)$.
\end{prop}
\begin{pf}
The strong embedding in Proposition~\ref{recap} is constructed by applying Lemma~\ref{firstcase} recursively. The derived manifold $(Y,E,\nu)$ is obtained by the construction in Lemma~\ref{lastcase} with $Z = \pi_0(\M)$, and the rest of required properties  follow from Lemma~\ref{lem:TopoLem} and the proof of Lemma~\ref{lastcase}.
\end{pf}

\subsection{Applications to quasi-isomorphisms}

Let us start with the case $(M,E,u)$ where $M$ is a manifold, and $u$ a regular section of a vector bundle $E$ over $M$. This means that for every point $P\in M$, such that $u(P)=0$, the derivative induces a surjective
linear map $D u|_P = D_P u :TM|_P\to E|_P$.  In other words, $u$ being a regular section 
is equivalent to that 
$u$ is transversal to the zero section of $E$. 
Let $Y\subset M$ be the zero locus of $u$, which is a submanifold of $M$. We have a canonical epimorphism of vector bundles $Du|_Y:TM|_Y\to E|_Y$, whose kernel is $TY$.  

We will place $E$ in degree $1$, so that $(M,E,u)$, as well as $(Y,0,0)$ are derived manifolds, and $(Y,0,0)\to (M,E,u)$ is a morphism of derived manifolds (in fact, a weak equivalence of derived manifolds).  We claim that it induces a quasi-isomorphism on function algebras. 

\begin{prop}\label{koszul}
The restriction map $\Gamma(M,\Sym E^\vee,\iota_u)\to\Gamma(Y,\O_Y)$ is a quasi-isomorphism.
\end{prop}
\begin{pf}
Let $Y\subset U\subset M$ be an open neighborhood of $Y$ in $M$. 

\paragraph{Claim.} The restriction map
$$\Gamma(M,\Sym E^\vee,\iota_u)\longrightarrow\Gamma(U,\Sym E^\vee|_U,\iota_u)$$ is a quasi-isomorphism.
To prove this, let $f:M\to\rr$ be a differentiable function such that $f|_Y=1$, and $f|_{M\setminus U}=0$. 
Multiplication by $f$ defines a homomorphism of differential graded  
$\Gamma(M,\Sym E^\vee,\iota_u)$-modules 
$$\Gamma(U,\Sym E^\vee|_U,\iota_u)\longrightarrow \Gamma(M,\Sym E^\vee,\iota_u)\,,$$
in the other direction, which is a section of the restriction map.  
We now claim that for every $i\leq0$, multiplication by $f$ induces the identity on $H^i(\Gamma(M,\Sym E^\vee,\iota_u))$.  This follows from the fact that $f$ restricts to the identity in $\Gamma(Y,\O_Y)$, and that every $H^i(\Gamma(M,\Sym E^\vee, \iota_u))$ is a $\Gamma(Y,\O_Y)$-module. 

By the claim, we can replace $M$ by any open neighborhood of $Y$. 
In fact, we will assume that $M$ is a tubular neighborhood of $Y$, with Euler vector field $v$ and projection $\rho:M\to Y$. We denote the relative tangent bundle by $T_{M/Y}$. The Euler vector field is a regular section of $T_{M/Y}$, with zero locus $Y$. 

\paragraph{Claim.} After restricting to a smaller neighborhood of $Y$ if necessary, there exists an isomorphism of vector bundles $\Psi:T_{M/Y}\to E$, such that $\Psi(v)=u$. 

To prove this claim, first assume that there exists a normal coordinate system $(x^i)_{i=0,\ldots,n}$ on $M$, compatible with the tubular neighborhood structure.  This means that the Euler vector field has the form 
$v=\sum_{i=1}^k x^i\del_i$, where $k$ is the rank of $E$, and the codimension of $Y$ in $M$. Also assume that $E$ is trivial, with basis $(e_i)_{i=1,\ldots,k}$. We have $u=\sum u^ie_i$, and define $\Psi:T_{M/Y}\to E$ in coordinates by the matrix
$$\Psi_i^j=\int_0^1(\del_iu^j)(tx_1,\ldots,tx_k,x_{k+1},\ldots,x_n)\,dt\,.$$
Then $\Psi(v)=u$, and $\Psi|_Y$ is the canonical isomorphism $Du|_Y$. 

For the general case, construct a family of $\Psi_\alpha$ locally, and define  a global endomorphism $\Psi:T_{M/Y}\to E$ using a partition of unity $(\psi_\alpha)$:
$$\Psi=\sum_{\alpha}\psi_\alpha \Psi_\alpha\,.$$
We have that $\Psi(v)=u$, and $\Psi|_Y=D u|_Y$.  Hence $\Psi$ is an isomorphism in an open neighborhood of $Y$.  Upon replacing $M$ by this neighborhood we may assume that $\Psi$ is an isomorphism globally.

We are now reduced to the case where $E=T_{M/Y}[-1]$, and $u$ is the Euler vector field. In this case, $\Sym \dual E = \Sym(T_{M/Y}^\vee[1]) \cong \Lambda T_{M/Y}^\vee$.  
We define a contraction operator  $\eta:\Gamma(M,\Lambda T_{M/Y}^\vee) \to \Gamma(M,\Lambda T_{M/Y}^\vee)$ by the formula 
$$\eta(\omega)= \int_0^1 \sigma^\ast d\omega\wedge \frac{dt}{t}\,,$$
where $\sigma:[0,1]\times M\to M$ is the multiplicative flow of $v$ in the tubular neighborhood $M$.  

\paragraph{Claim.} $[\eta,\iota_v]=\id - \rho^\ast\circ \iota^\ast$.  

Here, $\rho: M \to Y$ is the projection, $\iota:Y\to M$ is the inclusion morphism;  $\rho^\ast:\Gamma(Y,\O_Y) \to \Gamma(M,\Lambda T_{M/Y}^\vee)$ and  $\iota^\ast:\Gamma(M,\Lambda T_{M/Y}^\vee) \to \Gamma(Y,\O_Y)$ are the corresponding induced maps on function algebras.

The claim can be checked locally, so we may assume given normal coordinates  $(x^i)$, as above.  The  the multiplicative flow is given by 
$$\sigma(t,x^1,\ldots, x^n)=(tx^1,\ldots tx^i,x^{i+1},\ldots,x^n)\,,$$
and using this, the claim is straightforward to prove.

Thus, $\big(\Gamma(Y,\O_Y), 0 \big)$ is homotopy equivalent to $\big(\Gamma(M,\Lambda T_{M/Y}^\vee),\iota_v\big)$. This completes the proof.
\end{pf}

We need the following more general case. 

Suppose $(L,\lambda)$ is a bundle of curved $L_\infty[1]$-algebras over the manifold $M$.  Suppose that $L^1$ is split into a direct sum $L^1=E\oplus H$, in such a way that the curvature $\lambda_0\in \Gamma(M,L^1)$ splits as $\lambda_0=u+s$, where $u$ is a regular section of $E$. Let $Y=Z(u)$ be the vanishing locus of $u$, and consider the induced morphism of derived manifolds
$$(Y,H|_Y\oplus L^{\geq2}|_Y,\lambda)\longrightarrow (M,L,\lambda)\,.$$

\begin{cor}\label{former}
The restriction map $$\Gamma(M,\Sym L^\vee,Q_\lambda)\longrightarrow\Gamma\big(Y,\Sym(H^\vee|_Y\oplus L^{\geq2,\vee}|_Y),Q_\lambda\big)$$ is a quasi-isomorphism.
\end{cor}
\begin{pf}
Let us write $\tilde L=H\oplus L^{\geq2}$. 
By the decomposition $L = E \oplus \tilde L$ of graded vector bundles, we obtain an induced morphism of sheaves of differential graded algebras $\Sym E^\vee\to
\Sym L^\vee$ over $M$. Since $\Sym L^\vee$ is locally free over $\Sym E^\vee$, the morphism $\Sym E^\vee\to \Sym L^\vee$ is flat. Therefore the quasi-isomorphism $\Sym E^\vee\to \O_Y$ gives rise to another quasi-isomorphism
$$\Sym E^\vee\otimes_{\Sym E^\vee}\Sym L^\vee\longiso \O_Y\otimes_{\Sym E^\vee}\Sym L^\vee\,.$$
The left hand side is equal to $\Sym L^\vee$, and the right hand side is equal to $\O_Y\otimes_{\O_M}\Sym \tilde L^\vee$, proving that we have a quasi-isomorphism
$\Sym L^\vee\to \O_Y\otimes_{\O_M}\Sym \tilde L^\vee$.
\end{pf}

\begin{thm}\label{thm:WeakEqImplyQuasiIso}
Suppose that $(M,L,\lambda)\to (N,L',\lambda')$ is a weak equivalence of derived manifolds.  Then the induced morphism of differential graded algebras $\Gamma\big(N,\Sym (L')^\vee\big)\to\Gamma(M,\Sym L^\vee)$ is a quasi-isomorphism.
\end{thm}
\begin{proof}
By Proposition~\ref{prop:IsoToLinearMor} and Factorization Lemma \cite{MR341469}, we may assume that $(M,L,\lambda)\to (N,L',\lambda')$ is a linear trivial fibration of derived manifolds.  Then we produce a retraction of $(M,L,\lambda)\to(N,L',\lambda')$ in two steps as in Proposition~\ref{recap}.  It suffices to prove that both $(Y,E,\nu)\to(M,H,\mu)$ and $(M,H,\mu)\to (M,L,\lambda)$ induce quasi-isomorphisms on differential graded function algebras. The former is done in Corollary~\ref{former}.  
The latter follows from the following general fact.
\end{proof}

\begin{prop}
Any transferred inclusion  $(\id,\phi): (M,H,\mu) \to (M,L,\lambda)$ of
 derived manifolds induces a quasi-isomorphism on differential graded function algebras.
\end{prop}
\begin{proof}
 We assume that $H\to L$ is obtained by an application of the transfer theorem for bundles of
 curved $L_\infty[1]$-algebras as in Proposition~\ref{transfertheorem}. More precisely, let $(M,L,\delta)$  be a bundle of complexes endowed with the curved $L_\infty[1]$-structure $\lambda$, let $\eta$ be a contraction of $\delta$ and $(M,H,\delta)$ the bundle of complexes onto which $\eta$ contracts $(M,L,\delta)$.  Let $\mu$ be the family of transferred curved $L_\infty[1]$-structures on $(M,H,\delta)$ and $\phi:H\to L$ the transferred inclusion morphism.

Let $A=\Gamma(M,\Sym L^\vee)$   and $B=\Gamma(M,\Sym H^\vee)$.  Let us denote the derivations induced by the dual maps of $\delta$ and $\lambda_k$ on $A$ by $\delta$ and $q_k$, respectively. Similarly, denote the derivations induced by the dual maps of $\delta$ and $\mu_k$ on $B$ by $\delta$ and $r_k$, respectively.  

\paragraph{Claim.}
The morphism of function algebras
$$\phi^\sharp:(A,\delta+q)\longto (B,\delta+r)\,,$$
associated to the morphism of derived manifolds $\phi:(M,H,\delta+\mu)\to (M,L,\delta+\lambda)$ is a quasi-isomorphism.

To prove the claim, refine the grading on $A$ to a double grading
$$A^{k,\ell}=\Gamma(M,\Sym^{-\ell}L^\vee)^{k+\ell}\,.$$
It is contained in the region defined by $\ell\leq0$ and $k+\ell\leq0$. Moreover $A$ has only finitely many non-zero terms for each fixed value of $k+\ell$. This will imply that the spectral sequences we construct below will be bounded and hence convergent to the expected limit.

Note that $\delta$ and all $q_n$ are bigraded: the degree of $\delta$ is $(1,0)$, and the degree of $q_n$ is $(n,1-n)$. 

If we filter $A$ by $F_kA=A^{\bullet,\geq k}$, the differential $\delta +q$ preserves the filtration, and we obtain a bounded spectral sequence $E_n^{k,\ell}$, which is convergent to $H^{k+\ell}(A,\delta+q)$.  
$$E_n^{k,\ell}\Longrightarrow H^{k+\ell}(A,\delta+q)\,.$$

The same construction applies to $(B,\delta+r)$, and we get a bounded spectral sequence
$$\tilde E_n^{k,\ell}\Longrightarrow H^{k+\ell}(B,\delta+r)\,.$$
The morphism of differential graded algebras $\phi^\sharp$ induces a morphism of spectral sequences $E\to \tilde E$, because $\phi$  respects the filtrations introduces above. 

So to prove our claim, it will be enough to show that $\phi^\sharp$ induces a quasi-isomorphism $E_1\to \tilde E_1$. 

The differential on $E_1$ is induced by $\delta+q_1$ on the cohomology of $E_0$ with respect to $q_0$. Similarly, the differential on $\tilde E_1$ is induced by $\delta+r_1$ on the cohomology of $\tilde  E_0$ with respect to $r_0$. The homomorphism $E_1\to \tilde E_1$ is induced by $\phi_1^\sharp$. 

Recall the deformed projection $\tilde\pi_1:L\to H$.  It induces an algebra morphism $B\to A$, which we will denote by $\tilde\pi_1^\sharp$.

Also recall the deformed contraction $\tilde\eta:L\to L$ of degree $-1$. 
We extend its dual to a fiberwise derivation $\Sym L^\vee \to \Sym L^\vee$, which we also denote by $\tilde\eta$. We denote by $\tilde\eta': \Sym L^\vee \to \Sym L^\vee$ the endomorphism obtained by dividing the derivation $\tilde\eta$ by the weight (and setting $\tilde\eta'(1)=0$).

We will now construct a fiberwise homotopy operator $h:\Sym L^\vee \to \Sym L^\vee$, with the property that 
\begin{equation}\label{heq}
[\delta+q_1, h]=1-\phi_1^\sharp\tilde\pi_1^\sharp\,.
\end{equation}
To do this, recall that $\Sym L^\vee$ is a bundle of Hopf algebras. Let us denote the coproduct by $\Delta$, the product by $m$. Then we define $h$ fiberwisely by the formula
$$h= m \circ w\circ(\phi_1^\sharp\tilde\pi_1^\sharp\otimes\tilde\eta')\circ \Delta\,.$$
Here $w:\Sym L^\vee \otimes \Sym L^\vee \to \Sym L^\vee \otimes \Sym L^\vee$ is the operator that divides an element of bi-weight $(k,\ell)$ by the binomial coefficient $\frac{(k+\ell)!}{k!\ell!}$.

We conclude that $\phi$ induces a homotopy equivalence on $E_1$.  This concludes the proof.
\end{proof}

We note that a
transferred inclusion  $(\id,\phi): (M,H,\mu) \to (M,L,\lambda)$ of
derived manifolds  is a  weak equivalence according to 
Proposition~\ref{prop:TransInclWkEq}.

\appendix

\section{Some remarks on diffeology}\label{sec:Diffeology}

Although all our constructions have inputs and outputs in
finite-dimensional manifolds and finite-rank vector bundles over them,
some constructions have infinite-dimensional intermediate steps.  

This does not cause problems in any of the proofs of the paper.  Nevertheless, it is possible to endow these infinite-dimensional intermediate spaces with diffeologies, turning them into geometric objects. It helps to understand some of the constructions in Section~\ref{sec:DerPathSp} in a more conceptual way.

  We will collect all the constructions and
facts we need here.  They are all straightforward to check.  As reference, we refer to 
the textbook~\cite{IZ}. 

We denote by $\Hom$ the degree preserving linear maps, by $\Mor$ the morphisms of graded manifolds, and by $\Map$ the general smooth maps which may not preserve degrees.

\subsection{Diffeological spaces}

A {\em domain }is an open subset of a {\em cartesian space }$\rr^n$,
$n\geq0$. 

A {\em diffeological space} is a set $X$, together with a
distinguished family $\D$ of maps $U\to X$, where $U$ is a domain, called
the family of {\em plots }of $X$, or the {\em diffeology }on $X$, such that the
three {\em diffeology axioms }are 
satisfied:
\begin{items}
\item if $V\to U$ is a differentiable map between domains, and 
$U\to   X$ is a plot, then so is the composition $V\to X$, 
\item if a map $U\to X$ is locally (in $U$) a plot, then it is a plot,
\item every map from the one-point domain $\rr^0$ to $X$ is a plot.
\end{items}
A map $f:X\to Y$ is called {\em smooth}, if the composition
of $f$ with every plot of $X$ is a plot of $Y$. A {\em morphism }of
diffeological spaces is a smooth map. Isomorphisms of diffeological
spaces are called {\em diffeomorphisms}. We say $Y$ is a {\em diffeological space over $X$} (or $Y$ is a {\em $X$-space}) if we are given a smooth map $Y \to X$.

Products and fibered products of diffeological spaces are defined in
the obvious way.  If $Y$ is a diffeological space over $X$,  
and $U\to X$ is a morphism, we often write the fibered product $U\times_X Y$ as $Y|_U$. 

We denote by $\Map(X,Y)$ 
 the set of smooth maps from $X$ to $Y$. It is the diffeological space, whose underlying set of points also denoted
$\Map(X,Y)$, 
which is characterized by the property
$$\Map(U,\Map(X,Y))=\Map(U\times X,Y)\,.$$
 (The category of
diffeological spaces is {\em cartesian closed}.) Sometimes, we will
abbreviate $\Map(X,Y)$ by $Y^X$.  When we have fixed
an interval $I\subset\rr$, we denote $X^I$ also by $PX$, and call it
the {\em path space }of $X$.

If $Y,Z$ are two diffeological spaces over the diffeological space $X$, we denote by $\Mapu_X(Y,Z)$ the diffeological space of smooth maps over $X$, whose fiber over $x\in X$ is the set of maps $Y|_x\to Z|_x$. For every diffeological $X$-space $U$, we have
$$\Map_X(U,\Mapu_X(Y,Z))=\Map_X(U\times_XY,Z)\,.$$
We denote the diffeological space of sections of $\Mapu_X(Y,Z)$ over $X$ by $\Map_X(Y,Z)$.  For a diffeological space $V$, we have
$$\Map(V,\Map_X(Y,Z))=\Map_X(V\times Y,Z)\,.$$

A diffeology $\D$ on $X$ is {\em finer }than another diffeology $\D'$,
if every $\D$-plot is a $\D'$-plot. If $\D$ is finer than $\D'$, then
$(X,\D)\to (X,\D')$ is smooth. 

Any subset $Z$ of a diffeological space $X$ is a
diffeological space by endowing it with the coarsest diffeology making
the inclusion $Z\to X$ smooth (this is the {\em subspace
  diffeology}). If $Z$ is endowed with the subspace diffeology, then a
map $U\to Z$ is a plot of $Z$ if and only if the composition $U\to X$
is a plot of $X$. 

If $Z$ is a smooth retract of the diffeological space $X$, then $Z$ is
diffeomorphic to the image of $Z$ in $X$ endowed with the subspace
diffeology. 

Diffeological spaces are topological spaces: a subset $V\subset X$ is
open if and only if for every plot $U\to X$, the inverse image of $V$
in $U$ is open in $U$. 

Every manifold $M$ is a diffeological space by defining a map $U\to M$, from a domain $U$ to $M$, 
to be a plot if it is differentiable. A map between manifolds is
differentiable, if and only if it is a smooth map of the diffeological
spaces defined by the manifolds. The category of  manifolds and
differentiable maps is a full subcategory of the category of
diffeological spaces and smooth maps.

\subsection{Diffeological  vector spaces}

Let $X$ be a diffeological space. Following \cite{tangent}, we define
a {\em diffeological vector space over }$X$, or a {\em diffeological $X$-vector space},  to be a diffeological
space $E$, endowed with a smooth map $E\to X$, and vector space
structures on all fibers of $E\to X$, such that addition
$E\times_XE\to E$, scalar multiplication $\rr\times E\to E$ and the
zero section $X\to E$ are smooth. 

For $X=\rr^0$, we get the definition of diffeological vector space. 

Morphisms of diffeological vector spaces over $X$ are smooth maps,
which are fiberwise linear. 

Pullbacks of diffeological vector spaces are diffeological vector
spaces, and for a diffeological  $X$-vector 
space $E$, the set of global sections $\Gamma(X,E)$ is a diffeological
vector space, in fact a subspace of $\Map(X,E)$. 

A morphism $\phi:V\to W$ of diffeological vector spaces is an {\em epimorphism }if every plot of $W$ admits local lifts to $V$. Suppose $\phi:V\to W$ is an epimorphism of diffeological vector spaces over $X$. The kernel $K$ of $\phi$ is the fibered product $K=V\times_{X,0}W$. The formation of the kernel commutes with arbitrary base changes $Y\to X$. Moreover, for every section $s:Y\to W$, the preimage $V\times_{W,s}Y$ is a principal homogeneous $K_Y$-space.

If $M$ is a manifold, and $E$ a vector bundle over $M$, then $E$
defines a diffeological vector space over $M$.  The category of vector
bundles over $M$, with fiberwise linear differentiable maps is a full
subcategory of the category of diffeological vector spaces over $M$.

For diffeological $X$-vector spaces $E$ and $F$, we have the diffeological $X$-vector space $\Homu_X(E,F)$ of linear maps from $E$ to $F$. (It is a subspace of the space of morphisms $\Mapu_X(E,F)$ where $E$ and $F$ are considered as $X$-spaces without vector spaces structure.) The fiber of $\Homu_X(E,F)$ over $x\in X$ is the set of linear maps from $E|_x$ to $F|_x$. The set of global sections of $\Homu_X(E,F)$ is denoted by $\Hom_X(E,F)$.  We also have $X$-spaces of (symmetric) multilinear maps from $E$ to $F$, and their spaces of global sections.

Let $M$ be a manifold and $E$ a vector bundle over $M$. Let
$I\subset\rr$ be an open interval. Let 
$PM=\Map(I,M)$ and $PE=\Map(I,E)$ be the respective path spaces. 
Then $PE$ is a
diffeological vector space over $PM$.  
The fiber of $PE$ over the path $a\in PM$ is $\Gamma(I,a^\ast E)$. 

\begin{lem}\label{dersec}
If $E=TM$ is the tangent bundle of $M$, then  $a\mapsto a'$ is a
smooth section of $PTM\to PM$. 
\end{lem}
\begin{pf}
This amounts to saying that whenever $U\times I\to M$ is a differentiable
family of paths in $M$ ($U$ a domain), then the partial derivative $U\times
I\to TM$ with respect to the second component is again differentiable. 
\end{pf}

\subsection{Path spaces and connections}\label{sec:PathSpCntn}

Fix an open interval $I\subset\rr$. 
Let $\nabla$ be a connection in the vector bundle $E$ over the
manifold $M$.  In the following, $dt$ is a formal symbol of cohomological degree $1$.

\begin{lem}\label{covdev}
The covariant derivative is a smooth linear map $\delta:PE\to PE\,dt$ over $PM$.  
\end{lem}
\begin{pf}
If $a:I\to M$  is a path  in $M$, the map $\delta$ is given in the
fiber over $a$ by the
covariant derivative of the pullback connection 
$$\Gamma(I,a^\ast E)\longrightarrow \Gamma(I,a^\ast E\otimes\Omega^1_I)=
\Gamma(I,a^\ast E)\,dt\,.$$
To prove that that $\delta$ is smooth as a map from $PE$ to $PE\,dt$,
we need to prove that for every differentiable  family of paths
$\alpha:U\times I\to E$, the partial covariant  
derivative $dt \, \nabla_{\frac{\del}{\del t}}$ in the direction of $I$
gives rise to a differentiable map $(-1)^{|\alpha|} \nabla_{\frac{\del}{\del
    t}}\alpha\,dt:U\times I\to E\,dt$. (The sign $(-1)^{|\alpha|}$ appears if $E$ is a graded vector bundle.)  This is, of course, simply the
covariant derivative of the pullback connection $a^\ast\nabla:a^\ast
E\to a^\ast E\otimes \Omega^1_{U\times I}$, followed by the projection
$\Omega^1_{U\times I}\to \Omega^1_I$.  Here $a:U\times I\to M$ is the projection of $\alpha$ to $M$. 
\end{pf}

Often we will suppress $dt$ from the notation.

Let $P_\con E\subset PE$ be the kernel of $\delta:PE\to PE$. 
It is the 
diffeological subspace consisting of paths in $E$ which are horizontal
with respect to $\nabla$. The fiber of $P_\con E$ over $a\in PM$
consists of all covariant constant sections of $a^\ast E$. 

\begin{lem}\label{pulcon}
The map $P_\con E\to PM$ is a vector bundle of rank $\rk E$. For any
point $t\in I$, parallel transport gives rise to a pullback diagram of
diffeological spaces 
$$\xymatrix{P_\con E\rto^{\ev_t}\dto & E\dto\\
PM\rto^{\ev_t}& M\rlap{\,.}}$$
More universally, 
$$\xymatrix{
P_\con E\times I\rto^-\ev\dto & E\dto\\
PM\times I\rto^-\ev & M}$$
is a pullback diagram of diffeological spaces.
\end{lem}
\begin{pf}
Let $a:U\times I\to M$ be a family of paths in $M$, and 
$f:U\to I$ a differentiable function. Furthermore, let
$s:U\to E$ be a lift of the composition $a\circ (\id\times f):U\to M$.
We need to prove that the associated family of horizontal
lifts $\tilde s:U\times I\to E$ is differentiable,
$$\xymatrix{
U\rto^{s}\dto_{\id\times f} & E\dto\\
U\times I\rto^a\ar@{.>}[ur]^{\tilde s} & M}$$
which is a standard fact.  
\end{pf}

\begin{lem}
The contraction operator $\eta:PL\,dt\to PL$ defined in Section~\ref{sec:DrivedPathSp} is a morphism of diffeological vector spaces over $PM$.
\end{lem}
\begin{pf}
Fix a plot $a:U\to PM$, giving rise to the family of paths $\tilde a:U\times I\to M$. (For this proof it will be appropriate to distinguish between $a$ and $\tilde a$.)
A plot $\alpha:U\to PL\,dt$, lying over the plot $a$ of $PM$,  is the same thing as a section of ${\tilde a}^\ast L$ over $U\times I$. 
Consider the pullback diagram (see Lemma~\ref{pulcon})
$$\xymatrix{
& P_\con L\times I\rto\dto & L\dto\\
U\times I\ar@{.>}[ur]^\alpha\ar@/_/[rr]_{\tilde a}\rto^{a\times \id_I} & PM\times I\rto^{\ev} & M\rlap{\,.}}$$
It shows that a section of ${\tilde a}^\ast L$ over $U\times I$ is the same thing as a dotted arrow.  This, in turn, is the same thing as a lift in the diagram 
$$\xymatrix{
&P_\con L\dto\\
U\times I\ar@{.>}[ur]^\alpha\rto^{a\times \text{proj}}&PM{\,,}}$$
or a section of $a^\ast P_\con L$ over $U\times I$:
$$\xymatrix{
&a^\ast P_\con L\dto\\
U\times I\ar@{.>}[ur]^\alpha\rto^{\text{proj}}& U}$$
The
associated section $\eta(\alpha)$ is defined by 
\begin{equation}\label{defeta}
(x,t)\longmapsto (-1)^{|\alpha|}\, \Big( \int_0^t\alpha(x,u)\,du-t\int_0^1\alpha(x,u)\,du\,\Big) \, ,
\end{equation}
where the integrals are computed inside the fiber of $a^\ast P_\con L$
over $x$, which is $\Gamma_\con(I,a_x^\ast L)$. 

To prove that $\eta(\alpha)$ is differentiable is a local claim in
$U$, and so we may assume that the vector bundle $a^\ast P_\con L$
over $U$ is trivial. Then $\alpha$ is  a map $\alpha:U\times
I\to\rr^k$ and (\ref{defeta}) is visibly differentiable, combining
differentiation under the integral sign (for differentiability with
respect to $x$), and the fundamental theorem of calculus (for
differentiability with respect to $t$.)
\end{pf}

\begin{cor}
We define $P_\lin E\subset PE$ as the kernel of $\delta^2:PE\to PE$. It is a vector bundle of rank $2\rk E$ over $PM$.
\end{cor}

\section{The curved $L_\infty[1]$-transfer theorem}\label{linfty}

\subsection{Symmetric multilinear operations and curved $L_\infty[1]$-algebras}

Let $X$ be a diffeological space.  We will adapt the definition of curved $L_\infty[1]$-algebras to diffeological $X$-vector spaces. We refer the readers to \cite{MR2931635,MR1235010,MR1327129,MR1183483} for a general introduction to $L_\infty$-algebras.

Let $E$ and $F$ be graded  diffeological $X$-vector spaces,
concentrated in finitely many positive degrees. 
Following \cite{ezra}, we denote by $\Symu^{n,i}_{X}(E,F)$, for $n,i\geq0$,
the diffeological $X$-vector space 
whose fiber over $x\in X$ is the space of graded symmetric multilinear maps 
$(E|_x)^n\to F|_x$ of degree $i$. The space of global sections of $\Symu^{n,i}_X(E,F)$ is equal to $\Sym^{n,i}_X(E,F)$, the vector space of  smooth maps
$$\underbrace{E\times_X\ldots\times_XE}_{n}\longrightarrow F$$
which are fiberwise graded symmetric multilinear maps  of degree
$i$.
Note that $\Symu^{n,i}_{X}(E,F)$ vanishes, as soon as $n+i$ exceeds the
highest degree of $F$. We write
$$\Symu^i_{X}(E,F)=\bigoplus_{n\geq0}\Symu^{n,i}_{X}(E,F)\,.$$
There is a binary operation 
$$\Symu^i_{X}(E,F)\times_X \Symu^j_{X}(E,E)\longrightarrow \Symu_{X}^{i+j}(E,F)\,,$$
defined by 
$$(\lambda\circ\mu)_n(x_1,\ldots,x_n)=
\sum_{\sigma\in
  S_n}(-1)^\epsilon\sum_{k=0}^n\frac{1}{k!(n-k)!}\,
\lambda_{n+1-k}\big(\mu_k(x_{\sigma(1)}\ldots),
\ldots,x_{\sigma(n)}\big)\,,$$ 
for $x_1,\ldots,x_n$ elements of a (common) fiber of $E$. Here $(-1)^\epsilon$ is determined by Koszul sign convention. This operation is linear in each argument. It is not associative, but it does satisfy the pre-Lie
algebra axiom.  In particular, taking $F=E$, the graded commutator with respect to
$\circ$ defines the structure of a diffeological graded Lie algebra over $X$ on 
$$\Symu_{X}(E,E)=\bigoplus_{i\geq0}\Symu_{X}^i(E,E)\,.$$
In particular, the space of global sections $\Sym_X(E,E)$ is a diffeological graded Lie algebra. 

If $\delta$ is a smooth differential in $E$, then $[\delta,\argument]$
is an induced differential on $\Symu_{X}(E,E)$ (acting by derivations with respect to the bracket), turning $\Symu_{X}(E,E)$ into a
diffeological differential graded Lie algebra over $X$. The global sections $\Sym_X(E,E)$ form a diffeological differential graded Lie algebra. 

If we fix the  complex of diffeological $X$-vector spaces $(E,\delta)$, then the structure of a {\bf diffeological curved $L_\infty[1]$-algebra }on on $E$ is a Maurer-Cartan element in $\Sym_X(E,E)$, i.e. a global section $\lambda$ of  $\Symu^1_X(E,E)$, satisfying the Maurer-Cartan equation
$$[\delta,\lambda]+\lambda\circ\lambda=0\,.$$
The case where $\delta=0$ is no less general than the case of arbitrary $\delta$, because we can always incorporate $\delta$ into $\lambda$ by adding it to $\lambda_1$.  In the case where $\delta=0$, the Maurer-Cartan equation reduces to 
$$\lambda\circ\lambda=0\,.$$
Nevertheless, for the transfer theorem, the case of general $\delta$ will be important to us.

\begin{lem}\label{symult}
Suppose $L$ is a vector bundle over the manifold $M$, and
$\lambda\in\Sym_M^i(L,L)$ a family of smooth graded symmetric
operations in $L$, then we get an induced family of smooth graded
symmetric operations $P\lambda\in \Sym_{PM}^i(PL,PL)$. 
\end{lem}
\begin{pf}
Consider $PL$ as a diffeological vector space over $X = PM$. 
Let $a\in PM$ be a path in the base.  The fiber of $PL\to PM$ is
$\Gamma(I,a^\ast L)$. The operation $P\lambda$ is given in this fiber
as the operation induced on global sections by the pullback $a^\ast
\lambda$. Let us prove that $P\lambda$ is smooth.  For this choose a
smooth family of paths $a:U\times I\to M$, and a family of lifts
$\alpha_1,\ldots,\alpha_n:U\times I\to L$. Then
$\lambda_n(\alpha_1,\ldots,\alpha_n):U\times I\to L$ is smooth by the
definition of the diffeology on
$\underbrace{PL \times_X\ldots\times_X PL}_{n}$ and the fact that
compositions of smooth maps are smooth.
\end{pf}

There is another binary operation
$$\Symu_X^i(E,E)\times_X\Symu_X^0(F,E)\longrightarrow \Symu_X^i(F,E)\,,$$
defined by 
\begin{multline*}
(\lambda\bullet\phi)_n(x_1,\ldots,x_n)
=\sum_{\sigma\in S_n}(-1)^\epsilon\\
\sum_{k=0}^n\frac{1}{k!}\sum_{n_1+\ldots+n_k=n}\frac{1}{n_1!\ldots n_k!}\lambda_k\big(\phi_{n_1}(x_{\sigma(1)},\ldots),\dots,\phi_{n_k}(\ldots,x_{\sigma(n)})\big)\,.\end{multline*}
This operation is linear only in the first argument, but it is associative. 

If $\lambda$ and $\mu$ are curved $L_\infty[1]$-structures on $(E,\delta)$ and $(F,\delta')$, respectively, and $\phi\in \Sym^0(E,F)$, then $\phi$ is a {\bf morphism of diffeological curved $L_\infty[1]$-algebras }if 
$$\phi\circ(\delta+\lambda)=(\delta'+\mu)\bullet\phi\,.$$
The proof from \cite{ezra}, that any composition of morphisms of curved $L_\infty[1]$-algebras is another morphism of curved $L_\infty[1]$-algebras, carries over to this more general context.  

For example, 
if $\lambda$ is a curved $L_\infty[1]$-structure on $(E,\delta)$, and $F\subset E$ is a subcomplex preserved by all operations $\lambda$, then the inclusion $F\to E$, without any higher correction terms is a morphisms of curved $L_\infty[1]$-algebras. We say that $F$ is a curved $L_\infty[1]$-subalgebra of $E$. 

\subsection{The transfer theorem}\label{tranthem}

Let $M$ be a manifold, and $(L,\delta)$  a complex of diffeological vector spaces over $M$.  Assume that $L=L^1\oplus\ldots L^{n+1}$ is concentrated in finitely many positive degrees. Assume  also  given a descending filtration 
$$L=F_0L\supset F_1L\supset\ldots$$
on $L$, such that both the  grading and the differential $\delta$ are compatible with $F$. 
We assume that   $F_k L=0$, for $k\gg0$. 

Let $\eta$ be a smooth map of degree $-1$ on $L$, which is compatible with the filtration $F$, and satisfies the two conditions
$$\eta^2=0\,,\qquad \eta\delta\eta=\eta\,.$$
(We call $\eta$ a {\em contraction} of $\delta$.)

Under these hypotheses, $\delta\eta$ and $\eta\delta$ are idempotent operators on $L$.  We define $$H=\ker[\delta,\eta]=\ker(\delta\eta)\cap\ker(\eta\delta)=\im(\id_L-[\delta,\eta])\,.$$ Then $H$ is a graded diffeological vector subspace of $L$.  Let us write $\iota:H\to L$ for the inclusion, and $\pi:L\to H$ for the projection. We have 
$$\iota \pi=\id_L-[\delta,\eta]\,, \quad\text{and}\quad \pi\iota=\id_H\,.$$ 
We write $\delta$ also for the induced differential on $H$. Then $\iota:(H,\delta)\to (L,\delta)$ is a homotopy equivalence with homotopy inverse $\pi$. 

Let $\lambda=(\lambda_k)_{k\geq0}$ be a curved $L_\infty[1]$-structure in $(L,\delta)$.  We assume that this $L_\infty[1]$-structure is {\em nilpotent}, 
i.e., that   $\lambda$  increases the filtration degree by 1. This means that, for $n\geq1$, we have  $\lambda_n(F_{k_1},\ldots,F_{k_n})\subset F_{k_1+\ldots+k_n+1}$.  It puts no restriction on $\lambda_0$.  (This convention on $\lambda_0$ is different from \cite{MR1950958,ezra}.  In \cite{MR1950958,ezra}, it is assumed that the image of $\lambda_0$ is contained in $F_1L^1$.  We do not need this stronger requirement, because all our $L_\infty[1]$-algebras are in bounded degree.)

\begin{prop}[Transfer Theorem]\label{transfertheorem}
There is a unique  $\phi\in \Sym^0(H,L)$ satisfying the equation
\begin{equation}\label{recu}
\phi=\iota-\eta \lambda\bullet\phi\,.
\end{equation}
Setting
$$\mu=\pi\lambda\bullet\phi\in S^1(H,H)$$ defines a curved $L_\infty[1]$-structure on $(H,\delta)$, such that $\phi$ is a morphism of curved $L_\infty[1]$-structures from $(H,\delta+\mu)$ to $(L,\delta+\lambda)$. 

Furthermore, there exists a morphism of curved $L_\infty[1]$-algebras $\tilde\pi:(L, \delta + \lambda) \to (H, \delta + \nu)$ satisfying the equations $\tilde\pi_1 = \pi - \tilde\pi_1 \lambda_1 \eta$ and $\tilde\pi \bullet \phi = \id_H$.
\end{prop}
\begin{proof}
The proof in \cite{ezra} applies to our situation. 
Our assumptions imply that the number of internal nodes (i.e., nodes of valence at least 2)  is bounded above by the highest filtration degree in $L$, and the valence of all  nodes is bounded by the highest degree in $L$. This bounds the number of trees itself, and hence the number of transferred operations.

For the last part, since the curved $L_\infty[1]$-algebras are assumed to be bounded positively graded and nilpotent, one may modify the  approach in \cite{MR3276839} to obtain the transferred projection $\tilde \pi$. 
\end{proof}

\begin{figure}
\begin{subfigure}{.5\textwidth}
\centering
\begin{tikzpicture}[
    sibling distance  = 1.5cm,
    level distance  = 1cm,
    grow=up
  ]
\node {}
child {
node[circle, draw]{$\lambda_3$}
child {
node[circle, draw, black]{$\lambda_2$}
child {
node {}
edge from parent[black]
node[right]{ $\iota$}
}
child {
node {}
edge from parent[black]
node[left] {$\iota$}
}
edge from parent[blue]
node[right, near start] {\; $\eta$}
}
child {
node {}
edge from parent[black]
node[left] {$\iota$}
}
child {
node[circle, draw, black] {$\lambda_1$}
child {
node {}
edge from parent[black]
node[left] {$\iota$}
}
edge from parent[blue]
node[left, near start] {$\eta$ \;}
}
edge from parent[red]
node[right] {$\eta$}
}
;
\end{tikzpicture}
\caption{A tree for transferred inclusion}
  \label{fig:TreeForIncl}
\end{subfigure}
\begin{subfigure}{.5\textwidth}
\centering
\begin{tikzpicture}[
    sibling distance  = 1.5cm,
    level distance  = 1cm,
    grow=up,
  ]
\node {}
child {
edge from parent[red]
node[right] {$\iota$}
}
;
\end{tikzpicture}
\caption{The tree without nodes}
  \label{fig:TreeWithoutNodes}
\end{subfigure}
\caption{Decorated trees in the formulas of transferred inclusion}
\end{figure}
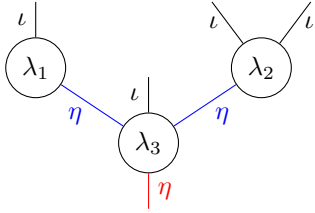
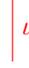

\begin{rmk}\label{rmk:TreeFormula} 
One can find $\phi:H\to L$ (and $\mu$) explicitly by solving (\ref{recu}) recursively.  The result is as follows.

Consider  rooted trees with a positive number $n$ of  leaves, and a non-negative number of nodes. For every such tree define an element of $S^{n,0}(H,L)$ by decorating each leaf (an edge without a node at the top) with $\iota$, each node with the applicable $\lambda_k$, and all edges other than leaves with $\eta$. See Figure~\eqref{fig:TreeForIncl} for an example.  (So the root is labeled with $\eta$, unless the root is a leaf, which happens only for the tree without any nodes as in Figure~\eqref{fig:TreeWithoutNodes}.) Then $\phi$ is equal to the sum over all such trees of the corresponding element of $S^0(H,L)$ with proper coefficients. Since $\lambda$ is nilpotent, only finitely many trees give non-zero contributions. 

Consider 
 also   rooted trees with a non-negative number $n$ of leaves, and a positive number of nodes. For every tree of this kind, define an element of $S^{n,1}(H,H)$ by decorating each leaf with $\iota$, each node with the applicable $\lambda_k$, all internal edges with $\eta$, and the root with $\pi$. See Figure~\eqref{fig:TreeForOp} for an example. Then $\mu$ is equal to the sum over all such trees of the corresponding element of $S^1(H,H)$ with proper coefficients. Again, only finitely many trees contribute to the sum.

A priori, one may have a tree with $\lambda_0$ labeled at a node such as Figure~\eqref{fig:VanishingTree}. Nevertheless, it follows from our degree assumptions that $\eta(\lambda_0)=0$, and thus the terms associated with such trees vanish except the transferred curvature $\pi(\lambda_0)$. 

In the homotopy transfer theorem for non-curved $L_\infty[1]$-algebras \cite{2017arXiv170502880B,FioMan},  the cochain differential $\delta$ is assumed to be the first bracket, and $\lambda$ consists of the operations with two or more inputs. Thus, the nodes in the tree formulas \cite{FioMan} have at least two ascending edges, while the nodes in our trees may have only one ascending edge as in Figure~\eqref{fig:TreeForOp}. 
\end{rmk}

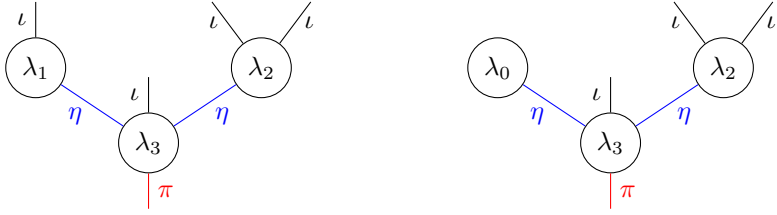
\begin{figure}
\begin{subfigure}{.5\textwidth}
\centering
\begin{tikzpicture}[
    sibling distance  = 1.5cm,
    level distance  = 1cm,
    grow=up
  ]
\node {}
child {
node[circle, draw]{$\lambda_3$}
child {
node[circle, draw, black]{$\lambda_2$}
child {
node {}
edge from parent[black]
node[right]{ $\iota$}
}
child {
node {}
edge from parent[black]
node[left] {$\iota$}
}
edge from parent[blue]
node[right, near start] {\; $\eta$}
}
child {
node {}
edge from parent[black]
node[left] {$\iota$}
}
child {
node[circle, draw, black] {$\lambda_1$}
child {
node {}
edge from parent[black]
node[left] {$\iota$}
}
edge from parent[blue]
node[left, near start] {$\eta$ \;}
}
edge from parent[red]
node[right] {$\pi$}
}
;
\end{tikzpicture}
\caption{A tree for transferred operations}
  \label{fig:TreeForOp}
\end{subfigure}
\begin{subfigure}{.5\textwidth}
\centering
\begin{tikzpicture}[
    sibling distance  = 1.5cm,
    level distance  = 1cm,
    grow=up,
  ]
\node {}
child {
node[circle, draw]{$\lambda_3$}
child {
node[circle, draw, black]{$\lambda_2$}
child {
node {}
edge from parent[black]
node[right]{ $\iota$}
}
child {
node {}
edge from parent[black]
node[left] {$\iota$}
}
edge from parent[blue]
node[right, near start]{\; $\eta$}
}
child {
node {}
edge from parent[black]
node[left] {$\iota$}
}
child {
node[circle, draw, black]{$\lambda_0$}
edge from parent[blue]
node[left, near start] {$\eta$ \;}
}
edge from parent[red]
node[right] {$\pi$}
}
;
\end{tikzpicture}
\caption{An ignorable tree}
  \label{fig:VanishingTree}
\end{subfigure}
\caption{Decorated trees in the formulas of transferred operations}
\end{figure}

\begin{rmk}
Consider the case $\lambda=\lambda_1$. Then   $\eta\lambda$ and $\lambda\eta$ are nilpotent linear operators on $L$.  The transfer theorem yields
$$\phi=\phi_1=(1+\eta\lambda)^{-1}\iota\,,$$
and
$$\mu=\mu_1=\pi\lambda(1+\eta\lambda)^{-1}\iota=\pi(1+\lambda\eta)^{-1}\lambda\iota\,.$$
In this case we have more: if we define 
$$\tilde\pi=\pi(1+\lambda\eta)^{-1}\quad\text{and}\quad
\tilde\eta=\eta(1+\lambda\eta)^{-1}=(1+\eta\lambda)^{-1}\eta\,,$$
we obtain a deformation of the `context' given by the pair $(\delta,\eta)$.  In fact, $\phi:(H,\delta+\mu)\to(L,\delta+\lambda)$ and $\tilde\pi:(L,\delta+\lambda)\to (H, \delta+\mu)$ are homomorphisms of complexes, $\tilde\pi\phi=\id_H$, and $\phi\tilde\pi=\id_L-[\delta+\lambda,\tilde\eta]$. This is a special case of the homological perturbation lemma \cite{MR0220273,2004math......3266C}.
\end{rmk}

\begin{rmk}\label{rmk:0th1stMapInHptTransfer}
In the general case, we always have
$$\phi_1=(1+\eta\lambda_1)^{-1}\iota\,,$$
$$\mu_0=\pi(\lambda_0)\, , \quad  \quad 
\mu_1=\pi\lambda_1(1+\eta\lambda_1)^{-1} \iota=\pi(1+\lambda_1\eta)^{-1}\lambda_1 \iota \,, $$
as well as
$$
\tilde\pi_1 = \pi (1+\lambda_1 \eta)^{-1} \, .
$$
\end{rmk}

\begin{rmk}\label{spec}
Let $P\in M$ be a point.  Since $\tilde\pi_1\circ\phi_1=\id$, we have that $\phi_1:H^1\to L^1$ is an isomorphism onto a subbundle of $L^1$, and thus $\mu_0(P)=0$ is equivalent to $\lambda_0(P)=0$.  In other words, the Maurer-Cartan loci of $(M, H,\delta+\mu)$ and $(M, L,\delta + \lambda)$ are equal.

Suppose $P$ is a Maurer-Cartan point.  Then $(H|_P,\delta+\mu_1)$ and $(L|_P,\delta+\lambda_1)$ are complexes of vector spaces, and $\phi_1$ is a morphism of complexes.  In fact, $$\phi_1: (H|_P,\delta+\mu_1)\to (L|_P,\delta+\lambda_1)$$ is a quasi-isomorphism. 
To prove it, considering the spectral sequence induced by the given filtration, one can reduce the proof to the fact that 
$$\iota:(H|_P,\delta)\longrightarrow (L|_P,\delta)$$ 
is a quasi-isomorphism. 
\end{rmk}

Remark~\ref{spec} leads to the following:

\begin{prop}\label{prop:TransInclWkEq}
Let $(M,L,\delta+\lambda)$ be a derived manifold, and $(M,H,\delta +\mu)$ the derived manifold obtained from applying the transfer theorem to $(M,L,\delta + \lambda)$. The transferred inclusion $(\id,\phi): (M,H,\delta +\mu) \to (M,L,\delta+\lambda)$ is a weak equivalence of derived manifolds.
\end{prop}
\begin{pf}
By Remark~\ref{spec}, it suffices to show that for any Maurer-Cartan point $P \in M$, the vertical maps 
$$
\xymatrix{
TM|_P \rto \dto_\id &  \dto^{\phi_1|_P} \tilde H^1|_P\\
TM|_P \rto & \tilde L^1|_P
}
$$
form a quasi-isomorphism, where $\tilde H^1 = \ker(\mu_1|_{H^1})$ and $\tilde L^1 = \ker(\lambda_1|_{L^1})$. By Remark~\ref{spec} again, the map $\phi_1|_P: \tilde H^1|_P \to \tilde L^1|_P$ is an isomorphism of vector spaces. Thus the vertical maps are an isomorphism of complexes and, in particular, a quasi-isomorphism. 
\end{pf}

\begin{ex}
The {\em natural filtration }of $L$ is given by 
$$F_k L=\bigoplus_{i\geq k}L^i\,.$$
Every $L_\infty[1]$-structure on $L$ is filtered with respect to this
natural filtration. 
\end{ex}

\begin{ex}\label{varfil}
In the proof of Lemma~\ref{firstcase}, we need the following variation of the natural
filtration $F$ at level $n$.  In fact, we define 
$$\tilde F_k=F_k\,,\qquad\text{for all $k\not=n,n+1$}\,,$$
and
$$\tilde F_n=\bigoplus_{j\geq n-1}L^{j}\,,$$
$$\tilde F_{n+1}=\bigoplus_{j\geq n+1} L^j\,.$$
The claim (which is easy to check) is that the only operation which is not of filtered degree 1
with respect to $\tilde F$ is $\lambda_1:L^n\to L^{n+1}$. 
\end{ex}

\section{Path space construction}
\label{sec:AKSZ}
In this section, we present constructions
of an   infinite dimensional dg manifolds and derived manifolds
 of path spaces.
The first subsection is to use AKSZ construction, while
the second  subsection is to apply the
derived manifold   construction in
Section \ref{sec:Ping} formally to vector bundles of path spaces. 

The discussion here aims to give a heuristic
reasoning as to where the curved $L_\infty[1]$-structure formula
in Proposition \ref{pro:paris} comes from.

We will not address the subtle issues such as in which
sense a path space is  an  infinite dimensional smooth
manifold. It can be understood as a smooth manifold in the sense of diffeology as in Section~\ref{sec:Diffeology} or in the sense of \cite{MR583436}.

\subsection{Mapping spaces: algebraic approach}

Let $L = \bigoplus_{i=1}^n L^i$ be a graded vector bundle over a manifold $M$, and $\M = (M, \A)$ the associated graded manifold of amplitude $[1,n]$. For any $P \in M$, $l \in L|_P$, one has the evaluation map 
$$
\ev_l: \Gamma(M,\A) \xrightarrow{\ev_P} \Sym L^\vee|_P \to \rr,
$$
where the second map is the evaluation of a polynomial on $L|_P$ at $l$.

\begin{lem}
The evaluation map defines a bijection from $L$ to the algebra morphism (which may not preserve degrees) from $\Gamma(M,\A)$ to $\rr$.  
\end{lem}
\begin{pf}
It is clear that $\ev_l$ is an algebra morphism. Conversely, let $\phi: \Gamma(M,\A) \to \rr$ be an algebra morphism. Then $\phi$ is uniquely determined by its restriction to the generators $\Gamma(M,\O_M) \oplus \Gamma(M,L^\vee)$. Decomposing the restriction according to degrees, we have
\begin{items}
\item
an algebra morphism $\phi_0: \Gamma(M,\O_M) \to \rr$, and
\item
$\phi_i: \Gamma(M,{L^i}^\vee) \to \rr$ such that 
$$\phi_i(ab) = \phi_0(a) \phi_i(b),$$ 
for any $a \in \Gamma(M,\O_M)$, $b \in \Gamma(M,{L^i}^\vee)$.  
\end{items}
It is well-known that such an algebra morphism $\phi_0$ is of the form $\ev_P$ for some $P\in M$, and such a map $\phi_i$ is equal to $\ev_l$ for some $l \in L^i|_P$.
\end{pf}
 
\begin{prop}\label{prop:PathTM}
There is a bijection between $TL[-1]$ and the set of algebra morphisms from $\Gamma(M,\A)$ to $\Gamma(\ast, \rr[1]) \cong \rr \oplus \rr[-1]$. 
\end{prop}
\begin{pf}
Let $\psi:\Gamma(M,\A) \to \rr \oplus \rr[-1]$ be any algebra morphism, and let $\psi^0$ and $\psi^1$ be the compositions of $\psi$ followed by the projections onto $\rr$ and $\rr[-1]$, respectively. Since $\psi$ is algebra morphism, so is $\psi^0$. Thus $\psi^0 = \ev_l$ for some $l\in L$. 

Since $\psi$ is an algebra morphism, from the algebra morphism property
$$
(\psi^0 +\psi^1)(ab) = (\psi^0 +\psi^1)(a)(\psi^0 +\psi^1)(b),
$$
it follows that 
\begin{align*}
\psi^1(a b)&  =  \psi^1(a) \psi^0(b) + \psi^0(a) \psi^1(b) \\
&=  \psi^1(a) \ev_l(b) + \ev_l(a) \psi^1(b)
\end{align*}
for any $a, \, b \in \Gamma(M,\A)$. 
Therefore $\psi^1$ can be identified with a tangent vector in $TL|_l$. This completes the proof. 
\end{pf}

Identifying mapping spaces with algebra morphisms of function algebras in the converse direction, we therefore have that
\begin{equation}
\label{eq:Map(R[1],M)}
\Map(\rr[1],\M) \cong T\M[-1].
\end{equation}

Note that \eqref{eq:Map(R[1],M)} is well-known. Here we  follow
the algebraic approach as in
\cite[Section 3.2.2]{helein2020introduction}.

\subsection{Motivation: AKSZ} 
Recall the following folklore concerning graded manifolds
 and dg manifolds following AKSZ \cite{MR1432574,MR2819233,2003math......7303K}. Here we follow \cite{MR2819233}
closely.  Let $\N$ and $\M$ be graded manifolds. Then
 $\Map (\N, \M)$ is a (usually infinite dimensional) graded manifold satisfying
\begin{align*}
\Mor(\Z\times \N, \M) \cong \Mor(\Z,\Map(\N, \M)).
\end{align*}

If, moreover, both $\N$ and  $\M$ are dg manifolds, then  homological
vector fields on $\N$ and $\M$ induce a homological vector field $Q_\mapnm$
so that $\big(\Map(\N, \M), \  Q_\mapnm\big)$ is a dg manifold.

Indeed, we have
\begin{align*}
Q_\mapnm:=\phi^L ( Q_\N ) + \phi^R (Q_\M),
\end{align*}
where $Q_\N^{L}$ and  $Q_\M^{R}$ denote the  homological vector fields
on   $\N$ and $\M$, respectively, and $\phi^L$ and  $\phi^R$
are the natural morphisms:
\begin{align*}
\xymatrix{\calx(\N)  \ar[rrr]^-{\phi^L} &&& \calx(\Map(\N, \M)) &&& \calx(\M) \ar[lll]_-{\phi^R}.}\end{align*}

Now let $\N=TI[1]$ be the natural dg manifold with
the homological vector field $Q_\N$ being the
de Rham differential, and $\M=(M,L,\lambda)$ be a derived manifold
 with homological vector field $Q$ on $\M$. Then it is known
\cite{MR2819233, 2003math......7303K} that as a graded manifold: (See also \eqref{eq:Map(R[1],M)} and Proposition~\ref{prop:PathTM}.)
\begin{eqnarray*}
\mapnm&=&\Map (TI[1] , \ \M)\\
&=&\Map (I\times \rr [1], \ \M)\\
&=& \Map \big(I, \Map ( \rr [1], \ \M) \big)\\
&\cong& \Map \big(I, T\M[-1]\big)\\
&=&PT\M[-1]\\
& \cong& T(P\M)[-1].
\end{eqnarray*}

Moreover, $\phi^L ( Q_\N )=\iota_\tau$, where $\tau \in \calx (P\M)$
is the tautological vector field on  $P\M$,
and $\phi^R (Q)=\widehat{PQ}[-1]$, the shifted complete lift  on $T(P\M)[-1]$ of
the homological vector field $PQ\in \calx (P\M)$.
As  a consequence, it follows that
 $\big(T(P\M)[-1], \ \iota_\tau+\widehat{PQ}[-1]\big)$
is an (infinite dimensional)  dg manifold. 

We summarize the
discussion above in the following

\begin{prop}
Let  $\tau \in \calx (P\M)$ be the tautological vector field on  $P\M$,
then $\big(T(P\M)[-1], \ \iota_\tau+\widehat{PQ}[-1]\big)$
is a  dg manifold.
\end{prop}

\subsection{Path space of derived manifolds}

Indeed,  the fact above can be derived from a direct argument.  The flow
of $\tau$ is simply just a linear re-parameterization of the path space,
which preserves the  path  derived manifold $P\M=(PM, PL, P\lambda)$.
It thus follows that 
\begin{equation}
[\tau, \ PQ]=0,
\end{equation}
where $PQ$ denotes the corresponding homological vector field on
$P\M$. See Remark  \ref{rmk:tauQ}.

 From the proof of Proposition
\ref{pro:TM3}, it follows that $\iota_\tau+\widehat{PQ}[-1]$
is indeed a homological vector field on $T(P\M)[-1]$.

It is simple to see that $\tau \in \calx (P\M)$ is indeed a linear
vector field on the underlying vector bundle $\pi: PL\to PM$,  and
$$\pi_* \tau =D \in \calx (PM),$$
where $D$ denotes the  tautological vector field on  $PM$.
Choose a  connection $\nabla$ on $L\to M$, which induces a connection on
$PL\to PM$. Now we can apply Proposition \ref{pro:TM3} formally
to the path  derived manifold $P\M=(PM, PL, P\lambda)$ to obtain
the path space of shifted tangent bundle.

We need a lemma first.  

Given a path $l: I\to L$  in $L$, let $a(t)=\pi (l(t))$, $t\in I$,
be its corresponding path in $M$. Then $\dot{l}(t)\in T_{l(t)}L$
and 
$$\pi_* (\dot{l}(t))=\dot{a} (t)\in T_{a(t)}M$$

By $\widehat{\dot{a} (t)}|_{\dot{l}(t)}\in T_{l(t)}L$, we denote
the horizontal lift of $\dot{a} (t)\in T_{a(t)}M$ at 
$l(t)$. 

The following lemma is standard. See~\cite[Theorem~12.32]{MR2572292} or~\cite[page~114]{MR0152974}.

\begin{lem}
$$\dot{l}(t) =\widehat{\dot{a} (t)}|_{\dot{l}(t)}+\nabla_{\dot{a} (t)}l(t)$$
as tangent vectors at  $T_{l(t)}L$
\end{lem}

Applying Proposition \ref{pro:TM3} formally to this situation,
 we obtain the following:

\begin{prop}
\label{pro:TM4}
Let $\M=(M,L,\lambda)$ be a derived manifold, and 
 $\nabla$  a connection on $L$. Then
 $(PM, PT_M \, dt\, \oplus PL\, dt\,\oplus PL, D+\delta+ P\mu)$ is a 
derived manifold,  where
$\mu=\lambda+\tilde\lambda+\nabla\lambda\,$ are as in
Proposition \ref{pro:TM1}, and $\delta: PL\to PL \, dt $ is the degree $1$ map
defined by  $\delta (l(t))=(-1)^{|l|} \nabla_{\dot{a} (t)}l(t) \, dt$.
\end{prop}

As a consequence,  for every path $a:I\to M$ in $M$, we get an induced
 curved $L_\infty [1]$-structure in the vector space
$\Gamma(I,a^\ast(T_M\oplus L)\,dt\oplus a^\ast L)$ by restricting
to the fiber $a\in PM$.

As $\delta^2=0$, it induces the structure of a complex on $\pathast$.
Also note  that $D|_a$ is the derivative $a'\, dt \in \Gamma(I,a^\ast T_M)\,dt$.
 It is simple to see that the resulting
 $L_\infty [1]$-operations on $\Gamma(I,a^\ast(T_M\oplus L)\,dt\oplus a^\ast L)$
  are exactly those in Proposition \ref{pro:paris}.

 We  summarize it in the following

\begin{cor}
\label{cor:paris}
Let $\delta$ be the covariant derivative of the pullback connection $a^\ast \nabla$ over $a \in PM$.
The sum $a'\, dt \, +a^\ast\mu$ is a curved $L_\infty[1]$-structure
on the complex $\big(\Gamma(I,a^\ast T\M[-1]), \delta \big) = \big( \pathast ,    \delta\big)$.
\end{cor}

\bibliographystyle{plain}

\bibliography{../../ref_CatFibObj}

\Addresses

\end{document}